\documentclass[12pt]{amsart}

\usepackage{amsfonts,amssymb,stmaryrd,amscd,amsmath,latexsym,amsbsy, fullpage}
\usepackage[all]{xy}
\usepackage{amssymb}
\usepackage{amsfonts}
\usepackage{latexsym}
\usepackage[mathscr]{euscript}

\numberwithin{equation}{section} 
\newtheorem{theorem}{Theorem}[section]
\newtheorem{lemma}[theorem]{Lemma}
\newtheorem{proposition}[theorem]{Proposition}
\newtheorem{corollary}[theorem]{Corollary}
\theoremstyle{definition}
\newtheorem{definition}[theorem]{Definition}

\newtheorem{example}[theorem]{Example}

\newtheorem{conjecture}[theorem]{Conjecture}
\newtheorem{remark}[theorem]{Remark}

\renewcommand{\Re}{{\mathrm{Re\,}}}
\renewcommand{\Im}{{\mathrm{Im\,}}}
\newcommand{\be}{{\mathbf e}}
\newcommand{\e}{\mathtt{e}}
\newcommand{\Spec}{\mathrm{Spec}}
\newcommand{\Ext}{\mathrm{Ext}}

\newcommand{\Ker}{\mathrm{Ker\,}}
\newcommand{\rank}{\mathrm{rank\,}}

\newcommand{\End}{\mathrm{End}}
\newcommand{\PGL}{\mathrm{PGL}}
\newcommand{\Supp}{\mathrm{Supp}}
\newcommand{\Par}{\mathrm{Par}}
\newcommand{\GL}{\mathrm{GL}}
\newcommand{\tO}{\mathrm{O}}
\newcommand{\Ob}{\mathscr{O}}
\newcommand{\HC}{\mathrm{HC}}
\renewcommand{\H}{\mathrm{H}}
\newcommand{\SL}{\mathrm{SL}}
\newcommand{\Hom}{\mathrm{Hom}}

\newcommand{\Id}{\mathrm{Id}}
\newcommand{\Aut}{\mathrm{Aut}}

\newcommand{\Rees}{\mathrm{Rees}}
\newcommand{\Ind}{\mathrm{Ind}}
\newcommand{\Res}{\mathrm{Res}}
\newcommand{\Loc}{\mathrm{Loc}}
\newcommand{\reg}{\mathrm{reg}}
\newcommand{\res}{\mathrm{res}}
\newcommand{\ind}{\mathrm{ind}}
\newcommand{\Rep}{\mathrm{Rep}}
\newcommand{\supp}{\mathrm{supp}\,}

\newcommand{\diag}{\mathrm{diag}}

\newcommand{\KZ}{\mathrm{KZ}}
\newcommand{\K}{\mathrm{K}}
\newcommand{\op}{{\mathrm{op}}}
\newcommand{\tr}{\mathrm{Tr}\,}
\newcommand{\ch}{\mathrm{ch}\,}
\newcommand{\bh}{{\bf h}}
\newcommand{\bi}{{\mathbf i}}
\newcommand{\Hilb}{\mathrm{Hilb}}

\newcommand{\sH}{\mathsf{H}}
\newcommand{\sB}{\mathsf{B}}
\newcommand{\sZ}{\mathsf{Z}}
\newcommand{\sM}{\mathsf{M}}
\newcommand{\sS}{\mathsf{S}}
\newcommand{\sU}{\mathsf{U}}

\newcommand{\g}{\mathfrak{g}}
\newcommand{\h}{\mathfrak{h}}
\newcommand{\n}{\mathfrak{n}}
\newcommand{\kS}{\mathfrak{S}}

\newcommand{\kZ}{\mathfrak{Z}}

\renewcommand{\d}{\mathrm{d}}

\newcommand{\gl}{\mathfrak{gl}}
\newcommand{\gsl}{\mathfrak{sl}}
\newcommand{\su}{\mathfrak{su}}
\newcommand{\bW}{\overline{W}}
\newcommand{\aW}{\overline{W}_{a}}
\newcommand{\eW}{
\overline{W}_{a}^{\mathrm{ext}}}
\def\HH{\hbox{${\mathcal H}$\kern-5.2pt${\mathcal H}$}}
\newcommand{\rS}{\mathscr{S}}
\newcommand{\gr}{\mathrm{gr}}

\newcommand{\mC}{{\mathcal C}}

\newcommand{\CC}{{\mathbb{C}}}
\newcommand{\ZZ}{{\mathbb{Z}}}
\newcommand{\RR}{{\mathbb{R}}}
\newcommand{\QQ}{{\mathbb{Q}}}
\newcommand{\cM}{{\mathcal M}}
\newcommand{\cW}{{\mathcal W}}
\newcommand{\cS}{{\mathcal S}}
\newcommand{\cO}{\mathcal{O}}
\newcommand{\cD}{{\mathcal D}}

\newcommand{\bc}{{\bar c}}
\newcommand{\hoc}{{\mathrm{HH}}}

\theoremstyle{definition}
\theoremstyle{remark}

\begin{document}

\title{Lecture notes on Cherednik algebras}

\author[P.~Etingof]{Pavel Etingof}
\address{{\bf P.E.}: Department of Mathematics, Massachusetts Institute of Technology,
Cambridge, MA 02139, USA}
\email{etingof@math.mit.edu}

\author[X.~Ma]{Xiaoguang Ma}
\address{{\bf X.M.}: Department of Mathematics, Massachusetts Institute of Technology,
Cambridge, MA 02139, USA}
\email{xma@math.mit.edu}

\maketitle
\tableofcontents

\newpage \section{Introduction}
Double affine Hecke algebras, also called Cherednik algebras, 
were introduced by Cherednik in 1993 as a tool in 
his proof of Macdonald's conjectures about orthogonal polynomials 
for root systems. Since then, it has been realized that 
Cherednik algebras are of great independent interest; 
they appeared in many different mathematical contexts
and found numerous applications.  

The present notes are based on a course on Cherednik algebras given by the first author
at MIT in the Fall of 2009. Their goal is to give an introduction 
to Cherednik algebras, and to review the web of connections between them 
and other mathematical objects. For this reason, 
the notes consist of many parts that are relatively 
independent of each other. Also, to keep the notes within the
bounds of a one-semester course, we had to limit the discussion of many 
important topics to a very brief outline, or to skip them altogether.
For a more in-depth discussion of Cherednik algebras, we refer the reader to 
research articles dedicated to this subject. 

The notes do not contain any original material. 
In each section, the sources of the exposition are listed in the 
notes at the end of the section. 

The organization of the notes is as follows. 

In Section 2, we define the classical and quantum Calogero-Moser systems, and their analogs for any 
Coxeter groups introduced by Olshanetsky and Perelomov. 
Then we introduce Dunkl operators, prove the fundamental result of their commutativity, and 
use them to establish integrability of the Calogero-Moser and Olshanetsky-Perelomov systems. 
We also prove the uniqueness of the first integrals for these systems. 

In Section 3, we conceptualize the commutation relations between Dunkl operators 
and coordinate operators by introducing the main object of these notes - the rational Cherednik algebra. 
We develop the basic theory of rational Cherednik algebras (proving the PBW theorem), and then pass to the 
representation theory of rational Cherednik algebras, more precisely, study the structure of category
$\mathcal{O}$. After developing the basic theory (parallel to the case of semisimple Lie algebras),
we completely work out the representations in the rank 1 case, and prove a number of results 
about finite dimensional representations and about representations of the rational Cherednik 
algebra attached to the symmetric group.

In Section 4, we evaluate the Macdonald-Mehta integral, and then use it to find the 
supports of irrieducible modules over the rational Cherednik algebras with the trivial lowest weight, 
in particular giving a simple proof of the theorem of Varagnolo and Vasserot, classifying 
such representations which are finite dimensional. 

In Section 5, we describe the theory of parabolic induction and restriction functors 
for rational Cherednik algebras, developed in \cite{BE}, and give some applications of this theory, 
such as the description of the category of Whittaker modules and of possible supports of modules lying in category 
$\mathcal{O}$. 

In Section 6, we define Hecke algebras of complex reflection groups, 
and the Knizhnik-Zamolodchikov (KZ) functor from the category $\mathcal{O}$
of a rational Cherednik algebra to the category of finite dimensional representations of the 
corresponding Hecke algebra. We use this functor to prove the formal flatness of Hecke algebras of 
complex reflection groups (a theorem of Brou\'e, Malle, and Rouquier), and state the 
theorem of Ginzburg-Guay-Opdam-Rouquier that the KZ functor is an equivalence from the category $\mathcal{O}$ 
modulo its torsion part to the category of representations of the Hecke algebra. 

In Section 7, we define rational Cherednik algebras for orbifolds. We also 
define the corresponding Hecke algebras, and define the KZ functor 
from the category of modules over the former to that over the latter. 
This generalizes to the ``curved'' case the KZ functor for rational Cherednik algebras 
of complex reflection groups, defined in Section 6. We then apply the KZ functor to showing that 
if the universal cover of the orbifold in question has a trivial $H^2$ (with complex coefficients), 
then the orbifold Hecke algebra is formally flat, and explain why the condition of 
trivial $H^2$ cannot be dropped. Next, we list examples of orbifold Hecke algebras 
which satisfy the condition of vanishing $H^2$ (and hence are formally flat). 
These include usual, affine, and double affine Hecke algebras, as well as Hecke algebras attached to Fuchsian groups, 
which include quantizations of del Pezzo surfaces and their Hilbert schemes; we work these examples out in some detail, 
highlighting connections with other subjects. Finally, we discuss the issue of algebraic flatness, and prove it in the case of algebras 
of rank 1 attached to Fuchsian groups, using the theory of deformations of group algebras of Coxeter groups developed in \cite{ER}. 
 
In Section 8, we define symplectic reflection algebras (which inlude rational Cherednik algebras as a special case), 
and generalize to them some of the theory of Section 3. Namely, we use the theory of deformations of Koszul algebras
to prove the PBW theorem for symplectic reflection algebras. We also determine the center of symplectic reflection algebras, 
showing that it is trivial when the parameter $t$ is nonzero, and is isomorphic to the shperical subalgebra if $t=0$. 
Next, we give a deformation-theoretic interpretation of symplectic reflection algebras as universal deformations of 
Weyl algebras smashed with finite groups. Finally, we discuss finite dimensional representations of symplectic reflection algebras for $t=0$, 
showing that the Azumaya locus on the space of such representations coincides with the smooth locus. 
This uses the theory of Cohen-Macaulay modules and of homological dimension in commutative algebra. 
In particular, we show that for 
Cherednik algebras of type $A_{n-1}$, the whole representation space is smooth and coincides with the spectrum of the center. 

In Section 9, we give another description of the spectrum of the center of the rational Cherednik algebra of type $A_{n-1}$ 
(for $t=0$), as a certain space of conjugacy classes of pairs of matrices, introduced by Kazhdan, Kostant, and Sternberg, 
and called the Calogero-Moser space (this space is obtained by classical hamiltonian reduction, and is a special case of a quiver variety). 
This yields a new construction of the Calogero-Moser integrable system. 
We also sketch a proof of the Gan-Ginzburg theorem claiming that the quotient of the commuting 
scheme by conjugation is reduced, and hence isomorphic to $\Bbb C^{2n}/\kS_n$. 
Finally, we explain that the Calogero-Moser space is a topologically trivial deformation of the 
Hilbert scheme of the plane, we use the theory of Cherednik algebras to compute the cohomology ring of this space. 
 
In Section 10, we generalize the results of Section 9 to the quantum case. 
Namely, we prove the quantum analog of the Gan-Ginzburg theorem (the Harish Chandra-Levasseur-Stafford 
theorem), and explain how to quantize the Calogero-Moser space using quantum Hamiltonian reduction. 
Not surprisingly, this gives the same quantization as was constructed in the previous sections, namely, 
the spherical subalgebra of the rational Cherednik algebra. 

 {\bf Acknowledgements.} 
We are grateful to the participants of the course on Cherednik 
algebras at MIT, especially to Roman Travkin and Aleksander Tsymbalyuk, 
for many useful comments and corrections. 
This work was  partially supported by the NSF grants
DMS-0504847 and DMS-0854764.

\newpage \section{Classical and quantum 
Olshanetsky-Perelomov systems for finite Coxeter groups}\label{Olper}

\subsection{The rational quantum Calogero-Moser system}

Consider the differential operator
$$
H=\sum_{i=1}^{n}\frac{\partial^{2}}{\partial x_{i}^{2}}
-c(c+1)\sum_{i\neq j}\frac{1}{(x_{i}-x_{j})^{2}}.
$$
This is the quantum Hamiltonian for a system of $n$ 
particles on the line of unit mass and the interaction 
potential (between particle 
$1$ and $2$) $c(c+1)/(x_{1}-x_{2})^{2}$. 
This system is called {\em the rational quantum Calogero-Moser system}.

It turns out that the rational 
quantum Calogero-Moser system is completely integrable. 
Namely, we have the following theorem.
\begin{theorem}{\label{thm-qintsys}}
There exist differential operators $L_{j}$ with rational coefficients of the form
$$
L_{j}=\sum_{i=1}^{n}(\frac{\partial}{\partial
x_{i}})^{j}+\textit{ lower order 
terms},
\quad j=1,\ldots, n,
$$
which are invariant under the symmetric group $\kS_n$, 
homogeneous of degree $-j$, and such that 
$L_{2}=H$ and $[L_{j}, L_{k}]=0, \forall j,k=1, \ldots, n$.
\end{theorem}
We will prove this theorem later.

\begin{remark}
$L_{1}=\sum_{i}\dfrac{\partial}{\partial x_{i}}$.
\end{remark}

\subsection{Complex reflection groups}

Theorem \ref{thm-qintsys} can be generalized to the case of any finite Coxeter group. 
To describe this generalization, let us recall the basic theory of finite Coxeter groups and,
more generally, complex reflection groups.

Let $\h$ be a finite-dimensional complex 
vector space. We say that a semisimple 
element $s\in \GL(\h)$ is a {\em (complex) reflection} if 
$\rank(1-s)=1$. This means that $s$ is conjugate to the diagonal matrix 
$\diag(\lambda, 1, \ldots, 1)$ where $\lambda\neq 1$.

Now assume $\h$ carries a nondegenerate inner product $( \cdot ,\cdot )$. We say that a semisimple element 
$s\in \tO(\h)$ is a {\em real reflection} if $\rank(1-s)=1$; equivalently, $s$ is conjugate to 
$\diag(-1, 1, \ldots, 1)$.

Now let $G\subset \GL(\h)$ be a finite subgroup.

\begin{definition} 
\begin{enumerate}
\item[(i)] We say that $G$ is a {\em complex reflection group} if it is generated by complex reflections.
\item[(ii)] If $\h$ carries an inner product, then a finite subgroup 
$G\subset \tO(\h)$ 
is a {\em real reflection group} (or a {\em finite Coxeter group}) 
if $G$ is generated by real reflections.
\end{enumerate}
\end{definition}

For the complex reflection groups, we have the following
important theorem. 

\begin{theorem}[The Chevalley-Shepard-Todd theorem, \cite{Che}]
A finite subgroup $G$ of $\GL(\h)$ is a complex reflection group if and only if the
algebra $(S\h)^{G}$ is a polynomial (i.e., free)
algebra.
\end{theorem}

By the Chevalley-Shepard-Todd theorem, the algebra $(S\h)^{G}$ has algebraically
independent generators $P_{i}$, homogeneous 
of some degrees $d_{i}$ for $i=1, \ldots, \dim\h$. The numbers $d_{i}$ are uniquely determined, 
and are called {\em the degrees} of $G$.

\begin{example}
If $G=\kS_{n}$, $\h=\CC^{n-1}$ (the space of vectors in $\CC^n$ with zero sum of coordinates), 
then one can take $P_{i}(p_{1},\ldots,
p_{n})=p_{1}^{i+1}+\cdots+p_{n}^{i+1}$, $i=1,\ldots,n-1$
(where $\sum_i p_i=0$).
\end{example}

\subsection{Parabolic subgroups}{\label{sec:paragp}}

Let $G\subset \GL(\h)$ be a finite subgroup.

\begin{definition} A parabolic 
subgroup of $G$ is the stabilizer $G_a$ of a point $a\in \h$. 
\end{definition}

Note that by Chevalley's theorem, a parabolic subgroup of a complex (respectively, real) reflection group is itself a 
complex (respectively, real) reflection group. 

Also, if $W$ is a real reflection group, then it can be shown that 
a subgroup $W'\subset W$ is parabolic if and only if it is conjugate to a subgroup generated by a subset of simple reflections of $W$. 
In this case, the rank of $W'$, i.e. the number of generating simple reflections,
equals the codimension of the space
$\h^{W'}$.

\begin{example}
Consider the Coxeter group of type $E_{8}$.
It has the Dynkin diagram:
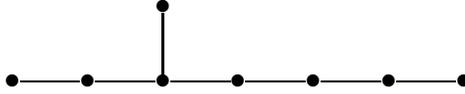
\begin{equation*}
\xygraph{ !{<0cm,0cm>;<1cm,0cm>:<0cm,1cm>::} 
!{(0,0) }*{\bullet}="a1" 
!{(1,0) }*{\bullet}="a2"
!{(2,0) }*{\bullet}="a3" 
!{(3,0)}*{\bullet}="a4" 
!{(4,0) }*{\bullet}="a5" 
!{(5,0) }*{\bullet}="a6"
!{(6,0) }*{\bullet}="a7" 
!{(2,1) }*{\bullet}="a8"
"a1"-"a2" "a2"-"a3" "a3"-"a4" "a4"-"a5" 
"a5"-"a6" "a6"-"a7" "a3"-"a8"}
\end{equation*}

The parabolic subgroups will be  
Coxeter groups whose Dynkin diagrams are obtained by
deleting vertices from the above graph.
In particular, the maximal parabolic subgroups are 
$D_7, A_7, A_1\times A_6, A_2\times A_1\times A_4, A_4\times A_3, D_5\times A_2, E_6\times A_1, E_7$. 
\end{example}

Suppose $G'\subset G$ is a parabolic subgroup, and $b \in\h$ is such that 
$G_{b} = G'$. In this case, we have a natural $G'$-invariant decomposition
$\h=\h^{G'}\oplus ({\h^{*}}^{G'})^{\perp}$, and $b \in \h^{G'}$. 
Thus we have a nonempty open set $\h_{\reg}^{G'}$ of all $a \in \h^{G'}$
for which
$G_{a} = G'$; this set is nonempty because it contains $b$. 
We also have a $G'$-invariant decomposition 
$\h^{*}={\h^{*}}^{G'} \oplus (\h^{G'} )^{\perp}$, 
and we can define the open set $\h_{\reg}^{*G'}$ 
of all $\lambda\in\h^{G'}$
for which $G_{\lambda} = G'$. 
It is clear that this set is nonempty. 
This implies, in particular, that one can make an alternative definition of 
a parabolic subgroup of $G$ as the stabilizer of a point in $\h^*$.

\subsection{Olshanetsky-Perelomov operators}
\label{sec-OP operators} 

Let $s\subset \GL(\h)$ be a complex reflection.
Denote by $\alpha_{s}\in \h^{*}$ an eigenvector in $\h^{*}$ of $s$ with nontrivial eigenvalue.

Let $W\subset \tO(\h)$ be a real reflection group and $\cS\subset W$ the set of reflections. Clearly, $W$
acts on $\cS$ by conjugation. 
Let $c: \cS\to \CC$ be a conjugation invariant function.

\begin{definition}\cite{OP}
The quantum Olshanetsky-Perelomov Hamiltonian attached 
to  $W$ is the second order differential operator
$$
H:=\Delta_\h-\sum_{s\in \cS}
\frac{c_s(c_s+1)(\alpha_s,\alpha_s)}{\alpha_s^2},
$$
where $\Delta_\h$ is the Laplace operator on $\h$. 
\end{definition}

Here we use the inner product on $\h^*$ which is dual to the inner product on $\h$. 

Let us assume that $\h$ is an irreducible representation of $W$ 
(i.e. $W$ is an irreducible finite Coxeter group, and $\h$ is its reflection representation.)
In this case, we can take $P_{1}(\mathbf p)={\mathbf p}^2$.

\begin{theorem}{\label{thm-OPsys}}
The system defined by the  Olshanetsky-Perelomov operator $H$ is completely integrable.
Namely, there exist differential operators $L_{j}$ on $\h$ with rational coefficients and
symbols $P_{j}$, such that $L_{j}$ are homogeneous (of degree $-d_{j}$), $L_{1}=H$, and 
$[L_{j}, L_{k}]=0$, $\forall j,k$.
\end{theorem}

This theorem is obviously a generalization of Theorem \ref{thm-qintsys} about $W=\kS_{n}$.

To prove Theorem \ref{thm-OPsys}, one needs to develop the theory of Dunkl operators.

\begin{remark}
1. We will show later that the operators $L_{j}$ are unique.

2. Theorem \ref{thm-OPsys} for classical root systems 
was proved by Olshanetsky and Perelomov (see \cite{OP}),
following earlier work of Calogero, Sutherland, and Moser 
in type A. For a general Weyl group, this theorem 
(in fact, its stronger trigonometric version) was 
proved by analytic methods in the series of papers 
\cite{HO},\cite{He3},\cite{Op3},\cite{Op4}. 
A few years later, a simple algebraic proof using Dunkl
operators, which works for any finite Coxeter group, 
was found by Heckman, \cite{He1}; this is the proof we will give
below. 

For the trigonometric version, Heckman also gave an
algebraic proof in \cite{He2}, which used non-commuting 
trigonometric counterparts of Dunkl operators. 
This proof was later improved by Cherednik (\cite{Ch1}), who defined commuting
(although not Weyl group invariant) versions of Heckman's
trigonometric Dunkl operators, now called Dunkl-Cherednik
operators.  
\end{remark}

\subsection{Dunkl operators}{\label{sec:dop}}

Let $G\subset \GL(\h)$ be a finite subgroup. Let $\cS$ be the set of reflections in $G$.
For any reflection $s\in \cS$, let $\lambda_{s}$ be the eigenvalue of $s$ on $\alpha_s\in \h^*$ 
(i.e. $s\alpha_{s}=\lambda_{s}\alpha_{s}$), and let $\alpha_{s}^{\vee}\in \h$ be an 
eigenvector such that $s\alpha_{s}^{\vee}=\lambda_{s}^{-1}\alpha_{s}^\vee$.
We normalize them in such a way that $\langle\alpha_s,\alpha_s^\vee\rangle=2$.

Let $c: \cS\to \CC $ be a function invariant with respect to
conjugation. Let $a\in \h$.

The following definition was made by Dunkl for real reflection
groups, and by Dunkl and Opdam for complex reflection groups. 

\begin{definition} 
The Dunkl operator $D_a=D_a(c)$ on $\CC(\h)$ is defined by the formula
$$
D_a=D_a(c):=\partial_a -\sum_{s\in \cS}\frac{2c_s\alpha_s(a)}{(1-\lambda_{s})\alpha_s}(1-s).
$$
\end{definition}

Clearly, $D_a\in \CC G\ltimes {\mathcal D}(\h_{\reg})$, where
$\h_{\reg}$ is the set of regular points of $\h$ (i.e. not preserved by any reflection),
and $\cD(\h_{\reg})$ denotes the algebra of differential operators on 
$\h_{\reg}$.

\begin{example}
Let $G=\ZZ_2$, $\h=\CC$. Then there is only one Dunkl operator up
to scaling, and it equals to
$$
D=\partial_x -\frac{c}{x}(1-s),
$$
where the operator $s$ is given by the formula $(sf)(x)=f(-x)$.
\end{example}

\begin{remark} 
The Dunkl operators $D_a$ map the space of polynomials
$\CC[\h]$ to itself.
\end{remark}

\begin{proposition}\label{prop-maincom}

\begin{enumerate}
\item[(i)] For any $x\in \h^*$, one has
$$
[D_a,x]=(a,x)-\sum_{s\in \cS}c_s(a,\alpha_s)(x,\alpha_s^\vee)s.
$$
\item[(ii)] If $g\in G$ then $gD_ag^{-1}=D_{ga}$.
\end{enumerate}
\end{proposition}

\begin{proof} 
(i) The proof follows immediately from the identity
$$x-sx=\frac{1-\lambda_{s}}{2}(x,\alpha_s^\vee)\alpha_s.$$

(ii) The identity is obvious from the invariance of the function $c$. 
\end{proof}

The main result about Dunkl operators, on which all their applications
are based, is the following theorem.

\begin{theorem}[C. Dunkl, \cite{Du1}]\label{thm-dunkl}  
The Dunkl operators commute:
$$[D_a,D_b]=0 \text{ for any } a,b\in \h.$$
\end{theorem}

\begin{proof} 
Let $x\in \h^*$. We have
$$
[[D_a,D_b],x]=[[D_a,x],D_b]-[[D_b,x],D_a].
$$
Now, using Proposition \ref{prop-maincom}, we obtain:
\begin{eqnarray*}
[[D_a,x],D_b]&=&-[\sum_{s\in \cS} c_s (a,\alpha_s)(x,\alpha_s^\vee)s,D_b]\\
&=&
-\sum_{s\in \cS}c_s (a,\alpha_s)(x,\alpha_s^\vee)(b,\alpha_s)s
D_{\alpha_s^\vee}\cdot \frac{1-\lambda_{s}^{-1}}{2}.
\end{eqnarray*}
Since $a$ and $b$ occur symmetrically, we obtain that $[[D_a,D_b],x]=0$.
This means that for any $f\in \CC[\h]$, $[D_a,D_b]f=f[D_a,D_b]1=0$.
So for $f,g\in \CC[\h]$, $g\cdot [D_{a}, D_{b}]\dfrac{f}{g}=[D_{a}, D_{b}]f=0$.
Thus $[D_{a}, D_{b}]\dfrac{f}{g}=0$ which implies $[D_{a}, D_{b}]=0$ in the algebra 
$\CC G\ltimes {\mathcal D}(\h_{\reg})$ (since this algebra acts faithfully
on $\CC(\h)$). 
\end{proof}

\subsection{Proof of Theorem \ref{thm-OPsys}} 

For any 
element $B\in \CC  W\ltimes {\mathcal D}(\h_{\reg})$, define
$m(B)$ to be the differential operator 
$\CC (\h)^{W}\to \CC (\h),$ defined by $B$.
That is, if $B=\sum_{g\in W}B_{g}g$, $B_{g}\in {\mathcal D}(\h_{\reg})$,
then $m(B)=\sum_{g\in W}B_{g}$.
It is clear that if $B$ is $W$-invariant, then 
$\forall A\in \CC W\ltimes {\mathcal D}(\h_{\reg})$, 
$$
m(AB)=m(A)m(B).
$$

\begin{proposition}[\cite{Du1}, \cite{He1}]\label{Hec}
Let $\{y_{1}, \ldots, y_{r}\}$ be an orthonormal basis of $\h$.
Then we have $$m(\sum_{i=1}^{r}D_{y_{i}}^{2})=\overline{H},$$ 
where
$\overline{H}=\Delta_{\h}-\sum_{s\in \cS}\dfrac{c_{s}(\alpha_{s}, \alpha_{s})}{\alpha_{s}}\partial_{\alpha_{s}^{\vee}}.$
\end{proposition}

\begin{proof} For any $y\in \h$,
we have $m(D_y^2)=m(D_y\partial_y)$. A simple computation
shows that
\begin{eqnarray*}
D_y\partial_y&=&
\partial_y^2-\sum_{s\in \cS}\frac{c_s\alpha_s(y)}{\alpha_s}(1-s)\partial_y\\
&=&
\partial_y^2-
\sum_{s\in \cS}\frac{c_s\alpha_s(y)}{\alpha_s}(\partial_y(1-s)+\alpha_s(y)
\partial_{\alpha_s^\vee}s).
\end{eqnarray*}
This means that
$$
m(D_y^2)=\partial_y^2-\sum_{s\in \cS}\frac{c_s\alpha_s(y)^2}{\alpha_s}\partial_{\alpha_s^\vee}.
$$
So we get
$$
m(\sum_{i=1}^rD_{y_i}^2)
=\sum_{i=1}^r\partial_{y_i}^2-
\sum_{s\in \cS}c_s\sum_{i=1}^r\frac{\alpha_s(y_{i})^2}{\alpha_s}\partial_{\alpha_s^\vee}
=\overline{H},
$$ 
since $\sum_{i=1}^r\alpha_s(y_i)^2= (\alpha_s,\alpha_s)$.
\end{proof} 
Recall that by the Chevalley-Shepard-Todd theorem,
the algebra $(S\h)^W$ is free. Let
$P_1={\mathbf p}^2,P_2,\ldots,P_r$ be homogeneous generators of 
$(S\h)^W$.

\begin{corollary}
The differential operators $\overline{L}_{j}=m(P_{j}(D_{y_{1}}, \ldots, D_{y_{r}}))$ are
pairwise commutative, have symbols $P_{j}$, homogeneity degree $-d_{j}$, 
and $\overline{L}_{1}=\overline{H}$.
\end{corollary}

\begin{proof}
Since Dunkl operators commute, the operators $L_j$ are well
defined. Since $m(AB)=m(A)m(B)$ when $B$ is invariant, 
the operators $L_j$ are pairwise commutative. The rest is clear. 
\end{proof}

Now to prove Theorem \ref{thm-OPsys}, we will show that the operators $H$ and $\overline{H}$
are conjugate to each other by a certain function; this will complete the proof.

\begin{proposition}{\label{prop-delta}} 
Let $\delta_c({\mathbf x}):=\prod_{s\in \cS}\alpha_s({\mathbf x})^{c_s}$.
Then we have
$$
\delta_c^{-1}\circ \overline{H}\circ \delta_c=H.
$$
\end{proposition}

\begin{remark}
The function $\delta_c({\mathbf x})$ is not rational. It is a multivalued analytic function. 
Nevertheless, it is easy to see that for any differential operator $L$ with rational coefficients, 
$\delta_c^{-1}\circ L\circ \delta_c$ also has rational coefficients.
\end{remark}

\begin{proof}[Proof of Proposition \ref{prop-delta}]
We have
$$
\sum_{i=1}^{r} \partial_{y_i}(\log\delta_c)\partial_{y_i}=
\sum_{s\in \cS}\frac{c_s(\alpha_s,\alpha_s)}{2\alpha_s}\partial_{\alpha_s^\vee}.
$$
Therefore, we have
$$
\delta_c \circ H\circ
\delta_c^{-1}=\Delta_\h-
\sum_{s\in \cS}\frac{c_s(\alpha_s,\alpha_s)}{\alpha_s}\partial_{\alpha_s^\vee}
+U,
$$
where
$$
U=\delta_c(\Delta_\h\delta_c^{-1})-
\sum_{s\in \cS}\frac{c_s(c_s+1)(\alpha_s,\alpha_s)}{\alpha_s^2}.
$$
Let us compute $U$. We have
$$
\delta_c(\Delta_\h \delta_c^{-1})=
\sum_{s\in \cS}\frac{c_s(c_s+1)(\alpha_s,\alpha_s)}{\alpha_s^2}+
\sum_{s\ne u\in \cS}\frac{c_sc_u(\alpha_s,\alpha_u)}{\alpha_s\alpha_u}.
$$
We claim that the last sum $\Sigma$ is actually zero. Indeed, this sum is
invariant under the Coxeter group, so $\prod_{s\in \cS}\alpha_s\cdot \Sigma$
is a regular anti-invariant function of degree $|\cS|-2$. But the smallest
degree of a nonzero anti-invariant is $|\cS|$, so
$\Sigma=0$, $U=0$, and we are done 
(Proposition \ref{prop-delta} and Theorem \ref{thm-OPsys} are proved).
\end{proof}

\begin{remark}
A similar method works for any complex reflection group $G$. Namely, the operators 
$L_{i}=m(P_{i}(D_{y_{1}}, \ldots, D_{y_{r}}))$ form a quantum integrable system.
However, if $G$ is not a real reflection group, this system does not have a quadratic 
Hamiltonian in momentum variables (so it does not have a physical meaning).  
\end{remark}

\subsection{Uniqueness of the operators $L_{j}$}

\begin{proposition} The operators $L_{j}$ are unique. 
\end{proposition}

\begin{proof}
Assume that we have two choices for 
$L_{j}$: $L_{j}$ and $L'_{j}$. Denote $L_{j}-L'_{j}$ by $M$.

Assume $M\neq 0$. We have
\begin{enumerate}
\item[(i)] $M$ is a differential operator on $\h$ with rational
coefficients, of order smaller than $d_j$ and homogeneity degree $-d_j$;
\label{cond-1}
\item[(ii)] $[M, H]=0$.\label{cond-2}
\end{enumerate}

Let $M_{0}$ be the symbol of $M$. Then $M_{0}$ is a polynomial of
${\mathbf p}\in \h^{*}$ with coefficients in 
$\CC(\h)$. We have, from (ii), 
$$
\{M_{0}, {\mathbf p}^{2}\}=0, \quad\forall {\mathbf p}\in \h^{*},
$$
and from (i) we see that the coefficients of $M_{0}$ are not polynomial (as they have 
negative degree).

However, we have the following lemma.

\begin{lemma}{\label{lem:Raj}}
Let $\h$ be a finite dimensional vector space. Let 
$\psi: ({\mathbf x},{\mathbf p})\mapsto \psi({\mathbf x},{\mathbf p})$ be a rational 
function on $\h\oplus\h^{*}$ which is a polynomial in ${\mathbf p}\in \h^{*}$.
Let $f: \h^{*}\to \CC$ be a polynomial such that the
differentials $\d f({\mathbf p})$ for ${\mathbf p}\in \h^*$ span $\h$ 
(e.g., $f({\mathbf p})={\mathbf p}^{2}$). Suppose that the Poisson
bracket of $f$ and $\psi$ vanishes: $\{\psi,f\}=0$. 
Then $\psi$ is a polynomial.
\end{lemma}

\begin{proof}(R. Raj)
Let $\kZ\subset \h$ be the pole divisor of $\psi$. 
Let ${\mathbf x}_{0}\in \h$ be a generic point in $\kZ$. 
Then $\psi^{-1}$ is regular 
and vanishes at $({\mathbf x}_{0}, {\mathbf p})$ for generic 
${\mathbf p}\in \h^*$. Also from $\{\psi^{-1}, f\}=0$, we have
$\psi^{-1}$ vanishes along the entire flowline of the Hamiltonian flow defined by $f$ and starting at
${\mathbf x}_{0}$. This flowline is defined by the formula 
$$
{\mathbf x}(t)={\mathbf x}_{0}+t\d f({\mathbf p}), \quad {\mathbf
p}(t)={\mathbf p},
$$
and it must be contained in the pole divisor of $\psi$ near
${\mathbf x}_{0}$. This implies that $\d f({\mathbf p})$
must be in $T_{{\mathbf x}_{0}}\kZ$ for almost every, hence for every
${\mathbf p}\in \h^*$. This is a contradiction 
with the assumption on $f$, which implies that in fact $\psi$ has no poles.
\end{proof}

\end{proof}

\subsection{Classical Dunkl operators and Olshanetsky-Perelomov Hamiltonians}
We continue to use the notations in Section \ref{sec-OP operators}.

\begin{definition}
The classical Olshanetsky-Perelomov Hamiltonian corresponding to $W$ is the following classical 
Hamiltonian on $\h_{\reg}\times\h^{*}=T^{*}\h_{\reg}$:
$$
H_{0}({\mathbf x},{\mathbf p})={\mathbf p}^{2}-
\sum_{s\in \cS}\frac{c_{s}^2(\alpha_{s},
\alpha_{s})}{\alpha_{s}^{2}({\mathbf x})}.
$$
\end{definition} 

\begin{theorem}[\cite{OP},\cite{HO,He3,Op3,Op4},\cite{He1}] 
\label{thm-classop}The Hamiltonian $H_{0}$ defines a classical integrable system. 
Namely, there exist unique regular functions $L_{j}^{0}$ on $\h_{\reg}\times \h^{*}$,
where highest terms in ${\mathbf p}$ are $P_{j}$, such that $L_{j}^{0}$ are homogeneous of
degree $-d_j$ (under ${\mathbf x}\mapsto \lambda {\mathbf x}, {\mathbf
x}\in \h^{*}, {\mathbf p}\mapsto \lambda^{-1}{\mathbf p}, {\mathbf p}\in \h$), 
and such that
$L_{1}^{0}=H_{0}$ and $\{L_{j}^{0}, L_{k}^{0}\}=0, \forall j,k$.
\end{theorem}
\begin{proof}
The proof is given in the next subsection.
\end{proof}

\begin{example}
Let $W=\kS_{n}$, $\h=\CC^{n-1}$. Then 
$$
H_{0}=\sum_{i=1}^np_{i}^{2}-c^{2}\sum_{i\neq j}\frac{1}{(x_{i}-x_{j})^{2}}
\quad(\text{ the classical Calogero-Moser Hamiltonian}).
$$
So the theorem says that there are functions $L_{j}^{0}, j=1, \ldots, n-1$, 
$$
L_{j}^{0}=\sum_{i}p_{i}^{j+1}+\text{ lower terms, }
$$
homogeneous of degree zero, such that $L_{1}^{0}=H_{0}$ and $\{L_{j}^{0}, L_{k}^{0}\}=0$.
\end{example}

\subsection{Rees algebras}
Let $\overline{A}$ be a filtered
algebra over a field $k$: $k=F^0\overline{A}\subset F^1\overline{A}\subset\cdots$,
$\cup_i F^i\overline{A}=\overline{A}$. 
Then the {\it Rees algebra} $A=\Rees(\overline{A})$ is defined by
the formula $A=\oplus_{n=0}^\infty F^n\overline{A}$. 
This is an algebra over $k[\hbar]$, where $\hbar$ is the element 1 
of the summand $F^1\overline{A}$. 

\subsection{Proof of Theorem \ref{thm-classop}}

The proof of Theorem \ref{thm-classop} is similar to the proof of
its quantum analog. 
Namely, to construct the functions $L_j^0$, we need to introduce 
classical Dunkl operators. To do so, we introduce a parameter $\hbar$ (Planck's constant)
and define Dunkl operators $D_a(\hbar)=D_a(c,\hbar)$ with $\hbar$: 
$$
D_{a}(c, \hbar)=\hbar D_{a}(c/\hbar)=\hbar\partial_{a}
-\sum_{s\in \cS}\frac{2c_{s}\alpha_{s}(a)}{(1-\lambda_{s})\alpha_{s}}(1-s), \text{ where } a\in \h.
$$
These operators can be regarded as elements of the Rees algebra 
$A=\Rees(\CC W\ltimes \cD(\h_{\reg}))$, where 
the filtration is by order of differential operators (and $W$ sits in 
degree $0$). Reducing these operators modulo $\hbar$, we get 
classical Dunkl operators $D_a^0(c)\in A_0:=A/\hbar A=\CC
W\ltimes \cO(T^*\h_{\reg})$. They are given by the formula 
$$
D_a^0(c)=p_a-\sum_{s\in \cS}\frac{2c_s\alpha_s(a)}{(1-\lambda_s)\alpha_s}(1-s),
$$
where $p_a$ is the classical momentum (the linear function on 
$\h^*$ corresponding to $a\in \h$). 

It follows from the commutativity of the quantum Dunkl operators $D_a(c)$ that 
the Dunkl operators $D_a(c,\hbar)$ commute. Hence, so do the classical Dunkl operators $D_a^0$: 
$$
[D_a^0,D_b^0]=0.
$$ 

We also have the following analog of Proposition \ref{prop-maincom}: 

\begin{proposition}\label{maincom1}
\begin{enumerate}
\item[(i)] For any $x\in \h^*$, one has
$$
[D_a^0,x]=-\sum_{s\in \cS}c_s(a,\alpha_s)(x,\alpha_s^\vee)s.
$$
\item[(ii)] If $g\in W$ then $gD_a^0g^{-1}=D_{ga}^0$.
\end{enumerate}
\end{proposition}

Now let us construct the classical Olshanetsky-Perelomov Hamiltonians. 
As in the quantum case, we have the operation $m(\cdot)$, which 
is given by the formula $\sum_{g\in W} B_g\cdot g\mapsto \sum B_g$, 
$B\in \cO(T^*\h_{\reg})$. We define 
the Hamiltonian
$$
\overline{H}_{0}:=m(\sum_{i=1}^r(D_{y_i}^0)^2).
$$
By taking the limit of quantum situation, we find 
$$
\overline{H}_{0}=
{\mathbf p}^2-\sum_{s\in\cS}
\frac{c_s(\alpha_s,\alpha_s)}{\alpha_s(\mathbf x)}p_{\alpha_s^\vee}.
$$

Unfortunately, this is no longer conjugate to $H_{0}$.
However, consider the (outer) automorphism $\theta_c$ of the algebra $\CC W\ltimes 
\cO(T^*\h_{\reg})$ defined by the formulas 
$$
\theta_c(x)=x,\ \theta_c(s)=s,\ \theta_c(p_a)=p_a+\partial_a \log\delta_c,
$$
for $x\in \h^*$, $a\in \h$, $s\in W$. 
It is easy to see that if $b_0\in A_0$
and $b\in A$ is a deformation of $b_0$ then 
$\theta_c(b_0)=\lim_{\hbar\to 0}\delta_{c/\hbar}^{-1}b\delta_{c/\hbar}$. 
Therefore, taking the limit $\hbar\to 0$ in Proposition
\ref{Hec}, we find that 
$H_{0}=\theta_{c}(\overline{H}_{0})$. 

Now set $L_j^{0}=m(\theta_{c}(P_{j}(D_{y_{1}}^{0}, \ldots,
D_{y_{r}}^{0})))$.
These functions are well defined since $D_a^0$
commute, are homogeneous of degree zero, and $L_1^0=H_0$.  

Moreover, we can define the operators 
$L_j(\hbar)$ in $\Rees(\cD(\h_{\reg})^W)$
in the same way as $L_j$, but using the Dunkl
operators $D_{y_i}(\hbar)$ instead of $D_{y_i}$. 
Then $[L_j(\hbar),L_k(\hbar)]=0$, and $L_j(\hbar)|_{\hbar=0}=L_j^0$. 
This implies that $L_j^0$ Poisson commute:
$\lbrace{L_j^0,L_k^0\rbrace}=0$.  

Theorem \ref{thm-classop} is proved.

\begin{remark}
As in the quantum situation, Theorem \ref{thm-classop} can be
generalized to complex reflection groups, giving integrable
systems with Hamiltonians which are non-quadratic in momentum variables. 
\end{remark}

\subsection{Notes} 
Section 2.1 follows Section 5.4 of \cite{E4}; 
the definition of complex reflection groups and their basic properties 
can be found in \cite{GM}; the definition of parabolic subgroups 
and the notations are borrowed from Section 3.1 of  \cite{BE};
the remaining parts of this section follow Section 6 of \cite{E4}.

\newpage \section{The rational Cherednik algebra}\label{rcera}

\subsection{Definition and examples} {\label{sec:rca}}
Above we have made essential use 
of the commutation relations between operators 
$x\in \h^{*}, g\in G$, and $D_{a}, a\in \h$. 
This makes it natural to consider the algebra generated by these operators.

\begin{definition}
The rational Cherednik algebra associated to $(G, \h)$ is the algebra 
$H_{c}(G, \h)$ generated 
inside $A=\Rees(\CC G\ltimes \cD(\h_{\reg}))$ by the elements 
$x\in \h^{*}, g\in G$, and $D_{a}(c, \hbar), a\in \h$. 
If $t\in \CC$, then the algebra 
$H_{t, c}(G, \h)$ is the specialization of $H_{c}(G, \h)$ at $\hbar=t$.
\end{definition}

\begin{proposition}{\label{prop-RCA}}
The algebra $H_{c}$ is the quotient of the algebra 
$\CC G\ltimes \mathbf T(\h\oplus \h^*)[\hbar]$ 
(where $\mathbf T$ denotes the tensor algebra) 
by the ideal generated by the relations
$$
[x,x']=0,\ [y,y']=0,\ [y,x]=\hbar(y,x)-\sum_{s\in \cS}
c_s(y,\alpha_s)(x,\alpha_s^\vee)s,
$$
where $x,x'\in \h^*$, $y,y'\in \h$.
\end{proposition}

\begin{proof}
Let us denote the algebra defined in the proposition 
by $H'_{c}=H'_{c}(G, \h)$. 
Then according to the results of the previous sections, we have a 
surjective homomorphism 
$\phi: H_{c}'\to H_{c}$ defined by the formula 
$\phi(x)=x$, $\phi(g)=g$, $\phi(y)=D_y(c,\hbar)$. 

Let us show that this homomorphism is injective. For this purpose 
assume that $y_i$ is a basis of $\h$, and $x_i$ is the dual basis
of $\h^*$. Then it is clear from the relations of $H_c'$ that 
$H_c'$ is spanned over $\CC[\hbar]$ by the elements 
\begin{equation}\label{basi}
g\prod_{i=1}^r y_i^{m_i}
\prod_{i=1}^r x_i^{n_i}.
\end{equation} 

Thus it remains to show that the images of the elements 
(\ref{basi}) under the map $\phi$, 
i.e. the elements 
$$
g\prod_{i=1}^r D_{y_i}(c,\hbar)^{m_i}
\prod_{i=1}^r x_i^{n_i}.
$$
are linearly independent. 
But this follows from the obvious fact that the symbols of these elements 
in $\CC G\ltimes \CC[\h^*\times \h_{\reg}][\hbar]$
are linearly independent. The proposition is proved.  
\end{proof}

\begin{remark}
1. Similarly, one can define the universal algebra $H(G, \h)$, in which
both $\hbar$ and $c$ are variables. (So this is an algebra over 
$\CC[\hbar, c]$.)
It has two equivalent definitions similar to the above.

2.  It is more convenient to work with algebras defined by generators and 
relations than with subalgebras of a given algebra generated by a 
given set of elements. Therefore, from now on we will use 
the statement of Proposition \ref{prop-RCA} 
as a definition of the rational Cherednik algebra $H_{c}$. 
According to Proposition \ref{prop-RCA}, this algebra comes with 
a natural embedding $\Theta_c: H_{c}\to \Rees(\CC G\ltimes 
\cD(\h_{\reg}))$,
defined by the formula $x\to x$, $g\to g$, $y\to D_y(c,\hbar)$. 
This embedding is called {\em the Dunkl operator embedding}. 

\end{remark}

\begin{example} 1. Let $G=\ZZ_2$, 
$\h=\CC$. In this case $c$ reduces to one parameter, and the algebra 
$H_{t,c}$ is generated by elements $x,y,s$ with defining relations
$$
s^2=1,\ sx=-xs,\ sy=-ys,\ [y,x]=t-2cs. 
$$

2. Let $G=\kS_n$, $\h=\CC^n$. In this case there is also only one 
complex parameter $c$, and the algebra 
$H_{t,c}$ is the quotient of 
$\kS_n\ltimes \CC\langle x_1,\ldots,x_n,y_1,\ldots,y_n\rangle$
by the relations 
\begin{equation*}
[x_i,x_j]=[y_i,y_j]=0,\ [y_i, x_j]=cs_{ij},\ [y_i,x_i]=t-c\sum_{j\ne i}s_{ij}.
\end{equation*}
Here $\CC\langle E\rangle$ denotes the free algebra on a set $E$, 
and $s_{ij}$ is the transposition of $i$ and $j$. 
\end{example}

\subsection{The PBW theorem for the rational Cherednik algebra}

Let us put a filtration on $H_c$ by setting $\deg y=1$ for $y\in \h$ and 
$\deg x=\deg g=0$ for
$x\in \h^{*}, g\in G$. 
Let $\gr(H_{c})$ denote the associated graded algebra 
of $H_{c}$ under this filtration, and similarly for $H_{t,c}$. We have a 
natural surjective homomorphism 
$$
\xi: \CC G\ltimes 
\CC[\h\oplus \h^*][\hbar]\to \gr(H_c).
$$  
For $t\in \CC$, it specializes to surjective homomorphisms 
$$
\xi_t: \CC G\ltimes 
\CC[\h\oplus \h^*]\to \gr(H_{t,c}).
$$

\begin{proposition}[The PBW theorem for rational Cherednik algebras]
\label{prop-pbw} 
The maps $\xi$ and $\xi_t$ are isomorphisms. 
\end{proposition}

\begin{proof} The statement is equivalent to 
the claim that the elements (\ref{basi}) are a basis 
of $H_{t, c}$, which follows from the proof of Proposition \ref{prop-RCA}. 
\end{proof}

\begin{remark} 
1. We have
$$H_{0,0}=\CC G\ltimes \CC[\h\oplus\h^{*}] \text{ and }
H_{1,0}=\CC G\ltimes \cD(\h).$$

2. For any $\lambda\in \CC^*$, the algebra $H_{t,c}$ is naturally 
isomorphic to $H_{\lambda t,\lambda c}$. 

3. The Dunkl operator embedding $\Theta_c$ specializes to 
embeddings 
$$\Theta_{0,c}: H_{0,c}\hookrightarrow \CC G\ltimes \CC[\h^*\times \h_{\reg}],$$ 
given by $x\mapsto x$, $g\mapsto g$, $y\mapsto D_a^0$, 
and 
$$\Theta_{1,c}: H_{1,c}\hookrightarrow \CC G\ltimes \cD(\h_{\reg}),$$
given by $x\mapsto x$, $g\mapsto g$, $y\mapsto D_a$. 
So $H_{0,c}$ is generated by $x, g, D_{a}^{0}$, 
and $H_{1,c}$ is generated by $x, g, D_{a}$.

Since Dunkl operators map polynomials to polynomials, 
the map $\Theta_{1,c}$ defines a representation 
of $H_{1,c}$ on $\CC[\h]$. This representation
is called the {\it polynomial representation} of $H_{1,c}$. 
\end{remark}

\subsection{The spherical subalgebra}

Let $\e\in \CC G$ be the symmetrizer, 
$\e=|G|^{-1}\sum_{g\in G}g$. We have $\e^2=\e$. 

\begin{definition}
$B_c:=\e H_{c}\be$ is called {\em the spherical subalgebra} of $H_c$.
The spherical subalgebra of $H_{t,c}$ 
is $B_{t,c}:=B_c/(\hbar-t)=\e H_{t,c}\e$.  
\end{definition} 

Note that 
$$
\e \left(\CC G\ltimes {\cD}(\h_{\reg})\right)\e=\cD(\h_{\reg})^G, \quad
\e \left(\CC G\ltimes \CC[\h_{\reg}\times \h^{*}] \right)\e
=\CC[\h_{\reg}\times \h^{*}]^{G}. 
$$
Therefore, the restriction gives the embeddings:
$\Theta_{1,c}: B_{1,c}\hookrightarrow {\cD}(\h_{\reg})^G$, 
and $\Theta_{0,c}: B_{0,c}\hookrightarrow \CC [\h^*\times \h_{\reg}]^G$.
In particular, we have
\begin{proposition}\label{prop-zerdiv} 
The spherical subalgebra 
$B_{0,c}$ is commutative and does not have zero divisors. 
Also $B_{0, c}$ is finitely generated.
\end{proposition}

\begin{proof}
The first statement is clear from the above.
The second statement follows from the fact that 
$\gr (B_{0,c})=B_{0,0}=\CC[\h\times\h^*]^G$,
which is finitely generated by Hilbert's theorem.
\end{proof}

\begin{corollary}
$M_{c}=\Spec B_{0,c}$ is an irreducible affine algebraic variety.
\end{corollary}
\begin{proof}
Directly from the definition and the proposition.
\end{proof}

We also obtain 

\begin{proposition}
$B_{c}$ is a flat quantization (non-commutative deformation) of 
$B_{0,c}$ over $\CC[\hbar]$.
\end{proposition}

So $B_{0,c}$ carries a Poisson bracket $\{\cdot, \cdot\}$(thus
$M_{c}$ is a Poisson variety), and $B_{c}$
is a quantization of the Poisson bracket, i.e.
if $a,b\in B_{c}$ and $a_{0}, b_{0}$ are the corresponding elements in $B_{0,c}$, 
then $$[a,b]/\hbar\equiv \{a_{0}, b_{0}\} \quad 
(\mathrm{ mod }\,\hbar).$$

\begin{definition}
The Poisson variety $M_c$ is called {\em the Calogero-Moser space} 
of $G,\h$ with parameter $c$. 
\end{definition}

\subsection{The localization lemma}

Let $H_{t,c}^{\mathrm loc}=H_{t,c}[\delta^{-1}]$ be the localization 
of $H_{t,c}$ as a module over $\CC[\h]$ with respect to
the discriminant $\delta$ (a polynomial vanishing to the first order on each reflection plane). 
Define also $B_{t,c}^{\mathrm loc}=\e H^{\mathrm loc}_{t,c}\e$. 

\begin{proposition}\label{prop-loc} 
\begin{enumerate}
\item[(i)] For $t\ne 0$ 
the map $\Theta_{t,c}$ induces an isomorphism 
of algebras $H_{t,c}^{\mathrm loc}\to 
\CC G\ltimes {\cD}(\h_{\reg})$, which restricts 
to an isomorphism $B_{t,c}^{\mathrm loc}\to {\cD}(\h_{\reg})^G$. 

\item[(ii)] The map $\Theta_{0,c}$ induces an isomorphism 
of algebras $H_{0,c}^{\mathrm loc}\to 
\CC G\ltimes \CC[\h^*\times \h_{\reg}]$,
which restricts to an isomorphism 
$B_{0,c}^{\mathrm loc}\to \CC[\h^*\times \h_{\reg}]^G$.
\end{enumerate}
\end{proposition}

\begin{proof}
This follows immediately from the fact that
the Dunkl operators have poles only on the reflection hyperplanes.  
\end{proof}

Since $\gr(B_{0,c})=B_{0,0}=\CC[\h^*\oplus \h]^G$,
we get the following geometric corollary. 

\begin{corollary} 
\begin{enumerate}
\item[(i)] The family of Poisson varieties 
$M_c$ is a flat deformation of the Poisson variety $M_0:=(\h\times \h^*)/G$. 
In particular, $M_c$ is smooth outside of a subset of codimension $2$. 

\item[(ii)] 
We have a natural map $\beta_c: M_c\to \h/G$, such that 
$\beta_c^{-1}(\h_{\reg}/G)$ is isomorphic to $( \h_{\reg}\times\h^*)/G$.
The Poisson structure on $M_{c}$ is obtained by extension of the symplectic Poisson structure on 
$( \h_{\reg}\times\h^*)/G$.
\end{enumerate}
\end{corollary}

\begin{example} Let $W=\ZZ_2$, $\h=\CC$. 
Then $B_{0,c}=\langle x^{2}, xp, p^{2}-c^{2}/x^{2}\rangle$.
Let $X:=x^{2}, Z:=xp$ and $Y:= p^{2}-c^{2}/x^{2}$.
Then $Z^2-XY=c^{2}$.
So $M_c$ is isomorphic to the quadric 
$Z^2-XY=c^{2}$ in the 3-dimensional space and it is smooth for $c\neq 0$.   
\end{example}

\subsection{Category $\cO$ for rational Cherednik algebras}
{\label{sec:catO}}
From the PBW theorem, we see that
$H_{1,c}=S\h^*\otimes\CC G\otimes S\h$.
It is similar to the
structure of the universal enveloping algebra of a simple Lie
algebra: $U(\g)=U(\n_-)\otimes U(\h)\otimes U(\n_+)$. 
Namely, the subalgebra $\CC G$ plays the role of the Cartan
subalgebra, and the subalgebras $S\h^*$ and $S\h$ 
play the role of the positive and negative nilpotent
subalgebras. This similarity allows one to define 
and study the category $\cO$ analogous to the
Bernstein-Gelfand-Gelfand category $\cO$ for simple Lie algebras.

\begin{definition} 
The category $\cO_c(G, \h)$ is the category of modules over 
$H_{1,c}(G, \h)$ which are finitely generated over $S\h^{*}$ and locally finite
under $S\h$ (i.e., for $M\in \cO_c(G, \h)$, $\forall v\in M$, 
$(S\h)v$ is finite dimensional). 
\end{definition}

If $M$ is a locally finite $(S\h)^{G}$-module, then 
$$M=\oplus_{\lambda\in \h^{*}/G}M_{\lambda}, 
$$
where $$M_{\lambda}=\{v\in M| \forall p\in (S\h)^{G}, \exists N \, s.t. 
\, (p-\lambda(p))^{N}v=0\},
$$
(notice that $\h^{*}/G=\mathrm{Specm}(S\h)^{G}$).

\begin{proposition}
$M_{\lambda}$ are $H_{1, c}$-submodules.
\end{proposition}

\begin{proof}
Note first that we have an isomorphism 
$\mu: H_{1,c}(G, \h)\cong H_{1,c}(G, \h^{*})$, which is given 
by $x_{a}\mapsto y_{a}, y_{b}\mapsto -x_{b}, g\mapsto g$.
Now let $x_{1}, \ldots, x_{r}$ be a basis of $\h^*$ 
and $y_{1}, \ldots, y_{r}$ a basis of $\h$.
Suppose $P=P(x_{1}, \ldots, x_{r})\in (S\h^{*})^{G}$. 
Then we have
$$
[y, P]=\frac{\partial}{\partial y}P\in S\h^{*},
\text{ where }y\in \h,
$$
(this follows from the fact that both sides act in the same way
in the polynomial representation, which is faithful). 
So using the isomorphism $\mu$, we conclude that if $Q\in (S\h)^{G}, 
Q=Q(y_{1}, \ldots, y_{r})$, 
then $[x, Q]=-\partial_{x}Q$ for $x\in \h^*$.

Now, to prove the proposition, 
the only thing we need to check is that $M_{\lambda}$ is
invariant under $x\in \h^*$.
For any $v\in M_{\lambda}$, we have $(Q-\lambda(Q))^{N}v=0$ for some $N$.
Then 
$$
(Q-\lambda(Q))^{N+1}xv=(N+1)\partial_{x}Q\cdot(Q-\lambda(Q))^{N}v=0.
$$
So $xv\in M_{\lambda}$.

\end{proof}

\begin{corollary}
We have the following decomposition:
$$\cO_{c}(G, \h)
=\bigoplus_{\lambda\in \h^{*}/G}\cO_{c}(G, \h)_{\lambda},$$
where $\cO_{c}(G, \h)_{\lambda}$ is the subcategory of
modules where $(S\h)^{G}$
acts with generalized eigenvalue $\lambda$.
\end{corollary}

\begin{proof}
Directly from the definition and the proposition.
\end{proof}

Note that $\cO_c(G,\h)_\lambda$ is an abelian category closed
under taking subquotients and extensions. 

\subsection{The grading element}

Let 
\begin{equation}{\label{eqn:h}}
\bh=\sum_{i}x_{i}y_{i}+\frac{1}{2}\dim\h-\sum_{s\in \cS}\frac{2c_{s}}{1-\lambda_{s}}s.
\end{equation}

\begin{proposition}
We have
$$[\bh,x]=x, \,x\in \h^{*},\quad [\bh,y]=-y, \,y\in \h.$$
\end{proposition}

\begin{proof}
Let us prove the first relation; the second one is proved
similarly. We have
\begin{eqnarray*}
[\bh,x]&=&\sum_{i} x_{i}[y_{i},x]-\sum_{s\in \cS}\frac{2c_{s}}{1-\lambda_{s}}\cdot
\frac{\lambda_{s}-1}{2}(\alpha_{s}^{\vee},x)\alpha_{s}\cdot s\\
&=&\sum_{i} x_{i}(y_{i},x)-\sum_{i}x_{i}\sum_{s\in \cS}c_{s}(\alpha_{s}^{\vee},x)(\alpha_s,y_i)s
+\sum_{s\in \cS}c_s(\alpha_{s}^{\vee},x)\alpha_{s}\cdot s.
\end{eqnarray*}
The last two terms cancel since $\sum_{i}x_{i}(\alpha_{s}, y_{i})=\alpha_{s}$, 
so we get $\sum_{i}x_{i}(y_{i},x)=x$. 
\end{proof}

\begin{proposition}\label{sl2rep}
Let $G=W$ be a real reflection group. 
Let 
$$\bh=\sum_{i}x_{i}y_{i}+\frac{1}{2}\dim\h-\sum_{s\in \cS}c_{s}s, \quad
{\mathbf E}=-\dfrac{1}{2}\sum_{i}x_{i}^{2},\quad
{\mathbf F}=\dfrac{1}{2}\sum_{i}y_{i}^{2}.$$
Then
\begin{enumerate}
\item[(i)] $\bh=\sum_{i}(x_{i}y_{i}+y_{i}x_{i})/2$;

\item[(ii)] $\bh, {\mathbf E}, {\mathbf F}$ form an $\mathfrak{sl}_{2}$-triple.
\end{enumerate}
\end{proposition}
\begin{proof}
A direct calculation.
\end{proof}

\begin{theorem}\label{finnilp}
Let $M$ be a module over $H_{1,c}(G, \h)$. 
\begin{enumerate}
\item[(i)] If $\h$ acts locally nilpotently on $M$, then   
$\bh$ acts locally finitely on $M$. 
\item[(ii)] If $M$ is finitely generated over $S\h^{*}$,
then $M\in \cO_{c}(G, \h)_{0}$ 
if and only if $\bh$ acts locally finitely on $M$.
\end{enumerate}
\end{theorem}

\begin{proof}
(i) Assume that $S\h$ acts locally nilpotently on $M$. 
Let $v\in M$. Then $S\h\cdot v$ is a finite dimensional vector space
and let $d=\dim S\h\cdot v$. We prove that
$v$ is $\bh$-finite by induction in dimension $d$.
We can use $d=0$ as base, so only need to do the induction step. 
The space $S\h\cdot v$ must contain a nonzero vector $u$ such that 
$y\cdot u=0$, $\forall y\in \h$. 
Let $U\subset M$ be the subspace of vectors with this property.
$\bh$ acts on $U$ by an element of $\CC G$, 
hence locally finitely. So it is sufficient to
show that the image of $v$ in $M/\langle U\rangle$ is 
$\bh$-finite (where $\langle U\rangle$ is
the submodule generated by $U$). 
But this is true by the induction assumption, 
as $u=0$ in $M/\langle U\rangle$.

(ii) We need to show that if $\bh$ acts locally finitely on $M$,
then $\h$ acts locally nilpotently on $M$. 
Assume $\bh$ acts locally finitely on $M$. Then 
$M=\oplus_{\beta\in B}M[\beta]$, where $B\subset \CC$. Since $M$ is finitely generated over 
$S\h^{*}$, $B$ is a finite union of sets of the form $z+\ZZ_{\geq 0}$, 
$z\in \CC$.
So $S\h$ must act locally nilpotently on $M$.
\end{proof}
We can obtain the following corollary easily.
\begin{corollary}
Any finite dimensional 
$H_{1,c}(G, \h)$-module is in $\cO_{c}(G, \h)_{0}$.
\end{corollary}

We see that any module $M\in \cO_{c}(G, \h)_0$
has a grading by generalized eigenvalues of $\bh$:
$M=\oplus_{\beta} M[\beta]$. 

\subsection{Standard modules}

Let $\tau$ be a finite dimensional representation of $G$. 
The standard module over $H_{1,c}(G, \h)$ corresponding to $\tau$ 
(also called the Verma module) is
$$
M_{c}(G, \h, \tau)=M_{c}(\tau)=H_{1,c}(G, \h)
\otimes _{\CC G\ltimes S\h}\tau\,\,\in 
\cO_{c}(G, \h)_{0},
$$
where $S\h$ acts on $\tau$ by zero.

So from the PBW theorem, we have that 
as vector spaces, $M_{c}(\tau)\cong \tau \otimes S\h^{*}$. 

\begin{remark}
More generally, $\forall \lambda\in \h^{*}$, let
$G_{\lambda}=\mathrm{Stab}(\lambda)$, 
and $\tau$ be a finite dimensional representation of $G_\lambda$.
Then we can define 
$M_{c, \lambda}(G, \h, \tau)=H_{1,c}(G, \h)\otimes _{\CC G_{\lambda}\ltimes S\h}
\tau$, where $S\h$ acts on $\tau$ by $\lambda$. These
modules are called  {\em the Whittaker modules}.
\end{remark}

Let $\tau$ be irreducible, and let $h_c(\tau)$ be 
the number given by the formula 
$$
h_c(\tau)=\frac{\dim \h}{2}-\sum_{s\in \cS}\frac{2c_s}
{1-\lambda_s}s|_\tau.
$$
Then we see that
$\bh$ acts on $\tau\otimes S^m\h^*$ by the scalar $h_c(\tau)+m$. 

\begin{definition}
A vector $v$ in an $H_{1,c}$-module $M$ is singular if $y_iv=0$
for all $i$. 
\end{definition}

\begin{proposition}\label{homm}
Let $U$ be an $H_{1,c}(G, \h)$-module.  
Let $\tau\subset U$ be a $G$-submodule 
consisting of singular vectors.
Then there is a unique homomorphism $\phi: M_{c}(\tau)\to U$ of $\CC[\h]$-modules such that 
$\phi|_{\tau}$ is the identity, and it
is an $H_{1,c}$-homomorphism. 
\end{proposition}

\begin{proof}
The first statement follows from the fact that $M_c(\tau)$ is a free module over $\CC[\h]$ 
generated by $\tau$. Also, it follows from the 
Frobenius reciprocity that there must exist a map $\phi$ which is an $H_{1,c}$-homomorphism. 
This implies the proposition. 
\end{proof}

\subsection{Finite length}

\begin{proposition} $\exists K\in \RR$ such that 
for any $M\subset N$ in $\cO_{c}(G, \h)_0$,
if $M[\beta]=N[\beta]$ for $\Re(\beta)\leq K$,
then $M=N$. 
\end{proposition}

\begin{proof}
Let $K=\max_{\tau}\Re h_{c}(\tau)$. Then if $M\ne N$, 
$M/N$ begins in degree $\beta_{0}$
with $\Re\beta_{0}>K$, which is impossible since 
by Proposition \ref{homm}, $\beta_{0}$ must equal $h_{c}(\tau)$
for some $\tau$.
\end{proof}

\begin{corollary}
Any $M\in \cO_{c}(G, \h)_{0}$ has finite length.
\end{corollary}
\begin{proof}
Directly from the proposition.
\end{proof}

\subsection{Characters}
For $M\in \cO_{c}(G, \h)_0$, 
define the character of $M$ as the following formal series in $t$:
$$
\ch_{M}(g, t)=\sum_{\beta}t^{\beta}\tr_{M[\beta]}(g)=\tr_{M}(gt^{\bh}), \quad g\in G.
$$

\begin{proposition}
We have
$$
\ch_{M_{c}(\tau)}(g, t)=\frac{\chi_{\tau}(g)t^{h_{c}(\tau)}}
{\det_{\h^{*}}(1-tg)}.
$$
\end{proposition}

\begin{proof}
We begin with the following lemma.
\begin{lemma}[MacMahon's Master theorem]
Let $V$ be a finite dimensional space, $A: V\to V$ a linear operator. Then 
$$\sum_{n\geq 0}t^{n}\tr (S^{n}A)=\frac{1}{\det (1-tA)}.$$
\end{lemma}
\begin{proof}[Proof of the lemma]
If $A$ is diagonalizable, this is obvious. 
The general statement follows by continuity.
\end{proof}

The lemma implies that 
$\tr_{S\h^{*}}(gt^{D})=\dfrac{1}{\det(1-gt)}$ where $D$ is the degree operator.
This implies the required statement. 
\end{proof}

\subsection{Irreducible modules}

Let $\tau$ be an irreducible representation of $G$. 

\begin{proposition}
$M_{c}(\tau)$ has a maximal proper submodule $J_{c}(\tau)$. 
\end{proposition}

\begin{proof}
The proof is standard. $J_{c}(\tau)$ is the sum of all proper submodules of $M_{c}(\tau)$, and it is
not equal to $M_{c}(\tau)$ because any proper submodule has a
grading by generalized eigenspaces of $\bh$, with eigenvalues
$\beta$ such that $\beta-h_c(\tau)>0$. 
\end{proof} 

We define $L_{c}(\tau)=M_{c}(\tau)/J_{c}(\tau)$, which is
an irreducible module.

\begin{proposition}
Any irreducible object of $\cO_{c}(G, \h)_{0}$ has the form $L_{c}(\tau)$ for an unique 
$\tau$.
\end{proposition}

\begin{proof} 
Let $L\in \cO_c(G,\h)_0$ be irreducible, with lowest eigenspace of
$\bh$ containing an irreducible $G$-module $\tau$. 
Then by Proposition \ref{homm}, we have a nonzero homomorphism 
$M_c(\tau)\to L$, which is surjective, since $L$ is irreducible.  
Then we must have $L=L_c(\tau)$. 
\end{proof} 

\begin{remark}\label{chara}
Let $\chi$ be a character of $G$. Then we have an isomorphism
$H_{1,c}(G, \h)\to H_{1,c\chi}(G, \h)$, mapping $g\in G$ to $\chi^{-1}(g)g$.
This automorphism maps $L_{c}(\tau)$ to $L_{c\chi}(\chi^{-1}\otimes \tau)$ isomorphically.
\end{remark}

\subsection{The contragredient module}{\label{sec:contmod}}
Set $\bar{c}(s)=c(s^{-1})$. We have a natural isomorphism
$\gamma: H_{1,\bar{c}}(G, \h^{*})^\op\to H_{1,c}(G, \h)$,
acting trivially on $\h$ and $\h^{*}$, and sending $g \in G$ to $g^{-1}$.

Thus if $M$ is an $H_{1,c}(G, \h)$-module, then the full dual space $M^{*}$ is an 
$H_{1,\bar{c}}(M, \h^{*})$-module.
If $M\in \cO_{c}(G, \h)_{0}$, 
then we can define $M^{\dag}$, which is the $\bh$-finite 
part of $M^{*}$.

\begin{proposition}
$M^{\dag}$ belongs to $\cO_{\bar{c}}(G, \h^{*})_0$. 
\end{proposition}

\begin{proof}
Clearly, if $L$ is irreducible, then so is $L^\dagger$.
Then $L^\dagger$ is generated by its lowest 
$\bh$-eigenspace over $H_{1,\bar{c}}(G,\h^{*})$, hence over $S\h^*$. 
Thus, $L^\dagger\in \cO_{\bar c}(G,\h^*)_0$. Now, let $M\in \cO_c(G,\h)_0$ 
be any object. Since $M$ has finite length, so does 
$M^\dagger$. Moreover, $M^\dagger$ has a finite filtration with
successive quotients of the form $L^\dagger$,
where $L\in \cO_c(G,\h)_0$ is irreducible. This implies the
required statement, since $\cO_c(G,\h)_0$ is closed under taking
extensions.     
\end{proof} 

Clearly, $M^{\dag\dag}=M$. Thus, $M\mapsto M^\dagger$ is an
equivalence of categories $\cO_c(G, \h)\to 
\cO_{\bar{c}}(G, \h^{*})^{\op}$. 

\subsection{The contravariant form}{\label{sec:cvform}}

Let $\tau$ be an irreducible representation of $G$. 
By Proposition \ref{homm}, we have a unique homomorphism 
$\phi: M_{c}(G, \h, \tau)\to M_{\bar{c}}(G, \h^{*}, \tau^{*})^{\dag}$
which is the identity in the lowest $\bh$-eigenspace.
Thus, we have a pairing 
$$
\beta_{c}: M_{c}(G, \h, \tau)\times M_{\bar{c}}(G, \h^{*},
\tau^{*})\to \CC,
$$
which is called {\em the contravariant form}. 

\begin{remark} If $G=W$ is a real reflection group, 
then $\h\cong \h^{*}$, $c=\bar{c}$, 
and $\tau\cong \tau^*$ via a symmetric form. 
So for real reflection groups, $\beta_c$ is a 
symmetric form on $M_{c}(\tau)$.
\end{remark}

\begin{proposition} The maximal proper submodule 
$J_c(\tau)$ is the kernel of $\phi$
(or, equivalently, of the contravariant form $\beta_c$). 
\end{proposition}

\begin{proof} 
Let $K$ be the kernel of the contravariant form. 
It suffices to show that $M_{c}(\tau)/K$ is irreducible. We have a diagram:
$$\xymatrix{M_{c}(G, \h, \tau)\ar[d]\ar[rd]^-{\xi} \ar[r]^-{\phi} &  M_{c}(G, \h^{*}, 
\tau^{*})^{\dag}\\
L_{c}(G, \h, \tau)\ar[r]^{\sim}_-{\eta} & L_{c}(G, \h^{*}, \tau^{*})^{\dag}\ar@{^{(}->}[u]}
$$
Indeed, a nonzero map $\xi$ exists by Proposition \ref{homm}, 
and it factors through $L_{c}(G,\h,\tau)$, with $\eta$ being an isomorphism,
since $L_{c}(G, \h^{*}, \tau^{*})^{\dag}$ is irreducible. 
Now, by Proposition \ref{homm} (uniqueness of $\phi$), the
diagram must commute up to scaling, which implies the statement.  
\end{proof}

\begin{proposition}
Assume that $h_{c}(\tau)-h_{c}(\tau')$ never equals a positive integer
for any $\tau, \tau'\in \mathrm{Irrep}G$. Then
$\cO_{c}(G, \h)_{0}$ is semisimple, with simple objects $M_{c}(\tau)$.
\end{proposition}
\begin{proof}
It is clear that in this situation, all $M_{c}(\tau)$ are simple.
Also consider $\Ext^{1}(M_{c}(\tau), M_{c}(\tau'))$. If $h_{c}(\tau)-h_{c}(\tau')\notin \ZZ$, 
it is clearly $0$. Otherwise, $h_{c}(\tau)=h_{c}(\tau')$, and again 
$\Ext^{1}(M_{c}(\tau), M_{c}(\tau'))=0$, since for any extension 
$$
0\to M_{c}(\tau')\to N\to M_{c}(\tau)\to 0, $$
by Proposition \ref{homm} 
we have a splitting $M_{c}(\tau)\to N$.
\end{proof}

\begin{remark}
In fact, our argument shows that 
if $\Ext^{1}(M_{c}(\tau), M_{c}(\tau'))\neq 0$, then 
$h_{c}(\tau)-h_{c}(\tau')\in \mathbb N$.
\end{remark} 

\subsection{The matrix of multiplicities} 

For $\tau, \sigma\in {\mathrm Irrep}G$,
write $\tau<\sigma$ if $$\Re h_{c}(\sigma)-\Re
h_{c}(\tau)\in \mathbb N.$$

\begin{proposition}
There exists a matrix of integers $N=(n_{\sigma, \tau})$, with $n_{\sigma, \tau}\geq 0$, 
such that $n_{\tau,\tau}=1$, $n_{\sigma, \tau}=0$ unless $\sigma<\tau$, and
$$
M_{c}(\sigma)=\sum n_{\sigma, \tau}L_{c}(\tau)\, \in \K_{0}(\cO_{c}(G, \h)_{0}).
$$
\end{proposition}

\begin{proof}
This follows from the Jordan-H\"older theorem and the fact that
objects in $\cO_c(G, \h)_0$ have finite length. 
\end{proof} 

\begin{corollary}
Let $N^{-1}=(\bar{n}_{\tau, \sigma})$. Then 
$$
L_{c}(\tau)=\sum \bar{n}_{\tau, \sigma}M_{c}(\sigma).
$$
\end{corollary}

\begin{corollary}
We have
$$
\ch_{L_{c}(\tau)}(g,t)=\frac{\sum \bar{n}_{\tau, \sigma}\chi_{\sigma}(g)t^{h_{c}(\tau)}}
{\det_{\h^{*}}(1-tg)}.
$$
\end{corollary}

Both of the corollaries can be obtained from the above proposition easily.

One of the main problems in the representation theory of rational
Cherednik algebras is the following problem. 

{\bf Problem:}
Compute the multiplicities $n_{\sigma, \tau}$ or, equivalently, $\ch_{L_{c}(\tau)}$ for all $\tau$.

In general, this problem is open.

\subsection{Example: the rank $1$ case}
Let $G=\ZZ/m\ZZ$ and $\lambda$ be an $m$-th primitive root of $1$.
Then the algebra $H_{1,c}(G, \h)$ is generated by $x, y, s$ with relations
$$
[y, x]=1-2\sum_{j=1}^{m-1}c_{j}s^{j},\quad sxs^{-1}=\lambda x, \quad sys^{-1}=\lambda^{-1}y.
$$ 

Consider the one-dimensional space $\CC$ and let $y$ act by $0$ and $g\in G$ act by $1$.
We have $M_{c}(\CC)=\CC[x]$. The contravariant form ${\bf \beta}_{c, \CC}$ on $M_{c}(\CC)$
is defined by 
$$
\beta_{c, \CC}(x^{n}, x^{n})=a_{n};\quad 
\beta_{c, \CC}(x^{n}, x^{n'})=0, n\ne n'.
$$
Recall that $\beta_{c, \CC}$ satisfies 
$\beta_{c,\CC }(x^{n}, x^{n})=\beta_{c,\CC }(x^{n-1}, yx^{n})$, 
which gives
$$
a_{n}=a_{n-1}(n-b_{n}), 
$$
where $b_{n}$ are new parameters:
$$
b_n:=2\sum_{j=1}^{m-1}\frac{1-\lambda^{jn}}{1-\lambda^{j}}c_{j}
\quad (b_{0}=0, b_{n+m}=b_{n}).
$$

Thus we obtain the following proposition. 

\begin{proposition}
\begin{enumerate}
\item[(i)] $M_{c}(\CC)$ is irreducible if only if $n-b_{n}\neq 0$ for any $n\geq 1$.
\item[(ii)] Assume that $r$ is the smallest positive integer such that $r=b_{r}$. Then $L_{c}(\CC)$
has dimension $r$ (which can be any number not divisible by $m$) with basis
$1, x, \ldots, x^{r-1}$. 
\end{enumerate}
\end{proposition}

\begin{remark}
According to Remark \ref{chara}, 
this proposition in fact describes all 
the irreducible lowest weight modules.
\end{remark}

\begin{example}
Consider the case $m=2$. The $M_{c}(\CC)$ is irreducible
unless $c\in 1/2+\ZZ_{\ge 0}$.
If $c=(2n+1)/2\in 1/2+\ZZ$, $n\ge 0$, then
$L_{c}(\CC)$ has dimension $2n+1$. A similar answer is
obtained for lowest weight $\CC_-$, replacing $c$ by $-c$. 
\end{example}

\subsection{The Frobenius property}
Let $A$ be a $\ZZ_+$-graded commutative algebra. 
The algebra $A$ is called Frobenius if the top degree $A[d]$ of
$A$ is $1$-dimensional, and the multiplication map 
$A[m]\times A[d-m]\to A[d]$ is a nondegenerate pairing for any
$0\leq m\leq d$. In particular, the Hilbert polynomial 
of a Frobenius algebra $A$ is palindromic.  

Now, let us go back to considering modules 
over the rational Cherednik algebra $H_{1,c}$. 
Any submodule $J$ of the polynomial 
representation $M_c(\CC)=M_c=\CC[\h]$ is an ideal in $\CC[\h]$, 
so the quotient $A=M_c/J$ is a $\ZZ_+$-graded commutative algebra. 

Now suppose that $G$ preserves an inner product in $\h$, i.e., $G\subseteq \tO(\h)$. 

\begin{theorem}{\label{thm-irr}}
If $A=M_{c}(\CC)/J$ is finite dimensional, then $A$ is irreducible ($A=L_{c}(\CC)$)
$\Longleftrightarrow$
$A$ is a Frobenius algebra.
\end{theorem}
\begin{proof}
1) Suppose $A$ is an irreducible $H_{1, c}$-module, i.e., $A=L_{c}(\CC)$.
By Proposition \ref{sl2rep}, $A$ is naturally a finite dimensional $\mathfrak{sl}_{2}$-module
(in particular, it integrates to the group $\SL_2(\CC)$).
Hence, by the representation theory of $\mathfrak{sl}_{2}$, the top
degree of $A$ is $1$-dimensional.
Let $\phi\in A^{*}$ denote a nonzero linear function on the top component. Let $\beta_c$ be the 
contravariant form on $M_{c}(\CC)$.  Consider the form 
$$
(v_{1}, v_{2})\mapsto E(v_{1}, v_{2}):=\beta_c(v_{1}, gv_{2}), \text{ where } 
g=\left(\begin{array}{cc}0 & 1 \\-1 & 0\end{array}\right)\in \SL_{2}(\CC).
$$
Then $E(xv_{1}, v_{2})=E(v_{1}, xv_{2})$. 
So for any $p,q\in M_{c}(\CC)=\CC[\h]$, $E(p,q)=\phi(p(x)q(x))$ 
(for a suitable normalization of $\phi$).

Since $E$ is a nondegenerate form, $A$ is a Frobenius algebra.

2) Suppose $A$ is Frobenius. Then the highest component is $1$-dimensional, 
and \linebreak $E: A\otimes A \to \CC$, $E(a, b)=\phi(ab)$ is nondegenerate. 
We have $E(xa, b)=E(a, xb)$. So set $\beta(a, b)=E(a, g^{-1}b)$. 
Then $\beta$ satisfies $\beta(a, x_{i}b)=\beta(y_{i}a, b)$.
Thus, for all $p,q\in \CC[\h]$, 
$\beta(p(x), q(x))=\beta(q(y)p(x), 1)$. 
So $\beta=\beta_c$ up to scaling. Thus, $\beta_c$ is nondegenerate and $A$ is irreducible.
\end{proof}

\begin{remark}
If $G\nsubseteq \tO(\h)$, this theorem is false, in general.
\end{remark}

Now consider the Frobenius property of $L_{c}(\CC)$ for any 
$G\subset \GL(\h)$.

\begin{theorem}{\label{thm-frob}}
Let $U\subset M_{c}(\CC)=\CC[\h]$ be a $G$-subrepresentation of dimension $l=\dim \h$,
sitting in degree $r$, which consists of singular vectors.
Let $J=\langle U\rangle$. Assume that $A=M_{c}/J$ is finite dimensional. Then
\begin{enumerate}
\item[(i)] $A$ is Frobenius.
\item[(ii)] $A$ admits a BGG type resolution:
$$
A\leftarrow M_c(\CC)\leftarrow M_c(U)\leftarrow M_c(\wedge^2U)
\leftarrow\cdots\leftarrow M_c(\wedge^l U)\leftarrow 0. 
$${\label{itm:bgg}}

\item[(iii)] The character of $A$ is given by the formula 
$$
\chi_{A}(g, t)
=t^{\frac{l}{2}-\sum_{s\in \cS}\frac{2c_{s}}{1-\lambda_{s}}}\frac{\det_{U}(1-gt^r)}
{\det_{\h^{*}}(1-gt)}.
$$
In particular, $\dim A=r^{l}$.
\item[(iv)] If $G$ preserves an inner product, then $A$ is irreducible.

\end{enumerate}

\end{theorem}

\begin{proof}
(i) Since $\Spec\,A $ is a complete intersection, $A$ is Frobenius.

(ii) We will use the following theorem: 

\begin{theorem}[Serre]
Let $f_{1}, \ldots, f_{n}\in \CC[t_{1}, \ldots, t_{n}]$ be homogeneous polynomials, and
assume that $\CC[t_{1}, \ldots, t_{n}]$ is a finitely generated module over
$\CC[f_{1}, \ldots, f_{n}]$. Then this is a free module.
\end{theorem}

Consider $SU\subset S\h^{*}$. Then $S\h^{*}$ is a finitely 
generated $SU$-module (as $S\h^{*}/\langle U\rangle$ is finite dimensional).
By Serre's theorem, we know that $S\h^{*}$ is a free $SU$-module.
The rank of this module is $r^{l}$. Consider the Koszul complex
attached to this module. Since the module is free, the Koszul
complex is exact 
(i.e., it is a resolution of the zero fiber).
At the level of $SU$-modules, it looks exactly like we want in \eqref{itm:bgg}.

So we only need to show that the maps of the resolution are morphisms over $H_{1,c}$.
This is shown by induction. Namely, let $\delta_j:
M_c(\wedge^jU)\to M_c(\wedge^{j-1}U)$ be the corresponding
differentials (so that $\delta_0: M_c(\CC)\to A$ is the
projection). Then $\delta_0$ is an $H_{1,c}$-morphism, which is the
base of induction. If $\delta_j$ is an $H_{1,c}$-morphism, 
then the kernel of $\delta_j$ is a submodule $K_j\subset
M_c(\wedge^jU)$. Its lowest degree part is $\wedge^{j+1}U$ 
sitting in degree $(j+1)r$ and consisting of singular vectors. 
Now, $\delta_{j+1}$ is a morphism over
$S\h^*$ which maps $\wedge^{j+1}U$ identically to
itself. By Proposition \ref{homm}, there is only one such morphism, and 
it must be an $H_{1,c}$-morphism. This completes the
induction step. 

(iii) follows from (ii) by the Euler-Poincar\'e formula.

(iv) follows from Theorem \ref{thm-irr}. 

\end{proof}

\subsection{Representations of $H_{1,c}$ of type $A$}
Let us now apply the above results to the case of type $A$. 
We will follow the paper \cite{CE}. 

Let $G=\kS_{n}$, and $\h$ be its reflection representation. 
In this case the function $c$ reduces to one number. 
We will denote the rational Cherednik algebra $H_{1,c}(\kS_{n})$ 
by $H_c(n)$. It is generated 
by $x_{1}, \ldots, x_{n}$, $y_{1}, \ldots, y_{n}$ and $\CC\kS_n$
with the following relations:
$$
\sum y_{i}=0,\quad \sum x_{i}=0,\quad 
[y_{i}, x_{j}]=-\frac{1}{n}+cs_{ij}, i\neq j,
$$
$$
[y_{i}, x_{i}]=\frac{n-1}{n}-c\sum_{j\neq i}s_{ij}.
$$
The polynomial representation $M_c(\CC)$ of this algebra is the space of
$\CC[x_1, \ldots, x_n]^T$ of polynomials of $x_1,\ldots,x_n$, which are 
invariant under simultaneous translation $T: x_i\mapsto x_i+a$. 
In other words, it is the space of regular functions
on $\h=\CC^n/\Delta$, where $\Delta$ is the diagonal.

\begin{proposition}[C. Dunkl] \label{Asing}
Let $r$ be a positive integer not divisible by $n$, and $c=r/n$. 
Then $M_{c}(\CC)$ contains a copy of the reflection representation 
$\h$ of $\kS_{n}$, which consists of singular vectors
(i.e. those killed by $y\in \h$). 
This copy sits in degree $r$ and is spanned 
by the functions 
$$
f_i(x_1, \ldots, x_n)=\Res_\infty [(z-x_1)\cdots(z-x_n)]^{\frac{r}{n}}
\frac{\d z}{z-x_i}.
$$
(the symbol $\Res_\infty$ denotes the residue 
at infinity). 
\end{proposition} 

\begin{remark} 
The space spanned by $f_i$ is $(n-1)$-dimensional, since 
$\sum_i f_i=0$ (this sum is the residue of an exact differential). 
\end{remark}

\begin{proof}
This proposition can be proved by a straightforward computation. 
The functions $f_i$ are a special case of 
Jack polynomials.
\end{proof}

Let $I_c$ be the submodule of $M_c(\CC)$ 
generated by $f_i$. 
Consider the $H_c(n)$-module $V_c=M_c(\CC)/I_c$, and 
regard it as a $\CC[\h]$-module.
We have the following results.

\begin{theorem}\label{thm-Arep} 
Let $d=(r,n)$ denote the greatest common 
divisor of $r$ and $n$. 
Then the (set-theoretical) 
support of $V_c$ is the union of $\kS_n$-translates 
of the subspaces of $\CC^n/\Delta$,
defined by the equations 
$$
\begin{array}{c}
x_1=x_2 = \cdots =x_{\frac{n}{d}};\quad
x_{\frac{n}{d}+1}= \cdots =x_{2\frac{n}{d}};\quad
\dots\quad
x_{(d-1)\frac{n}{d}+1}= \cdots = x_n.
\end{array}
$$
In particular, the Gelfand-Kirillov dimension of $V_c$ 
is $d-1$. 
\end{theorem}

\begin{corollary} [\cite{BEG}] If $d=1$ 
then the module $V_c$ is finite dimensional, irreducible, 
admits a BGG type resolution, and its character is 
$$
\chi_{V_c}(g,t)=t^{(1-r)(n-1)/2}\frac{\det|_\h(1-gt^r)}{\det|_\h(1-gt)}.
$$
\end{corollary}

\begin{proof} 
For $d=1$ 
Theorem \ref{thm-Arep} says that the support of 
$M_c(\CC)/I_c$ is $\{0\}$. This implies that 
$M_c(\CC)/I_c$ is finite dimensional. The rest follows from  
Theorem \ref{thm-frob}.
\end{proof} 

\begin{proof}[Proof of Theorem \ref{thm-Arep}]
The support of $V_c$ is the zero set of $I_c$, i.e. 
the common zero set of $f_i$. 
Fix $x_1,\ldots,x_n\in \CC$. Then $f_i(x_1,\ldots,x_n)=0$ for all
$i$ iff 
$\displaystyle \sum_{i=1}^n \lambda_i f_i=0$ 
for all $\lambda_i$, i.e.
$$
\operatorname{Res}_{\infty}\left( 
\prod_{j=1}^n(z-x_j)^{\frac{r}{n}}\sum_{i=1}^n\frac{\lambda_i}{z-x_i} 
\right)\d z=0.  
$$

Assume that $x_1, \ldots x_n$ take distinct values $y_1, \ldots, y_p$ with 
positive multiplicities $m_1, \dots, m_p$.
The previous equation implies that the point 
$(x_1,\ldots,x_n)$ is in the zero set iff
$$
{\mathrm Res}_\infty 
\prod_{j=1}^p (z-y_j)^{m_j\frac{r}{n}-1} \left(\sum_{i=1}^p
\nu_i(z-y_1)\cdots \widehat{(z-y_i)} \cdots (z-y_p)\right)\d z=0 \quad
\forall \nu_i.
$$
Since $\nu_i$ are arbitrary, this is equivalent to the 
condition
$$
{\mathrm Res}_\infty 
\prod_{j=1}^p (z-y_j)^{m_j\frac{r}{n}-1}z^i \d z=0, \quad i=0, \ldots, 
p-1. 
$$

We will now need the following lemma. 

\begin{lemma}\label{mainlemma}
Let $\displaystyle a(z)=\prod_{j=1}^p(z-y_j)^{\mu_j}$, where 
$\mu_j\in \CC$, $\sum_j \mu_j 
\in \ZZ$ 
and $\sum_j \mu_j > -p$. Suppose 
$${\mathrm Res}_\infty a(z)z^i\d z=0, \quad i=0,1,\ldots, p-2.$$
Then $a(z)$ is a polynomial.
\end{lemma}

\begin{proof}
Let $g(z)$ be a polynomial.
Then 
$$0={\mathrm Res}_\infty \d(g(z)\cdot a(z))=
{\rm Res}_\infty(g^{\prime}(z)a(z)+a^{\prime}(z)g(z))\d z,$$
and hence
$${\mathrm Res}_\infty \left(g^{\prime}(z)+ \sum_i 
\frac{\mu_j}{z-y_j}g(z)\right)a(z)\d z=0.$$

Let $\displaystyle g(z)=z^l \prod_j(z-y_j)$. Then $ 
\displaystyle g^{\prime}(z)+ \sum_j 
\frac{\mu_j}{z-y_j}g(z)$ is a polynomial of degree $l+p-1$ with highest 
coefficient $l+p+\sum \mu_j \ne 0$ (as $\sum \mu_j>-p$). 
This means that for every $l \ge 0$,
$\displaystyle {\mathrm Res}_\infty z^{l+p-1}a(z)\d z$ is a linear combination of 
residues of $z^qa(z)\d z$ 
with $q<l+p-1$. By the assumption of the lemma, 
this implies by induction in $l$ 
that all such residues are $0$ and hence $a$ is a polynomial.
\end{proof}

In our case $\sum (m_j r/n-1)=r-p$ (since $\sum m_j=n$) and 
the conditions of the lemma are satisfied. Hence 
$(x_1,\ldots,x_n)$ is in the zero set of $I_c$ iff
$\displaystyle\prod_{j=1}^p(z-y_j)^{m_j\frac{r}{n}-1}$ is a polynomial.
This is equivalent to saying that all  $m_j$ are divisible by 
$n/d$. 

We have proved that $(x_1, \ldots, x_n)$ is in the zero set of $I_c$ if and only if
$(z-x_1) \cdots (z-x_n)$ is the $(n/d)$-th power of a polynomial of degree $d$.
This implies the theorem. 
\end{proof}

\begin{remark}
For $c>0$, the above representations are the only irreducible finite dimensional
representations of $H_{1,c}(\kS_n)$. Namely,  
it is proved in \cite{BEG} that the only finite dimensional representations 
of $H_{1,c}(\kS_{n})$ are multiples of $L_c(\Bbb C)$ for $c=r/n$, and of $L_c(\Bbb C_-)$ 
(where $\Bbb C_-$ is the sign representation) for $c=-r/n$, where $r$ is a positive integer 
relatively prime to $n$. 
\end{remark}

\subsection{Notes}
The discussion of the definition 
of rational Cherednik algebras and their basic properties 
follows Section 7 of \cite{E4}.
The discussion of the category $\cO$ for rational Cherednik algebras follows
Section 11 of \cite{E4}. The material in Sections 3.14-3.16 
is borrowed from \cite{CE}.

\newpage \section{The Macdonald-Mehta integral}

\subsection{Finite Coxeter groups and the Macdonald-Mehta integral}
{\label{sec:coxeter}}
Let $W$ be a finite  
Coxeter group of rank $r$ with real reflection representation $\h_{\RR}$
equipped with a Euclidean $W$-invariant inner product $(\cdot , \cdot )$.
Denote by $\h$ the complexification of $\h_{\RR}$. 
The reflection hyperplanes subdivide $\h_{\RR}$ into $|W|$ 
chambers; let us pick one of them to be the 
dominant chamber and call its interior $D$. 
For each reflection hyperplane, pick the perpendicular vector 
$\alpha\in \h_\RR$ with $(\alpha,\alpha)=2$ 
which has positive inner products with elements of $D$, and 
call it the positive root corresponding to this hyperplane.
The walls of $D$ are then defined by the equations $(\alpha_i,v)=0$, 
where $\alpha_i$ are simple roots. 
Denote by $\cS$ the set of reflections in $W$, 
and for a reflection $s\in \cS$ 
denote by $\alpha_s$ the corresponding 
positive root. 
Let 
$$
\delta({\mathbf x})=\prod_{s\in \cS}(\alpha_s,{\mathbf x})
$$ 
be the corresponding discriminant 
polynomial. Let $d_i,i=1,\ldots,r$, be the degrees of the
generators of the algebra $\CC[\h]^W$. 
Note that $|W|=\prod_i d_i$. 

Let $H_{1, c}(W, \h)$ be the rational Cherednik algebra of $W$. 
Here we choose $c=-k$ as a constant function.
Let $M_c=M_c(\CC)$ be the polynomial representation of $H_{1,c}(W, \h)$,
and 
$\beta_{c}$ be the contravariant form on $M_{c}$ 
defined in Section \ref{sec:cvform}.
We normalize it by
the condition $\beta_{c}(1, 1)=1$.

\begin{theorem}\label{thm-MMI}
\begin{enumerate}
\item[(i)] (The Macdonald-Mehta integral) For $\Re(k)\ge 0$, one has 
\begin{equation}\label{eqn-MMI}
(2\pi)^{-r/2}\int_{\h_{\RR}}\be^{-({\mathbf x},{\mathbf x})/2}|\delta({\mathbf x})|^{2k}\d {\mathbf x}=
\prod_{i=1}^r \frac{\Gamma(1+kd_i)}{\Gamma(1+k)}.
\end{equation}
\item[(ii)] Let $b(k):=\beta_c(\delta,\delta)$. Then
$$
b(k)=|W|\prod_{i=1}^r\prod_{m=1}^{d_i-1}(kd_i+m).
$$
\end{enumerate}
\end{theorem}

For Weyl groups, this theorem was proved by E. Opdam \cite{Op1}.
The non-crystallographic cases were done by Opdam in \cite{Op2} 
using a direct computation in the rank 2 case 
(reducing \eqref{eqn-MMI} to the beta integral by passing to polar coordinates), 
and a computer calculation by F. Garvan for $H_3$ and $H_4$. 

\begin{example}
In the case $W=\kS_{n}$, we have the following integral (the Mehta integral):
$$
(2\pi)^{-(n-1)/2}\int_{\{{\mathbf x}\in \RR^{n}|\sum_{i}x_{i}=0\}}
\be^{-({\mathbf x}, {\mathbf x})/2}\prod_{i\neq j} |x_{i}-x_{j}|^{2k}\d {\mathbf x}=
\prod_{d=2}^{n} \frac{\Gamma(1+kd)}{\Gamma(1+k)}.
$$
\end{example}

In the next subsection, we give a uniform proof of Theorem \ref{thm-MMI}
which is given in \cite{E2}.
We emphasize that many parts of this proof are borrowed from  
Opdam's previous proof of this theorem. 

\subsection{Proof of Theorem \ref{thm-MMI}}{\label{sec:pfMMI}}

\begin{proposition}\label{c1}
The function $b$ is a polynomial of degree at most $|\cS|$, and 
the roots of $b$ are negative rational numbers. 
\end{proposition} 

\begin{proof}
Since $\delta$ has degree $|\cS|$,  
it follows from the definition of $b$ 
that it is a polynomial of degree $\le |\cS|$. 

Suppose that $b(k)=0$ for some $k\in \CC$. Then $\beta_c(\delta,P)=0$ for any polynomial $P$.
Indeed, if there exists a $P$ such that $\beta_{c}(\delta, P)\neq 0$, then there exists such a $P$ which 
is antisymmetric of degree $|\cS|$.
Then $P$ must be a multiple of $\delta$ which contradicts the equality $\beta_{c}(\delta, \delta)=0$.

Thus, $M_c$ is reducible and hence has a singular vector, i.e. 
a nonzero homogeneous polynomial $f$ of positive degree $d$ 
living in an irreducible representation $\tau$ of $W$ killed by $y_a$. 
Applying the element $\mathbf h=\sum_i x_{a_i}y_{a_i}+r/2
+k\sum_{s\in \cS}s$ to $f$, we get 
$$
k=-\frac{d}{m_\tau},
$$
where $m_\tau$ is the eigenvalue of the operator 
$T:=\sum_{s\in \cS}(1-s)$ 
on $\tau$. But it is clear (by computing the trace of $T$) 
that $m_{\tau}\ge 0$ and $m_\tau\in \QQ$. This implies that any root of $b$ is negative rational. 
\end{proof}

Denote the Macdonald-Mehta integral by $F(k)$. 

\begin{proposition}\label{l1}
One has 
$$
F(k+1)=b(k)F(k).
$$
\end{proposition}

\begin{proof}
Let $\mathbf F=\sum_{i} y_{a_i}^2/2$.
Introduce {\em the Gaussian inner product} on $M_c$ as follows:

\begin{definition}
{\em The Gaussian inner product}
$\gamma_c$ on $M_c$
is given by the formula
$$
\gamma_c(v,v')=\beta_c(\exp(\mathbf F)v,\exp(\mathbf F)v').
$$
\end{definition}

This makes sense because the operator $\mathbf F$ is locally nilpotent on
$M_c$. 
Note that $\delta$ is a nonzero $W$-antisymmetric polynomial of the 
smallest possible degree, so $(\sum y_{a_i}^2)\delta=0$ and hence 
\begin{equation}\label{eqn-bk}
\gamma_c(\delta,\delta)=\beta_c(\delta,\delta)=b(k).
\end{equation} 

For $a\in \h$, let $x_a\in \h^*\subset H_{1,c}(W,\h)$, 
$y_a\in \h\subset H_{1,c}(W,\h)$ be the corresponding generators of the rational 
Cherednik algebra. 

\begin{proposition}\label{x-inv}
Up to scaling, $\gamma_c$
is the unique $W$-invariant symmetric bilinear form on $M_c$ satisfying
the condition
$$
\gamma_c((x_a-y_a)v,v')=\gamma_c(v,y_av'),\ a\in
\h.
$$
\end{proposition}

\begin{proof} We have
\begin{eqnarray*}
&&\gamma_c((x_a-y_a)v,v')=
\beta_c(\exp(\mathbf F)(x_a-y_a)v,\exp(\mathbf F)v')=
\beta_c(x_a\exp(\mathbf F)v,\exp(\mathbf F)v')\\
&=&\beta_c(\exp(\mathbf F)v,y_a\exp(\mathbf
F)v')=
\beta_c(\exp(\mathbf F)v,\exp(\mathbf F)y_av')=
\gamma_c(v,y_av').
\end{eqnarray*}

Let us now show uniqueness.
If $\gamma$ is any $W$-invariant
symmetric bilinear form satisfying the condition of the Proposition, then
let $\beta(v,v')=\gamma(\exp(-\mathbf F)v,\exp(-\mathbf F)v')$.
Then $\beta$ is contravariant, so
it's a multiple of $\beta_c$, hence $\gamma$ is a multiple of
$\gamma_c$.
\end{proof}

Now we will need the following known result (see \cite{Du2}, Theorem 3.10).

\begin{proposition}\label{inte}
For $\Re(k)\ge 0$ we have
\begin{equation}\label{intformu}
\gamma_c(f,g)=F(k)^{-1}\int_{\h_{\RR}}f({\mathbf x})g({\mathbf x})\d\mu_c({\mathbf x})
\end{equation}
where
$$
\d\mu_c({\mathbf x}):=\be^{-({\mathbf x},{\mathbf x})/2}|\delta({\mathbf x})|^{2k}\d {\mathbf x}.
$$
\end{proposition}

\begin{proof}
It follows from Proposition \ref{x-inv} that $\gamma_c$
is uniquely, up to scaling,
determined by the condition that it is $W$-invariant,
and $y_a^\dagger=x_a-y_a$. These properties
are easy to check for the right hand side of 
(\ref{intformu}), using the fact that
the action of $y_a$ is given by Dunkl operators.
\end{proof}

Now we can complete the proof of Proposition \ref{l1}.
By Proposition \ref{inte}, we have  
$$
F(k+1)=F(k)\gamma_c(\delta,\delta),
$$
so by \eqref{eqn-bk} we have 
$$
F(k+1)=F(k)b(k).
$$
\end{proof}

Let 
$$
b(k)=b_0\prod (k+k_i)^{n_i}.
$$
We know that $k_i>0$, and also $b_0>0$
(because the inner product $\beta_0$ on real polynomials 
is positive definite).  

\begin{corollary}\label{c2}
We have 
$$
F(k)=b_0^k\prod_i \left(\frac{\Gamma(k+k_i)}{\Gamma(k_i)}\right)^{n_i}.
$$
\end{corollary}

\begin{proof} Denote the right hand side by $F_*(k)$ 
and let $\phi(k)=F(k)/F_*(k)$. Clearly, $\phi(0)=1$.
Proposition \ref{l1} implies that $\phi(k)$ is a 1-periodic 
positive function on $[0,\infty)$. Also by the Cauchy-Schwarz inequality, 
$$
F(k)F(k')\ge F((k+k')/2)^2,
$$
so $\log F(k)$ is convex for $k\ge 0$. 
This implies that $\phi=1$, since $(\log F_*(k))''\to 0$ as $k\to +\infty$. 
\end{proof} 

\begin{remark}
The proof of this corollary is motivated by the standard proof
of the following well known characterization of the $\Gamma$ function.

\begin{proposition}
The $\Gamma$ function is determined by three 
properties:
\begin{enumerate}
\item[(i)] $\Gamma(x)$ is positive on $[1, +\infty)$ and 
$\Gamma(1)=1$;
\item[(ii)] $\Gamma(x+1)=x\Gamma(x)$;
\item[(iii)] $\log \Gamma(x)$ is a convex function on 
$[1, +\infty)$.
\end{enumerate}
\end{proposition}

\begin{proof}
It is easy to see from the definition $\Gamma(x)=\int_{0}^{\infty} t^{x-1}e^{-t}\d t$
that the $\Gamma$ function has properties (i) and (ii);
property (iii) follows from this definition and the Cauchy-Schwarz inequality.

Conversely, suppose we have a function $F(x)$ satisfying the above properties, 
then we have $F(x)=\phi(x)\Gamma(x)$ for some $1$-periodic function
$\phi(x)$ with $\phi(x)>0$.
Thus, we have 
$$(\log F)''=(\log \phi)''+(\log \Gamma)''.$$
Since $\lim_{x\to +\infty}(\log \Gamma)''=0$, $(\log F)''\geq 0$, 
and $\phi$ is periodic, we have $(\log \phi)''\geq 0$. Since 
$\int_{n}^{n+1}(\log \phi)''\d x=0$, we see that $(\log \phi)''\equiv 0$. So we have $\phi(x)\equiv 1$.
\end{proof}

\end{remark}
\medskip

In particular, we see from Corollary \ref{c2} 
and the multiplication formulas for the $\Gamma$ function 
that part (ii) of Theorem \ref{thm-MMI} implies part (i).

It remains to establish (ii). 

\begin{proposition}\label{degree}
The polynomial $b$ has degree exactly $|\cS|$. 
\end{proposition}

\begin{proof}
By Proposition \ref{c1}, $b$ is a polynomial of degree at most $|\cS|$.
To see that the degree is precisely $|\cS|$, let us 
make the change of variable 
${\mathbf x}=k^{1/2}{\mathbf y}$ in the Macdonald-Mehta integral 
and use the steepest descent method. 
We find that the leading term of the asymptotics 
of $\log F(k)$ as $k\to +\infty$ is 
$|\cS|k\log k$. This together with 
the Stirling formula and Corollary \ref{c2} 
implies the statement.    
\end{proof}

\begin{proposition}\label{l2}
The function 
$$
G(k):=F(k)\prod_{j=1}^r \frac{1-\be^{2\pi \bi kd_j}}{1-\be^{2\pi \bi k}}
$$
analytically continues to an entire function of $k$.
\end{proposition}

\begin{proof}
Let $\xi\in D$ be an element. Consider the real hyperplane 
$C_t=\bi t\xi+\h_{\RR}$, $t>0$. Then $C_t$ does not intersect 
reflection hyperplanes, so we have a continuous branch 
of $\delta({\mathbf x})^{2k}$ on $C_t$ which tends to the positive branch in $D$ as 
$t\to 0$. Then, it is easy to see that 
for any $w\in W$, the limit of this branch in 
the chamber $w(D)$ will be  
$e^{2\pi \bi k\ell(w)}|\delta({\mathbf x})|^{2k}$, where $\ell(w)$ is the length of $w$. 
Therefore, by letting $t=0$, we get 
$$
(2\pi)^{-r/2}\int_{C_t}\be^{-({\mathbf x},{\mathbf x})/2}\delta({\mathbf x})^{2k}\d {\mathbf x}=
\frac{1}{|W|}F(k)(\sum_{w\in W}e^{2\pi \bi k\ell(w)})
$$
(as this integral does not depend on $t$ by Cauchy's theorem). 
But it is well known that 
$$
\sum_{w\in W}\be^{2\pi \bi k\ell(w)}=
\prod_{j=1}^r \frac{1-\be^{2\pi \bi kd_j}}{1-\be^{2\pi \bi k}},
$$
(\cite{Hu}, p.73), so
$$ 
(2\pi)^{-r/2}|W|\int_{C_t}\be^{-({\mathbf x},{\mathbf x})/2}\delta({\mathbf x})^{2k}\d {\mathbf x}=G(k).
$$
Since  $\int_{C_t}\be^{-({\mathbf x},{\mathbf x})/2}\delta({\mathbf x})^{2k}\d {\mathbf x}$ is clearly an entire function,
the statement is proved. 

\end{proof}

\begin{corollary}\label{c3}
For every $k_0\in [-1,0]$ 
the total multiplicity of all the roots of $b$
of the form $k_0-p$, $p\in \ZZ_+$, equals
the number of ways to represent $k_0$ in the form $-m/d_i$, 
$m=1,\ldots, d_i-1$. In other words,  
the roots of $b$ are $k_{i,m}=-m/d_i-p_{i,m}$, $1\le m\le d_i-1$, 
where $p_{i,m}\in \ZZ_+$. 
\end{corollary}

\begin{proof}
We have 
$$
G(k-p)=\frac{F(k)}{b(k-1)\cdots b(k-p)}
\prod_{j=1}^r \frac{1-\be^{2\pi \bi kd_j}}{1-\be^{2\pi \bi k}},
$$
Now plug in $k=1+k_0$ and a large positive integer $p$. 
Since by Proposition \ref{l2} the left hand side is regular, so must be the right hand side, 
which implies the claimed upper bound for the total multiplicity, 
as $F(1+k_0)>0$. The fact that the bound is actually attained follows from 
the fact that the polynomial $b$ has degree exactly $|\cS|$ (Proposition \ref{degree}), and the fact that 
all roots of $b$ are negative rational (Proposition \ref{c1}). 
\end{proof}

It remains to show that in fact in Corollary \ref{c3}, $p_{i,m}=0$ for all $i,m$;
this would imply (ii) and hence (i). 

\begin{proposition}\label{k2}
Identity \eqref{eqn-MMI} of Theorem \ref{thm-MMI} 
is satisfied in $\CC[k]/k^2$. 
\end{proposition}

\begin{proof}
Indeed, we clearly have $F(0)=1$. Next, a rank $1$ computation gives
$F'(0)=-\gamma|\cS|$, where $\gamma$ is the Euler constant
(i.e. $\gamma=\lim_{n\to +\infty}(1+\cdots+1/n-\log n)$), 
while the  
derivative of the right hand side of \eqref{eqn-MMI} at zero equals to 
$$
-\gamma\sum_{i=1}^r (d_i-1).
$$
But it is well known that 
$$
\sum_{i=1}^r (d_i-1)
=|\cS|,
$$ 
(\cite{Hu}, p.62), 
which implies the result.
\end{proof}

\begin{proposition}\label{l5}Identity \eqref{eqn-MMI} of Theorem \ref{thm-MMI} is satisfied in $\CC[k]/k^3$. 
\end{proposition}

Note that Proposition \ref{l5} immediately implies (ii), and hence the whole theorem. Indeed, 
it yields that 
$$
(\log F)''(0)=\sum_{i=1}^r\sum_{m=1}^{d_i-1}(\log \Gamma)''(m/d_i),
$$
so by Corollary \ref{c3} 
$$
\sum_{i=1}^r\sum_{m=1}^{d_i-1}(\log \Gamma)''(m/d_i+p_{i,m})=
\sum_{i=1}^r\sum_{m=1}^{d_i-1}(\log \Gamma)''(m/d_i),
$$
which implies that $p_{i,m}=0$ since $(\log \Gamma)''$ is strictly decreasing on $[0,\infty)$. 

To prove Proposition \ref{l5}, we will need the following result 
about finite Coxeter groups.

Let $\psi(W)=3|\cS|^2-\sum_{i=1}^r (d_i^2-1)$. 

\begin{lemma}\label{rk2}
One has 
\begin{equation}\label{rk2e}
\psi(W)=\sum_{G\in {\mathrm Par}_2(W)}\psi(G),
\end{equation}
where ${\mathrm Par}_2(W)$ is the set of parabolic subgroups of $W$ of rank 2. 
\end{lemma}

\begin{proof}
Let 
$$
Q(q)=|W|\prod_{i=1}^r\frac{1-q}{1-q^{d_i}}.
$$  
It follows from Chevalley's theorem that
$$
Q(q)=(1-q)^r\sum_{w\in W}\det(1-qw|_\h)^{-1}.
$$
Let us subtract the terms for $w=1$ and $w\in \cS$ from both sides of this equation, 
divide both sides by $(q-1)^2$, and set $q=1$ (cf. \cite{Hu}, p.62, formula (21)). Let $W_2$ be the set of elements of $W$
that can be written as a product of two different reflections. Then by a straightforward computation 
we get
$$
\frac{1}{24}\psi(W)=\sum_{w\in W_2}\frac{1}{r-\tr_\h(w)}. 
$$
In particular, this is true for rank 2 groups. 
The result follows, as any element $w\in W_2$ belongs 
to a unique parabolic subgroup $G_w$ of rank $2$
(namely, the stabilizer of a generic point $\h^w$, 
\cite{Hu}, p.22). 
\end{proof}

\begin{proof}[Proof of Proposition \ref{l5}]
Now we are ready to prove the proposition.  
By Proposition \ref{k2}, it suffices to show the coincidence of the second derivatives of \eqref{eqn-MMI}
at $k=0$. The second derivative of the right hand side of \eqref{eqn-MMI} at zero 
is equal to 
$$
\frac{\pi^2}{6}\sum_{i=1}^r (d_i^2-1)+\gamma^2 |\cS|^2.
$$
On the other hand, we have 
$$
F''(0)=(2\pi)^{-r/2}\sum_{\alpha,\beta\in \cS}\int_{\h_{\Bbb R}} 
\be^{-({\mathbf x},{\mathbf x})/2}\log \alpha^2({\mathbf x})\log \beta^2({\mathbf x})\d {\mathbf x}.
$$
Thus, from a rank 1 computation we see that our job is to 
establish the equality 
\begin{eqnarray*}
(2\pi)^{-r/2}\sum_{\alpha\ne \beta\in \cS}\int_{\h_{\Bbb R}} \be^{-({\mathbf x},{\mathbf x})/2}\log \alpha^2({\mathbf x})\log \frac{\beta^2({\mathbf x})}{\alpha^2({\mathbf x})}\d {\mathbf x}=\frac{\pi^2}{6}(\sum_{i=1}^r (d_i^2-1)-3|\cS|^2)=-\frac{\pi^2}{6}\psi(W). 
\end{eqnarray*}
Since this equality holds in rank 2 (as in this case \eqref{eqn-MMI} reduces to the beta integral),  
in general it reduces to equation (\ref{rk2e}) (as for any $\alpha\ne \beta\in S$,
$s_\alpha$ and $s_\beta$ are contained in 
a unique parabolic subgroup of $W$ of rank 2). 
The proposition is proved. 
\end{proof}

\subsection{Application: the supports of $L_{c}(\CC)$}

In this subsection we will use the Macdonald-Mehta integral
to computation of the support of the irreducible 
quotient of the polynoamial representation 
of a rational Cherednik algebra (with equal parameters). 
We will follow the paper \cite{E3}.

First note that the vector space $\h$ has a stratification labeled 
by parabolic subgroups of $W$. Indeed, for a parabolic 
subgroup $W'\subset W$, let $\h_{\reg}^{W'}$
be the set of points in $\h$ whose stabilizer is $W'$.
Then 
$$\h=\coprod_{W'\in \mathrm{Par}(W)}\h_{\reg}^{W'},$$
where
$\mathrm{Par}(W)$ is the set of parabolic subgroups in $W$.

For a finitely generated module $M$ over $\CC[\h]$, denote the support of $M$ by 
$\supp(M)$. 

The following theorem is proved in \cite{Gi1},
Section 6 and in \cite{BE} with different method. We will 
recall the proof from \cite{BE} later.

\begin{theorem} 
Consider the stratification of $\h$ with respect to stabilizers
of points in $W$. Then the support $\supp(M)$ of any object
$M$ of $\cO_{c}(W,\h)$ in $\h$ is a union of strata of this stratification.  
\end{theorem}

This makes one wonder which strata occur in $\supp(L_{c}(\tau))$, for given $c$ and $\tau$. 
In \cite{VV}, Varagnolo and Vasserot gave a partial answer for $\tau=\CC$.
Namely, they determined (for $W$ being a Weyl group) when $L_{c}(\CC)$ is finite dimensional, which is 
equivalent to $\supp(L_{c}(\CC))=0$. For the proof (which is quite complicated), they used the geometry affine Springer
fibers. Here we will give a different (and simpler) proof. In fact, we will prove a more general 
result.

Recall that for any Coxeter group $W$, we have the Poincar\'e polynomial:
$$
P_{W}(q)=\sum_{w\in W}q^{\ell(w)}=\prod_{i=1}^{r}\frac{1-q^{d_{i}(W)}}{1-q}, \text{ where } d_i(W) \text{ are the degrees of } W.
$$ 

\begin{lemma}\label{divisi}
If $W'\subset W$ is a parabolic subgroup of $W$, then $P_{W}$ is
divisible by $P_{W'}$.
\end{lemma}

\begin{proof}
By Chevalley's theorem, $\CC[\h]$ is a free module over 
$\CC[\h]^{W}$ and $\CC[\h]^{W'}$ is a direct summand 
in this module. 
So $\CC[\h]^{W'}$ is a projective module, thus free (since it is graded).

Hence, there exists a polynomial $Q(q)$ 
such that we have
$$Q(q)h_{\CC[\h]^{W}}(q)=h_{\CC[\h]^{W'}}(q),
$$
where $h_V(q)$ denotes the Hilbert series of a graded vector space $V$. 
Notice that we have $h_{\CC[\h]^{W}}(q)=\dfrac{1}{P_{W}(q)(1-q)^{r}}$, so we have
$$\frac{Q(q)}{P_{W}(q)}=\frac{1}{P_{W'}(q)},
\text{ i.e. }Q(q)=P_{W}(q)/P_{W'}(q).$$
\end{proof}

\begin{corollary} If $m\ge 2$ then
we have the following inequality:
$$
\#\{i| m \text{ divides }d_{i}(W)\}\geq
\#\{i| m \text{ divides }d_{i}(W')\}.
$$
\end{corollary}

\begin{proof} This follows from Lemma \ref{divisi} 
by looking at the roots of the polynomials $P_W$ and $P_{W'}$. 
\end{proof} 

Our main result is the following theorem.

\begin{theorem}{\label{thm-supp}}\cite{E3}
Let $c\geq 0$. Then $a\in \supp(L_{c}(\CC))$ if and only if
$$\dfrac{P_{W}}{P_{W_{a}}}(\be^{2\pi \bi c})\neq 0.$$
\end{theorem}

We can obtain the following corollary easily.
\begin{corollary}
\begin{enumerate}
\item[(i)]$L_{c}(\CC)\neq M_{c}(\CC)$ if and only if 
$c\in \QQ_{>0}$ and the denominator $m$ of $c$ divides $d_{i}$
for some $i$;
\item[(ii)]$L_{c}(\CC)$ is finite dimensional if and only if 
$\dfrac{P_{W}}{P_{W'}}(\be^{2\pi \bi c})=0$,
i.e., iff
$$
\#\{i| m \text{ divides }d_{i}(W)\}>
\#\{i| m \text{ divides }d_{i}(W')\}.
$$
for any maximal parabolic subgroup $W'\subset W$.
\end{enumerate}

\end{corollary}

\begin{remark}
Varagnolo and Vasserot prove that 
$L_c(\CC)$ is finite dimensional if and only 
if there exists a regular elliptic element 
in $W$ of order $m$. Case-by-case inspection 
shows that this condition is equivalent to the combinatorial
condition of (2). Also, a uniform proof of this equivalence is 
given in the appendix to \cite{E3}, written by S. Griffeth.     
\end{remark}

\begin{example}
For type $A_{n-1}$, i.e., $W=\kS_{n}$, 
we get that $L_{c}(\CC)$ is finite dimensional if and only if 
the denominator of $c$ is $n$. This agrees with our previous results in type $A_{n-1}$. 
\end{example}

\begin{example}
Suppose $W$ is the Coxeter group of type $E_{7}$.
Then we have the following list of maximal parabolic subgroups 
and the degrees (note that $E_7$ itself is not a maximal parabolic).

\vspace{.25in}

{\small
\begin{tabular}{|c|c|c|c|c|}\hline Subgroups & $E_7$ & $D_6$ & $A_3\times A_2\times A_1$ & $A_6$ \\\hline Degrees  & 2,6,8,10,12,14,18 & 2,4,6,6,8,10 & 2,3,4,2,3,2 & 2,3,4,5,6,7 \\\hline\hline Subgroups & $A_4\times A_2$ & $E_6$ & $D_5\times A_1$ & $A_5\times A_1$ \\\hline Degrees  & 2,3,4,5,2,3 & 2,5,6,8,9,12 & 2,4,5,6,8,2 & 2,3,4,5,6,2 \\\hline \end{tabular}
}

\vspace{.25in}

So $L_{c}(\CC)$ is finite dimensional if and only if 
the denominator of $c$ is $2, 6, 14, 18$.
\end{example}

The rest of the subsection is dedicated to the proof 
of Theorem \ref{thm-supp}. First we recall some basic facts about 
the Schwartz space and tempered distributions.

Let $\rS(\RR^{n})$ be the set of Schwartz functions on 
$\RR^{n}$, i.e. 
$$\rS(\RR^{n})=\{f\in C^{\infty}(\RR^{n})|
\forall \alpha, \beta, \sup|{\mathbf x}^{\alpha}\partial^{\beta}f({\mathbf x})|<\infty\}.
$$
This space has a natural topology.

A tempered distribution on $\RR^{n}$ is a continuous linear functional
on $\rS(\RR^{n})$. Let $\rS'(\RR^{n})$ denote the 
space of tempered distributions.

We will use the following well known lemma.

\begin{lemma}{\label{lem-dist}}
\begin{enumerate}
\item[(i)] $\CC[{\mathbf x}]\be^{-{\mathbf x}^{2}/2}\subset \rS(\RR^{n})$
is a dense subspace.

\item[(ii)] Any tempered distribution $\xi$ has finite order, i.e., $\exists N=N(\xi)$ 
such that if $f\in \rS(\RR^{n})$ satisfying 
$f=\d f=\cdots=\d^{N-1}f=0$ on $\supp \xi$, then $\langle \xi, f\rangle=0$.
\end{enumerate}
\end{lemma}

\begin{proof}[Proof of Theorem \ref{thm-supp}]
Recall that on $M_{c}(\CC)$, we have the Gaussian form 
$\gamma_{c}$ from Section \ref{sec:pfMMI}.
We have for $\Re(c) \leq 0 $, 
$$\gamma_{c}(P, Q)=\frac{(2\pi)^{-r/2}}{F_W(-c)}\int_{\h_{\RR}}
\be^{-{\mathbf x}^{2}/2}|\delta({\mathbf x})|^{-2c}P({\mathbf x})Q({\mathbf x})\d {\mathbf x}, $$
where 
$P, Q$ are polynomials and 
$$F_W(k)=(2\pi)^{-r/2}\int_{\h_{\RR}}
\be^{-{\mathbf x}^{2}/2}|\delta({\mathbf x})|^{2k}\d {\mathbf x}
$$
is the Macdonald-Mehta integral. 

Consider the distribution: 
$$\xi_{c}^{W}=\frac{(2\pi)^{-r/2}}{F_W(-c)}|\delta({\mathbf x})|^{-2c}.
$$
It is well-known that this distribution is meromorphic in $c$ 
(Bernstein's theorem). Moreover, since 
$\gamma_c(P,Q)$ is a polynomial in $c$ for any $P$ and $Q$,  
this distribution is in fact holomorphic in $c\in \CC$.

\begin{proposition}{\label{prop-1}}
\begin{eqnarray*}
\supp(\xi_{c}^{W})&=&
\{a\in \h_{\RR}|\frac{F_{W_{a}}}{F_{W}}(-c)\neq 0\}
=\{a\in \h_{\RR}|\frac{P_{W}}{P_{W_{a}}}(e^{2\pi\mathrm{i}c})\neq 0\}\\
&=&\{a\in \h_{\RR}|\#\{i|\text{denominator of $c$ divides } d_{i}(W)\}\\
&&\qquad\qquad=\#\{i|\text{denominator of $c$ divides } d_{i}(W_{a})\}\}.
\end{eqnarray*}
\end{proposition}

\begin{proof}
First note that the last equality follows from the product formula for the Poincar\'e polynomial, and 
the second equality from the Macdonald-Mehta identity. Now let us prove the first equality. 

Look at $\xi_{c}^{W}$ near $a\in \h$. Equivalently, 
we can consider 
$$\xi^{W}_{c}({\mathbf x}+a)=\frac{(2\pi)^{-r/2}}{F_W(-c)}|\delta({\mathbf x}+a)|^{-2c}$$
with ${\mathbf x}$ near $0$.
We have
\begin{eqnarray*}
\delta_{W}({\mathbf x}+a)&=&\prod_{s\in \cS}\alpha_{s}({\mathbf x}+a)
=\prod_{s\in \cS}(\alpha_{s}({\mathbf x})+\alpha_{s}(a))\\
&=&\prod_{s\in \cS\cap W_{a}}\alpha_{s}({\mathbf x})\cdot
\prod_{s\in \cS\backslash \cS\cap W_{a}}(\alpha_{s}({\mathbf x})+\alpha_{s}(a))\\
&=&\delta_{W_{a}}({\mathbf x})\cdot \Psi({\mathbf x}),
\end{eqnarray*}
where $\Psi$ is a nonvanishing function near $a$
(since $\alpha_s(a)\ne 0$ if $s\notin \cS\cap W_a$).

So near $a$, we have
$$
\xi_{c}^{W}({\mathbf x}+a)=\frac{F_{W_{a}}}{F_{W}}(-c)\cdot \xi_{c}^{W_{a}}({\mathbf x})\cdot 
|\Psi|^{-2c},
$$
and the last factor is well defined since $\Psi$ is nonvanishing.
Thus $\xi_{c}^{W}({\mathbf x})$ is nonzero near $a$ if and only if
$\dfrac{F_{W_{a}}}{F_{W}}(-c)\neq 0$
which finishes the proof.
\end{proof}

\begin{proposition}{\label{prop-2}} For $c\ge 0$,
$$\supp(\xi_{c}^{W})=\supp L_{c}(\CC)_{\RR},$$
where the right hand side stands for the real points of the support. 
\end{proposition}

\begin{proof}
Let $a\notin \supp L_{c}(\CC)$ and assume $a\in \supp \xi^{W}_{c}$.
Then we can find a $P\in J_{c}(\CC)=\ker \gamma_{c}$ such that 
$P(a)\neq 0$.
Pick a compactly supported test function $\phi\in C_{c}^{\infty}(\h_{\RR})$ such that 
$P$ does not vanish anywhere on $\supp \phi$, and $\langle\xi^{W}_{c},\phi\rangle\neq 0$
(this can be done since $P(a)\ne 0$ and $\xi^W_c$ is nonzero near $a$). Then we have
$\phi/P\in \rS(\h_{\RR})$.
Thus from Lemma \ref{lem-dist} (i) it follows that 
there exists a sequence of polynomials $P_{n}$ such that
$$P_{n}({\mathbf x})\be^{-{\mathbf x}^{2}/2}\to \frac{\phi}{P} \text{ in }
\mathscr{S}(\h_{\RR})\text{, when }n\to \infty.$$
So $PP_{n}\be^{-{\mathbf x}^{2}/2}\to \phi\text{ in }
\rS(\h_{\RR})\text{, when }n\to \infty.$

But we have 
$\langle \xi^{W}_{c},PP_{n}\be^{-{\mathbf x}^{2}/2}\rangle=\gamma_{c}(P, P_{n})=0$
which is a contradiction.
This implies that $\supp \xi_c^W\subset (\supp L_c(\CC))_{\RR}$. 

To show the opposite inclusion, 
let $P$ be a polynomial on $\h$ which vanishes identically on 
$\supp \xi_c^W$. By Lemma \ref{lem-dist} (ii), there exists $N$ such that 
$\langle \xi_c^W, P^N({\mathbf x})Q({\mathbf x})\be^{-{\mathbf x}^2/2}\rangle =0$. 
Thus, for any polynomial $Q$, $\gamma_c(P^N,Q)=0$, i.e. $P^N\in \Ker\gamma_c$. 
Thus, $P|_{\supp L_c(\CC)}=0$. This implies the required inclusion, since
$\supp \xi_c^W$ is a union of strata.  
\end{proof}

Theorem \ref{thm-supp} follows from Proposition \ref{prop-1} and Proposition \ref{prop-2}.
\end{proof}

\subsection{Notes}
Our exposition in Sections 4.1 and 4.2 follows the paper \cite{E2};
Section 4.3 follows the paper \cite{E3}. 

\newpage \section{Parabolic induction and restriction functors for rational Cherednik
algebras}

\subsection{A geometric approach to rational Cherednik algebras}\label{sec:georca}

An important property of the rational Cherednik algebra 
$H_{1,c}(G,\h)$ is that it can be sheafified, as an algebra, over $\h/G$ (see \cite{E1}).
More specifically, the usual sheafification of $H_{1,c}(G,\h)$ 
as a $\cO_{\h/G}$-module is in fact a quasicoherent sheaf
of algebras, $H_{1,c,G,\h}$. Namely, for
every affine open subset $U\subset \h/G$, the algebra of sections
$H_{1,c,G,\h}(U)$ is, by definition, 
$\CC[U]\otimes_{\CC[\h]^G}H_{1,c}(G,\h)$. 

The same sheaf can be 
defined more geometrically as follows (see \cite{E1}, Section
2.9). Let $\widetilde U$ be the preimage of
$U$ in $\h$. Then the algebra $H_{1,c,G,\h}(U)$ is the algebra of
linear operators on $\cO(\widetilde U)$ generated by 
$\cO(\widetilde U)$, the group $G$, and Dunkl
operators
$$
\partial_a-\sum_{s\in \cS}\frac{2c_s}{1-\lambda_s}
\frac{\alpha_s(a)}{\alpha_s}(1-s), \text{ where } a\in \h. 
$$

\subsection{Completion of rational Cherednik algebras}

For any $b\in \h$ we can define the completion
$\widehat{H_{1,c}}(G,\h)_b$ to be the algebra of sections  
of the sheaf $H_{1,c,G,\h}$ on the formal neighborhood of the image
of $b$ in $\h/G$. Namely, $\widehat{H_{1,c}}(G,\h)_b$ 
is generated by regular functions on the formal neighborhood of
the $G$-orbit of $b$, the group $G$, and Dunkl operators.  

The algebra $\widehat{H_{1,c}}(G,\h)_b$ inherits from $H_{1,c}(G,\h)$ the natural
filtration $F^\bullet$ by order of differential operators,
and each of the spaces $F^n\widehat{H_{1, c}}(G,\h)_b$ has a projective limit
topology; the whole algebra is then equipped with the topology of
the nested union (or inductive limit).  

Consider the completion of the rational Cherednik algebra at
zero, $\widehat{H_{1,c}}(G, \h)_0$. It naturally contains the algebra
$\CC[[\h]]$. Define the category $\widehat\cO _c(G, \h)$
of representations of $\widehat{H_{1,c}}(G, \h)_0$ which are 
finitely generated over $\CC[[\h]]_{0}=\CC[[\h]]$. 

We have a completion functor $\,\widehat{}: \cO _c(G, \h)
\to\widehat\cO _c(G, \h)$, defined by 
$$
\widehat{M}=\widehat{H_{1,c}}(G, \h)_0\otimes_{H_{1, c}(G, \h)}M=
\CC[[\h]]\otimes_{\CC[\h]}M.
$$ 

Also, for $N\in \widehat\cO _c(G, \h)$, let $E(N)$ be the subspace
spanned by generalized eigenvectors of $\mathbf h$ in $N$
where $\mathbf h$ is defined by \eqref{eqn:h}. Then
it is easy to see that $E(N)\in \mathcal O_c(G, \h)_0$. 

\begin{theorem}\label{equi} 
The restriction of the completion functor $\,\,\widehat{}\,\,$ 
to $\cO _c(G, \h)_0$ is an equivalence 
of categories $\cO _c(G, \h)_0\to \widehat{\mathcal
O}_c(G, \h)$. The inverse equivalence is given by the functor
$E$. 
\end{theorem} 

\begin{proof} 
It is clear that $M\subset \widehat{M}$, 
so $M\subset E(\widehat{M})$ (as $M$ is spanned by generalized 
eigenvectors of $\mathbf h$). Let us demonstrate the opposite
inclusion. Pick generators $m_1,\ldots,m_r$ of $M$ which are
generalized eigenvectors of $\mathbf h$ with eigenvalues 
$\mu_1,\ldots,\mu_r$. Let $0\ne v\in E(\widehat{M})$.
Then $v=\sum_i f_im_i$, where $f_i\in \CC[[\h]]$. 
Assume that $(\mathbf h-\mu)^Nv=0$ for some $N$. Then 
$v=\sum_i f_i^{(\mu-\mu_i)}m_i$, where for $f\in \CC[[\h]]$ we
denote by $f^{(d)}$ the degree $d$ part of $f$. 
Thus $v\in M$, so $M=E(\widehat{M})$. 

It remains to show that $\widehat{E(N)}=N$, i.e. that 
$N$ is the closure of $E(N)$. In other words, 
letting $\mathfrak{m}$ denote the maximal ideal 
in $\CC[[\h]]$, we need to show that the natural map 
$E(N)\to N/\mathfrak{m}^jN$ is surjective for every $j$. 

To do so, note that $\mathbf h$ preserves the descending filtration of $N$
by subspaces $\mathfrak{m}^jN$. On the other hand, the successive
quotients of these subspaces, $\mathfrak{m}^jN/\mathfrak{m}^{j+1}N$,
are finite dimensional, which implies that $\mathbf h$ acts locally
finitely on their direct sum $\gr N$, and moreover each
generalized eigenspace is finite dimensional. Now for each 
$\beta\in \CC$ denote by $N_{j,\beta}$ the generalized
$\beta$-eigenspace of $\mathbf h$ in $N/{\mathfrak m}^jN$. 
We have surjective homomorphisms $N_{j+1,\beta}\to N_{j,\beta}$, 
and for large enough $j$ they are isomorphisms. This implies that
the map $E(N)\to N/\mathfrak{m}^jN$ is surjective for every $j$,
as desired. 
\end{proof}

{\bf Example.} Suppose that $c=0$. Then Theorem \ref{equi} 
specializes to the well known fact that the category of 
$G$-equivariant local systems on $\h$ with a locally nilpotent
action of partial differentiations is equivalent to the category 
of all $G$-equivariant local systems on the formal neighborhood
of zero in $\h$. In fact, both categories in this case are
equivalent to the category of finite dimensional representations
of $G$.

We can now define the composition functor
${\mathcal J}: \cO _c(G, \h)\to \cO _c(G, \h)_0$, 
by the formula ${\mathcal J}(M)=E(\widehat{M})$.  
The functor ${\mathcal J}$ is called the Jacquet functor
(\cite{Gi2}). 

\subsection{The duality functor}{\label{sec:duali}}

Recall that in Section \ref{sec:contmod}, 
for any $H_{1, c}(G, \h)$-module $M$,
the full dual space $M^*$ 
is naturally an $H_{1, \bar c}(G,\h^*)$-module, via 
$\pi_{M^*}(a)=\pi_M(\gamma(a))^*$. 

It is clear that the duality functor $*$ defines an equivalence 
between the category $\cO _c(G, \h)_0$ and 
$\widehat\cO _{\bar c}(G,\h^*)^{\op}$, and that 
$M^\dagger=E(M^*)$ (where $M^\dagger$ is the contragredient, or
restricted dual module to $M$ defined in Section \ref{sec:contmod}).

\subsection{Generalized Jacquet functors} 

\begin{proposition}\label{nilpo} For any $M\in \widehat{\mathcal
O}_c(G, \h)$, a vector $v\in M$ is ${\mathbf h}$-finite 
if and only if it is $\h$-nilpotent.
\end{proposition}

\begin{proof} The ``if'' part 
follows from Theorem \ref{finnilp}. 
To prove the ``only if'' part, assume that $({\mathbf
h}-\mu)^Nv=0$. Then for any $u\in S^r\h\cdot v$, we
have $({\mathbf h}-\mu+r)^Nu=0$. But by Theorem \ref{equi},
the real parts of generalized eigenvalues of ${\mathbf h}$ in $M$
are bounded below. Hence $S^r\h\cdot v=0$ for large enough
$r$, as desired. 
\end{proof} 

According to Proposition \ref{nilpo}, the functor $E$
can be alternatively defined by setting $E(M)$ to be the subspace of
$M$ which is locally nilpotent under the action of $\h$. 

This gives rise to the following generalization of $E$: 
for any $\lambda\in \h^*$ we define the functor 
$E_\lambda: \widehat\cO _c(G, \h)\to 
\cO_c(G, \h)_\lambda$ by setting $E_\lambda(M)$ to be the space of 
generalized eigenvectors of $\CC[\h^*]^G$ in $M$ with eigenvalue
$\lambda$. This way, we have $E_0=E$. 

We can also define the generalized Jacquet functor 
${\mathcal J}_\lambda: \cO _c(G, \h)\to {\mathcal
O}_c(G, \h)_\lambda$ by the formula 
${\mathcal J}_\lambda(M)=E_\lambda({\widehat M})$. 
Then we have ${\mathcal J}_0={\mathcal J}$, and one can show that
the restriction of ${\mathcal J}_\lambda$ to ${\mathcal
O}_c(G, \h)_\lambda$ is the identity functor. 

\subsection{The centralizer construction}

For a finite group $H$, let $\e_H=|H|^{-1}\sum_{g\in H}g$ 
be the symmetrizer of $H$. 

If $G\supset H$ are finite groups, 
and $A$ is an algebra containing $\CC[H]$, 
then define the algebra 
$Z(G,H,A)$ to be the centralizer $\End_A(P)$ of $A$ in the right $A$-module
$P={\mathrm Fun}_H(G,A)$ of $H$-invariant $A$-valued functions on $G$,
i.e. such functions $f: G\to A$ that $f(hg)=hf(g)$. Clearly, 
$P$ is a free $A$-module of rank $|G/H|$, so the algebra 
$Z(G,H,A)$ is isomorphic to ${\mathrm Mat}_{|G/H|}(A)$, but this
isomorphism is not canonical. 

The following lemma is trivial. 

\begin{lemma}\label{le}
The functor $N\mapsto I(N):=P\otimes_A N={\mathrm Fun}_H(G,N)$ defines 
an equivalence of categories 
$A-{\mathrm mod}\to Z(G,H,A)-{\mathrm mod}$. 
\end{lemma} 

\subsection{Completion of rational Cherednik algebras at
arbitrary points of $\h/G$}

The following result is, in essence, a consequence 
of the geometric approach to rational Cherednik algebras,
described in Subsection \ref{sec:georca}. 
It should be regarded as a direct generalization to the 
case of Cherednik algebras of Theorem 8.6 of \cite{L} 
for affine Hecke algebras.

Let $b\in \h$. Abusing notation, denote the restriction of $c$ to 
the set $\cS_b$ of reflections in $G_b$ also by $c$. 

\begin{theorem}\label{comp}
One has a natural isomorphism 
$$
\theta: \widehat{H_{1,c}}(G, \h)_b\to Z(G, G_b,\widehat{H_{1,c}}(G_b,\h)_0),
$$
defined by the following formulas. 
Suppose that $f\in P={\mathrm Fun}_{G_b}(G, \widehat{H_{1,c}}(G_b,\h)_0)$.
Then 
$$
(\theta(u)f)(w)=f(wu), u\in G;
$$
for any $\alpha\in \h^*$, 
$$
(\theta(x_\alpha)f)(w)=(x_{w\alpha}^{(b)}+(w\alpha,b))f(w), 
$$
where $x_\alpha\in \h^*\subset H_{1,c}(G, \h)$, 
$x_\alpha^{(b)}\in \h^*\subset H_{1,c}(G_b,\h)$ are the elements
corresponding to $\alpha$; and for any 
$a\in \h$, 
\begin{equation}\label{eqn:thetay}
(\theta(y_a)f)(w)=y_{wa}^{(b)}f(w)-\sum_{s\in \cS: s\notin G_b}
\frac{2c_s}{1-\lambda_s}\frac{\alpha_s(wa)}{x_{\alpha_s}^{(b)}+\alpha_s(b)}
(f(w)-f(sw)).
\end{equation}
where $y_a\in \h\subset H_{1,c}(G, \h)$, 
$y_a^{(b)}\in \h\subset H_{1,c}(G_b,\h)$. 
\end{theorem}

\begin{proof}
The proof is by a direct computation. 
We note that in the last formula, the fraction 
$\alpha_s(wa)/(x_{\alpha_s}^{(b)}+\alpha_s(b))$
is viewed as a power series (i.e., an element of $\CC[[\h]]$),
and that only the entire sum, and not each summand separately, 
is in the centralizer algebra.  
\end{proof} 

{\bf Remark.} Let us explain how to see the existence of $\theta$
without writing explicit formulas, and how to guess the formula 
\eqref{eqn:thetay} for $\theta$. 
It is explained in \cite{E1} (see e.g. \cite{E1}, Section 2.9) 
that the sheaf of algebras obtained
by sheafification of $H_{1,c}(G, \h)$ over $\h/G$ is generated (on
every affine open set in $\h/G$) by regular functions on $\h$,
elements of $G$, and Dunkl operators. Therefore, this
statement holds for formal neighborhoods, i.e., it
is true on the formal neighborhood of the image in $\h/G$ 
of any point $b\in \h$. However, looking at the formula for Dunkl
operators near $b$, we see that the summands corresponding to
$s\in \cS,s\notin G_b$ are actually regular at $b$, so they can be
safely deleted without changing the generated algebra (as all regular functions
on the formal neighborhood of $b$ are included into the system of
generators). But after these terms are deleted, what remains is
nothing but the Dunkl operators for $(G_b,\h)$, which, together
with functions on the formal neighborhood of $b$ and the group
$G_b$, generate the completion of $H_{1,c}(G_b,\h)$. This gives a 
construction of $\theta$ without using explicit formulas. 

Also, this argument explains why $\theta$ should be defined by
formula (\ref{eqn:thetay}) 
of Theorem \ref{comp}. Indeed, what this formula does
is just restores the terms with $s\notin G_b$ that have been
previously deleted.

The map $\theta$ defines an equivalence of categories
$$
\theta_{*}: \widehat{H_{1,c}}(G, \h)_b-{\mathrm mod}\to
Z(G, G_b, \widehat{H_{1, c}}(G_b,\h)_0)-{\mathrm mod}.
$$

\begin{corollary}\label{whit}
We have a natural equivalence of categories 
$$\psi_\lambda: \cO_c(G, \h)_\lambda\to 
\cO_{c}(G_\lambda,\h/\h^{G_\lambda})_0.$$ 
\end{corollary}

\begin{proof} The category $\cO_c(G, \h)_\lambda$
is the category of modules over $H_{1,c}(G, \h)$ 
which are finitely generated over $\CC[\h]$ and extend by
continuity to the completion of the algebra $H_{1,c}(G, \h)$ at $\lambda$. So
it follows from Theorem \ref{comp}
that we have an equivalence $\cO_c(G, \h)_\lambda\to 
\cO _{c}(G_\lambda,\h)_0$. Composing this equivalence
with the equivalence 
$\zeta: \cO_{c}(G_\lambda,\h)_0\to 
\cO_{c}(G_\lambda,\h/\h^{G_\lambda})_0$, we obtain
the desired equivalence $\psi_\lambda$.    
\end{proof} 

\begin{remark}
Note that in this proof, we take the completion of $H_{1,c}(G, \h)$ 
at a point of $\lambda\in \h^*$ rather than $b\in \h$.  
\end{remark}

\subsection{The completion functor}

Let $\widehat\cO_c(G, \h)^b$ be the category 
of modules over $\widehat{H_{1,c}}(G, \h)_b$ 
which are finitely generated over $\widehat{\CC[\h]}_b$. 

\begin{proposition}\label{isom}
The duality functor $*$ defines an anti-equivalence of categories 
$\cO_c(G, \h)_\lambda\to \widehat\cO_{\bar c}(G,\h^*)^\lambda$.  
\end{proposition}

\begin{proof}
This follows from the fact (already mentioned above) 
that $\cO_{c}(G, \h)_\lambda$ is the category of 
modules over $H_{1, c}(G, \h)$ 
which are finitely generated over $\CC[\h]$ and extend by
continuity to the completion of the algebra $H_{1, c}(G, \h)$ at
$\lambda$. 
\end{proof} 

Let us denote the functor inverse to $*$ also by $*$; it is the 
functor of continuous dual (in the formal series topology). 

We have an exact functor of completion at $b$, 
$\cO_c(G, \h)_0\to \widehat\cO_c(G, \h)^b$,
$M\mapsto \widehat{M}_b$. We also have a functor 
$E^b: \widehat\cO_c(G, \h)^b\to \cO_c(G, \h)_0$
in the opposite direction, sending a module $N$ to 
the space $E^b(N)$ of $\h$-nilpotent vectors in $N$. 

\begin{proposition}\label{adj1}
The functor $E^b$ is right adjoint to the completion functor
$\,\widehat{\quad}_b$. 
\end{proposition}

\begin{proof}
We have
$$
\Hom_{\widehat{H_{1,c}}(G, \h)_b}(\widehat{M}_b, N)
=\Hom_{\widehat{H_{1,c}}(G, \h)_b}(\widehat{H_{1,c}}(G, \h)_b
\otimes_{H_{1,c}(G, \h)}M, N)
$$
$$
=\Hom_{H_{1,c}(G, \h)}(M, N|_{H_{1,c}(G, \h)})
=\Hom_{H_{1,c}(G, \h)}(M, E^{b}(N)).
$$
\end{proof} 

\begin{remark}
Recall that by Theorem \ref{equi}, if $b=0$ then these functors 
are not only adjoint but also inverse to each other. 
\end{remark} 

\begin{proposition}\label{dua}
\begin{enumerate}
\item[(i)] For $M\in \cO_{\bar c}(G,\h^*)_b$, one has
$E^b(M^*)=(\widehat{M})^*$ in $\cO_c(G, \h)_0$.
\item[(ii)] For $M\in \cO_c(G, \h)_0$, 
$(\widehat{M}_b)^*=E_b(M^*)$ in $\cO_\bc(G,\h^*)_b$.
\item[(iii)] The functors $E_b$, $E^b$ are exact.
\end{enumerate}

\end{proposition}

\begin{proof}
(i),(ii) are straightforward from the definitions. 
(iii) follows from (i),(ii), since the completion functors are
exact.  
\end{proof}

\subsection{Parabolic induction and restriction functors for rational Cherednik
algebras} 

Theorem \ref{comp} allows us to define analogs of parabolic
restriction functors for rational Cherednik algebras. 

Namely, let $b\in\h$, and $G_b=G'$. 
Define a functor $\Res_b: \cO _c(G, \h)_0\to
\cO _{c}(G',\h/\h^{G'})_0$ by the formula
$$
\Res_b(M)=(\zeta\circ E\circ I^{-1}\circ
\theta_*)(\widehat{M}_b).
$$  

We can also define the parabolic induction functors
in the opposite direction. Namely, 
let $N\in \cO _{c}(G',\h/\h^{G'})_0$. 
Then we can define the object $\Ind_b(N)\in {\mathcal
O}_c(G, \h)_0$ by the formula 
$$
\Ind_b(N)=(E^b\circ \theta_*^{-1}\circ
I)(\widehat{\zeta^{-1}(N)}_0).
$$

\begin{proposition}\label{propres} 
\begin{enumerate}
\item[(i)] The functors $\Ind_b$, $\Res_b$ are exact.
\item[(ii)] One has $\Ind_b(\Res_b(M))=E^b(\widehat{M}_b)$.
\end{enumerate}
\end{proposition} 

\begin{proof} Part (i) follows from the fact that the
functor $E^b$ and the completion functor $\,\widehat{\quad}_b$ 
are exact (see Proposition \ref{dua}).
Part (ii) is straightforward from the definition.  
\end{proof}

\begin{theorem}\label{adj2}
The functor $\Ind_b$ is right adjoint to 
$\Res_b$.  
\end{theorem} 

\begin{proof} We have 
\begin{eqnarray*}
&&\Hom(\Res_b(M),N)=
\Hom((\zeta\circ E\circ I^{-1}\circ \theta_*)(\widehat{M}_b),N)=
\Hom((E\circ I^{-1}\circ \theta_*)(\widehat{M}_b),\zeta^{-1}(N))\\
&=&
\Hom((I^{-1}\circ \theta_*)(\widehat{M}_b),\widehat{\zeta^{-1}(N)}_0)=
\Hom(\widehat{M}_b,(\theta_*^{-1}\circ
I)(\widehat{\zeta^{-1}(N)}_0))\\
&=&\Hom(M,(E^b\circ \theta_*^{-1}\circ
I)(\widehat{\zeta^{-1}(N)}_0))=\Hom(M,\Ind_b(N)).
\end{eqnarray*}
At the end we used Proposition \ref{adj1}.  
\end{proof} 
Then we can obtain the following corollary easily.
\begin{corollary}\label{proj} The functor $\Res_b$ maps projective objects 
to projective ones, and the functor $\Ind_b$ maps injective objects to
injective ones. 
\end{corollary} 

We can also define functors
$\res_\lambda: \cO _c(G, \h)_0\to 
\cO _{c}(G',\h/\h^{G'})_0$ and $\ind_\lambda: 
\cO _{c}(G',\h/\h^{G'})_0\to 
\cO _c(G, \h)_0$, attached to $\lambda\in
\h^{*G'}_{\reg}$, by 
$$
\res_\lambda:= \dagger\circ \Res_\lambda\circ \dagger, 
\ind_\lambda:= \dagger\circ \Ind_\lambda\circ \dagger, 
$$
where $\dagger$ is as in Subsection \ref{sec:duali}.

\begin{corollary}\label{proj1} 
The functors $\res_\lambda$, 
$\ind_\lambda$ are exact. 
The functor $\ind_\lambda$ is left adjoint to $\res_\lambda$. 
The functor $\ind_\lambda$ maps projective objects 
to projective ones, and the functor $\res_\lambda$ injective objects to
injective ones. 
\end{corollary} 

\begin{proof}
Easy to see from the definition of the functors and the Theorem \ref{adj2}.
\end{proof}

We also have the following proposition, whose proof is
straightforward. 

\begin{proposition}\label{Jacc}
We have 
$$
\ind_\lambda(N)=({\mathcal J}\circ \psi_\lambda^{-1})(N),\quad\text{
and }\quad
\res_\lambda(M)=(\psi_\lambda\circ E_\lambda)(\widehat{M}),
$$
where $\psi_\lambda$ is defined in Corollary \ref{whit}.
\end{proposition}

\subsection{Some evaluations of the parabolic induction and 
restriction functors}

For generic $c$, the category $\cO_c(G, \h)$ is
semisimple, and naturally equivalent to the category $\Rep G$ 
of finite dimensional representations of $G$, via the functor $\tau\mapsto
M_c(G,\h,\tau)$. (If $G$ is a Coxeter group, the exact set of
such $c$ (which are called regular) is known from 
\cite{GGOR} and \cite{Gy}).   

\begin{proposition}\label{compuRI} 
\begin{enumerate}
\item[(i)] Suppose that $c$ is generic. 
Upon the above identification, the functors
$\Ind_b$, $\ind_\lambda$ and $\Res_b$, $\res_\lambda$ 
go to the usual induction and restriction
functors between categories $\Rep G$ and $\Rep G'$.  
In other words, we have 
$$
\Res_b(M_c(G,\h,\tau))=\oplus_{\xi\in \widehat{G'}}n_{\tau\xi}M_{c}(G',\h/\h^{G'},\xi),
$$
and 
$$
\Ind_b(M_{c}(G',\h/\h^{G'},\xi))=\oplus_{\tau\in \widehat{G}}
n_{\tau\xi}M_{c}(G,\h,\tau),
$$
where $n_{\tau\xi}$ is the multiplicity of occurrence of 
$\xi$ in $\tau|_{G'}$, and similarly for $\res_\lambda$, 
${\mathrm ind}_\lambda$. 
\item[(ii)] The equations of (i) hold at the level of Grothendieck
groups for all $c$. 
\end{enumerate}

\end{proposition} 

\begin{proof}
Part (i) is easy for $c=0$, and is obtained for generic $c$ 
by a deformation argument. Part (ii) is also obtained 
by deformation argument, taking into account that the functors
$\Res_b$ and $\Ind_b$ are exact and flat with respect to $c$. 
\end{proof} 

\begin{example} 
Suppose that $G'=1$. Then 
$\Res_b(M)$ is the fiber of $M$ at $b$, while
$\Ind_b(\CC)=P_{KZ}$, the object defined in \cite{GGOR}, which is
projective and injective (see Remark \ref{conhol}). This shows that Proposition
\ref{compuRI} (i) does not hold for special $c$, as $P_{KZ}$ is
not, in general, a direct sum of standard modules.  
\end{example}

\subsection{Dependence of the functor $\Res_b$ on $b$}\label{dep}
\label{sec:res}

Let $G'\subset G$ be a parabolic subgroup. 
In the construction of the functor $\Res_b$, 
the point $b$ can be made a variable which belongs to the open
set $\h^{G'}_{\reg}$. 

Namely, let $\widehat{\h^{G'}_{\reg}}$
be the formal neighborhood of the locally closed set $\h^{G'}_{\reg}$
in $\h$, and let $\pi: \widehat{\h^{G'}_{\reg}}\to \h/G$ be 
the natural map (note that this map is an \'etale covering
of the image with the Galois group $N_G(G')/G'$, where
$N_G(G')$ is the normalizer of $G'$ in $G$). 
Let $\widehat{H_{1, c}}(G, \h)_{\h^{G'}_{\reg}}$ be the
pullback of the sheaf $H_{1,c,G,\h}$ under $\pi$. 
We can regard it as a sheaf of algebras over 
$\h^{G'}_{\reg}$. Similarly to Theorem \ref{comp}
we have an isomorphism 
$$
\theta: \widehat{H_{1, c}}(G, \h)_{\h^{G'}_{\reg}}\to 
Z(G,G',\widehat{H_{1, c}}(G',\h/\h^{G'})_0)\hat\otimes
\cD(\h^{G'}_{\reg}),
$$
where $\cD(\h^{G'}_{\reg})$ is the sheaf of differential
operators on $\h^{G'}_{\reg}$, and $\hat\otimes$ is an
appropriate completion of the tensor product.

Thus, repeating the construction of $\Res_b$, we can define 
the functor 
$$
\Res: \cO _c(G, \h)_0\to 
\cO _{c}(G',\h/\h^{G'})_0\boxtimes {\mathrm
Loc}(\h^{G'}_{\reg}),
$$
where ${\mathrm Loc}(\h^{G'}_{\reg})$ stands for the category of
local systems (i.e. $\cO$-coherent $\cD$-modules) on $\h^{G'}_{\reg}$. 
This functor has the property that $\Res_b$ is the fiber of $\Res$ at $b$.
Namely, the functor $\Res$ is defined by the formula 
$$
\Res(M)=(E\circ I^{-1}\circ
\theta_*)(\widehat{M}_{\h^{G'}_{\reg}}),
$$  
where $\widehat{M}_{\h^{G'}_{\reg}}$ is the restriction of the
sheaf $M$ on $\h$ to the formal neighborhood of $\h^{G'}_{\reg}$.

\begin{remark}
If $G'$ is the trivial group, the functor $\Res$
is just the KZ functor from \cite{GGOR}, which we will discuss later. 
Thus, $\Res$ is a relative version of the KZ functor. 
\end{remark}

\begin{remark}
Note that the object ${\mathrm Res}(M)$ is naturally equivariant under the
group $N_G(G')/G'$. 
\end{remark}

Thus, we see that the functor $\Res_b$ does not depend on $b$, up
to an isomorphism. A similar statement is true for the functors 
$\Ind_b$, $\res_\lambda$, $\ind_\lambda$.

\begin{conjecture}\label{bl}
For any $b\in \h,\lambda\in \h^*$ such that $G_b=G_\lambda$, 
we have isomorphisms of functors $\Res_b\cong \res_\lambda$,
$\Ind_b\cong \ind_\lambda$. 
\end{conjecture}

\begin{remark} Conjecture \ref{bl} would imply that 
$\Ind_b$ is left adjoint to $\Res_b$, and 
that $\Res_b$ maps injective objects to injective ones, while
$\Ind_b$ maps projective objects to projective ones. 
\end{remark}

\begin{remark}\label{conhol} If $b$ and $\lambda$ are generic (i.e.,
$G_b=G_\lambda=1$) then the conjecture holds. 
Indeed, in this case the conjecture reduces to showing that 
we have an isomorphism of functors ${\mathrm Fiber}_b(M)\cong
{\mathrm Fiber}_\lambda(M^\dagger)^*$ ($M\in \cO_c(G, \h)$). 
Since both functors are exact functors
to the category of vector spaces, it suffices to check that 
$\dim {\mathrm Fiber}_b(M)=\dim {\mathrm Fiber}_\lambda(M^\dagger)$. 
But this is true because both
dimensions are given by the leading coefficient of the Hilbert
polynomial of $M$ (characterizing the growth of $M$).   
\end{remark}

It is important to mention, however, that
although $\Res_b$ is isomorphic to $\Res_{b'}$ if $G_b=G_{b'}$,
this isomorphism is not canonical. So let us examine the dependence
of $\Res_b$ on $b$ a little more carefully. 

Theorem \ref{compuRI} implies that if $c$ is generic, then 
$$
\Res(M_c(G,\h,\tau))=\oplus_{\xi}M_{c}(G',\h/\h^{G'},\xi)
\otimes{\mathcal L}_{\tau\xi},
$$
where ${\mathcal L}_{\tau\xi}$ is a local system on $\h^{G'}_{\reg}$ of rank
$n_{\tau\xi}$. Let us characterize the local system 
${\mathcal L}_{\tau\xi}$ explicitly. 

\begin{proposition}\label{tauxi}
The local system ${\mathcal L}_{\tau\xi}$ is given by the
connection on the trivial bundle given by the formula 
$$
\nabla=\d-\sum_{s\in \cS: s\notin G'}\frac{2c_s}{1-\lambda_s}
\frac{\d\alpha_s}{\alpha_s}(1-s).
$$
with values in $\Hom_{G'}(\xi,\tau|_{G'})$. 
\end{proposition}

\begin{proof}
This follows immediately from formula (\ref{eqn:thetay}).
\end{proof} 

\begin{definition}
We will call the connection of Proposition \ref{tauxi} 
the parabolic KZ (Knizhnik-Zamolodchikov) connection.  
\end{definition}

\begin{example}
Let $G=\kS_{n}$ and 
$G'=\kS_{n_{1}}\times\cdots\times \kS_{n_{k}}$ 
with $n_{1}+\cdots+n_{k}=n$. In this case, there is only one parameter
$c$.

Let $\h=\CC^{n}$ be the permutation representation of $G$. Then
$$\h^{G'}=(\CC^{n})^{G'}=\{v\in \h|v
=(\underbrace{z_{1}, \ldots, z_{1}}_{n_{1}}, 
\underbrace{z_{2}, \ldots, z_{2}}_{n_{2}}, \ldots, 
\underbrace{z_{k}, \ldots, z_{k}}_{n_{k}})\}.$$
Thus, the parabolic KZ connection on the trivial
bundle with fiber being a representation $\tau$ of
$\kS_n$ has the form
$$
\d -c\sum_{1\leq p<q\leq k}
\sum_{i=n_{1}+\cdots+n_{p-1}+1}^{n_{1}+
\cdots+n_{p}}\sum_{j=n_{1}+\cdots+n_{q-1}+1}^{ 
n_{1}+\cdots+n_{q}}\frac{\d z_{p}-\d z_{q}}{z_{p}-z_{q}}(1-s_{ij}).
$$
So the differential equations for a flat section $F(z)$ of this
bundle have the form 
$$
\frac{\partial F}{\partial z_p}=
c\sum_{q\ne p}
\sum_{i=n_{1}+\cdots+n_{p-1}+1}^{n_{1}+
\cdots+n_{p}}\sum_{j=n_{1}+\cdots+n_{q-1}+1}^{ 
n_{1}+\cdots+n_{q}}\frac{(1-s_{ij})F}{z_{p}-z_{q}}.
$$
So $F(z)=G(z)\prod_{p<q}(z_p-z_q)^{cn_pn_q}$, where 
the function $G$ satisfies the differential equation 
$$
\frac{\partial G}{\partial z_p}=
-c\sum_{q\ne p}
\sum_{i=n_{1}+\cdots+n_{p-1}+1}^{n_{1}+
\cdots+n_{p}}\sum_{j=n_{1}+\cdots+n_{q-1}+1}^{ 
n_{1}+\cdots+n_{q}}\frac{s_{ij}G}{z_{p}-z_{q}}.
$$

Let $\tau=V^{\otimes n}$ where $V$ is a finite dimensional space with 
$\dim V=N$ (this class of representations contains as summands all
irreducible representations of $\kS_n$). 
Let $V_p=V^{\otimes n_p}$, so that $\tau=V_1\otimes\cdots\otimes
V_k$. Then the equation for $G$
can be written as 
$$
\frac{\partial G}{\partial z_{p}}=-c\sum_{q\neq p}
\frac{\Omega_{pq}G}{z_{p}-z_{q}},\quad p=1, \ldots, k,
$$
where $\Omega=\sum_{s,t=1}^{N}E_{s,t}\otimes
E_{t,s}$ is the Casimir element for $\gl_{N}$
($E_{i,j}$ is the $N$ by $N$ matrix with the only $1$ at the
$(i,j)$-th place, and the rest of the entries being $0$).

This is nothing but the well known Knizhnik-Zamolodchikov 
system of equations of the WZW conformal field 
theory, for the Lie algebra $\gl_{N}$, see \cite{EFK}.
(Note that the representations $V_i$ are ``the most general''
in the sense that any irreducible finite dimensional representation of  
$\gl_{N}$ occurs in $V^{\otimes r}$ for some $r$, 
up to tensoring with a character.) 

This motivates the term ``parabolic KZ connection''. 
\end{example}

\subsection{Supports of modules}

The following two basic propositions are proved in \cite{Gi1},
Section 6. We will give different proofs of them, based on the 
restriction functors. 

\begin{proposition}\label{prop:supp} 
Consider the stratification of $\h$ with respect to stabilizers
of points in $G$. Then the (set-theoretical) 
support $\Supp M$ of any object
$M$ of $\cO_c(G, \h)$ in $\h$ is a union of strata of this stratification.  
\end{proposition}

\begin{proof}
This follows immediately from the existence of the flat connection 
along the set of points $b$ with a fixed stabilizer $G'$ on the bundle
$\Res_b(M)$. 
\end{proof} 

\begin{proposition}\label{prop:supp1}
For any irreducible object $M$ in 
$\cO_c(G, \h)$, ${\mathrm Supp}M/G$ is an irreducible algebraic variety. 
\end{proposition} 

\begin{proof}
Let $X$ be a component of $\Supp M/G$. 
Let $M'$ be the subspace of elements of $M$ whose 
restriction to a neighborhood of a generic point 
of $X$ is zero. It is obvious that $M'$ is an 
$H_{1, c}(G, \h)$-submodule in $M$. By definition, it is a proper
submodule. Therefore, by the irreducibility of $M$, we have
$M'=0$. Now let $f\in \CC[\h]^{G}$ be a function that vanishes
on $X$. Then there exists a positive integer 
$N$ such that $f^N$ maps $M$ to $M'$, hence acts by zero 
on $M$. This implies that
$\Supp M/G=X$, as desired. 
\end{proof} 

Propositions \ref{prop:supp} and \ref{prop:supp1} allow us to attach 
to every irreducible module $M\in \cO_c(G, \h)$, 
a conjugacy class of parabolic subgroups, $C_M\in
\Par(G)$, namely, the conjugacy class of the 
stabilizer of a generic point of the support
of $M$. Also, for a parabolic subgroup $G'\subset G$, 
denote by ${\mathcal X}(G')$ the set of points $b\in \h$ whose
stabilizer contains a subgroup conjugate to $G'$.

The following proposition is immediate. 

\begin{proposition}\label{supp2} 
\begin{enumerate}
\item[(i)] Let $M\in \cO_c(G, \h)_0$ be irreducible. If $b$ is such that $G_b\in C_M$,
then $\Res_b(M)$ is a nonzero finite dimensional module over
$H_{1,c}(G_b,\h/\h^{G_b})$. 

\item[(ii)] Conversely, let $b\in \h$, and $L$ be a finite dimensional
module $H_{1, c}(G_b,\h/\h^{G_b})$. 
Then the support of $\Ind_b(L)$ in $\h$ 
is ${\mathcal X}(G_b)$. 
\end{enumerate}

\end{proposition}

Let $\mathrm{FD}(G, \h)$ be the set of $c$ for which $H_{1, c}(G, \h)$ admits a
finite dimensional representation. 

\begin{corollary}\label{supp3} 
Let $G'$ be a parabolic subgroup of $G$.   
Then ${\mathcal X}(G')$ is the support of some irreducible representation 
from $\cO_c(G, \h)_0$ if and only if $c\in
\mathrm{FD}(G',\h/\h^{G'})$.  
\end{corollary}

\begin{proof}
Immediate from Proposition \ref{supp2}. 
\end{proof} 

\begin{example}
Let $G=\kS_n$, $\h=\CC^{n-1}$.  In this case, the set $\Par(G)$ is 
the set of partitions of $n$. Assume that $c=r/m$, $(r,m)=1$, 
$2\le m\le n$. By a result of \cite{BEG},
finite dimensional representations of $H_c(G, \h)$ exist 
if and only if $m=n$. Thus the only possible classes $C_M$
for irreducible modules $M$ have stabilizers $\kS_m\times\cdots\times
\kS_m$, i.e., correspond to partitions into parts, 
where each part is equal to $m$ or $1$. So there are 
$[n/m]+1$ possible supports for modules, 
where $[a]$ denotes the integer part of $a$.    
\end{example} 

\subsection{Notes}
Our discussion of the geometric 
approach to rational Cherednik algebras in Section 5.1 follows
\cite{E1} and Section 2.2 of \cite{BE}. 
Our exposition in the other sections follows the corresponding parts of 
the paper \cite{BE}. 

\newpage \section{The Knizhnik-Zamolodchikov functor}

\subsection{Braid groups and Hecke algebras}

Let $G$ be a complex reflection group and 
let $\h$ be its reflection representation. For any reflection hyperplane
$H\subset \h$, its pointwise 
stabilizer is a cyclic group 
of order $m_{H}$. Fix a collection of nonzero constants 
$q_{1, H}, \ldots, q_{m_{H}-1, H}$ which are $G$-invariant,
namely, if $H$ and $H'$ are conjugate to each other under 
some element in $G$, then $q_{i, H}=q_{i, H'}$ for $i=1, \ldots, m_{H}-1$.

Let $B_{G}=\pi_{1}(\h_{\reg}/G,{\mathbf x}_0)$ be the braid group of $G$, 
and $T_{H}\in B_{G}$ be a representative of the conjugacy class
defined by a small circle around 
the image of $H$ in $\h/G$ oriented in the counterclockwise direction.

The following theorem follows from elementary algebraic
topology. 

\begin{proposition}\label{quott}
The group $G$ is the quotient 
of the braid group $B_{G}$ 
by the relations 
$$
T_{H}^{m_H}=1
$$ 
for all reflection hyperplanes $H$. 
\end{proposition}

\begin{proof}
See, e.g., \cite{BMR} Proposition 2.17.
\end{proof}

\begin{definition}
The Hecke algebra of $G$ is defined to be
$$
\mathscr{H}_{q}(G)=\CC[B_{G}]/
\langle (T_{H}-1)\prod_{j=1}^{m_{H}-1}
(T_{H}-\exp(2\pi{\bf i} j/m_{H})q_{j, H}), \text{ for all }H\rangle.
$$
\end{definition} 

Thus, by Proposition \ref{quott} 
we have an isomorphism
$$
\mathscr{H}_{1}(G)\cong \CC G.
$$
So $\mathscr{H}_{q}(G)$ is a deformation of 
$\CC G$. 

\begin{example}[Coxeter group case]
Now let $W$ be a Coxeter group. Let $\cS$ be the set of reflections
and let $\alpha_{s}=0$ be the reflection hyperplane corresponding to $s\in \cS$. 
The Hecke algebra $\mathscr{H}_{q}(W)$ is the quotient of $\CC[B_{W}]$ by the relations
$$(T_{s}-1)(T_{s}+q_s)=0,$$
for all reflections $s$ where $T_{s}$ is a small
counterclockwise circle around the image of the hyperplane
$\alpha_{s}=0$ in $\h/W$.
\end{example}

\subsection{KZ functors}
For a complex reflection group $G$,
let $\Loc(\h_{\reg})$ be the category of local systems 
(i.e., $\cO$-coherent $\cD$-modules) on 
$\h_{\reg}$, and let $\Loc(\h_{\reg})^{G}$
be the category of $G$-equivariant local systems on 
$\h_{\reg}$, i.e. of local systems on $\h_\reg/G$.

Suppose that $G'=1$ is the trivial subgroup in $G$. 
Then the restriction functor defined in 
Section \ref{sec:res} defines a functor 
$\Res: \cO_c(G, \h)_{0}\to \Loc(\h_{\reg}/G)$.
Also, we have the monodromy functor 
${\mathrm Mon}: \Loc(\h_{\reg}/G)\cong \Rep(B_{G})$. 
The composition of these two functors is a functor 
from $\cO_c(G, \h)_{0}$ to 
$\Rep(B_{G})$, which is exactly the KZ functor 
defined in \cite{GGOR}. We will denote this functor by 
$\KZ$.

\begin{theorem}[Ginzburg, Guay, Opdam, Rouquier, \cite{GGOR}]\label{GGORt}
The KZ functor factors through \linebreak
$\Rep \mathscr{H}_{q}(G)$, where
$$
q_{j, H}=\exp(2\pi {\bi}b_{j, H}/m_{H}),
\quad\text{and}\quad
b_{j,H}=2\sum_{\ell=1}^{m_{H}-1}\frac{c_{s_{H}^{\ell}}(1-\be^{2\pi\bi j\ell/m_{H}})}{1-\be^{-2\pi {\bf i}\ell/m_{H}}}.
$$
\end{theorem}

\begin{proof}
Assume first that $c$ is generic. Then 
the category $\cO_c(G, \h)_0$ is semisimple, with simple objects
$M_c(\tau)$, so it is enough to check the
statement on $M_{c}(\tau)$. Consider the trivial bundle over $\h_\reg$ 
with fiber $\tau$. 
The KZ connection on it has the form
$$
\d-\sum_{s\in \cS}\frac{2c_{s}}{1-\lambda_{s}}\frac{\d \alpha_{s}}{\alpha_{s}}(1-s).
$$ 
The residue of the connection form 
of this connection on the hyperplane $H$ on the 
$j$-th irreducible representation of $\ZZ/m_{H}\ZZ$ is
$$
2\sum_{\ell=1}^{m_{H}-1}\frac{c_{s_{H}^{\ell}}}{1-\be^{-2\pi{\bf i}\ell/m_{H}}}
(1-\be^{2\pi{\bf i}j\ell/m_{H}}).
$$
Therefore, the monodromy operator around 
this hyperplane is diagonalizable, and
the eigenvalues of this operator are 
$1$ and $\exp(2\pi{\bf i}j/m_{H})q_{j,H}$, as desired.

For special $c$, introduce the generalized Verma module 
$$
M_{c, n}(\tau)=H_{c}(G, \h)\otimes_{\CC G\ltimes S\h}
(\tau\otimes S\h/\mathfrak{m}^{n+1}),
$$
where $\mathfrak{m}\subset S\h$ is the maximal ideal of
$0$, $n\geq 0$. Clearly, $M_{c, 0}=M_{c}(\tau)$. Moreover,
$M_{c, n}\in \cO_c(G, \h)_{0}$ for any $n$,
since it has a finite 
filtration whose successive quotients are Verma modules.

\begin{theorem}{\label{thm-proj}}
For large enough $n$, $M_{c, n}(\CC G)$ contains a direct summand 
which is a projective generator of $\cO_{c}(G, \h)_{0}$.
\end{theorem}

\begin{proof}
From the definition, $M_{c, n}=S\h^{*}\otimes\CC G\otimes S\h/\mathfrak{m}^{n+1}$.
Let $\partial$ be the degree operator on $M_{c, n}(\CC G)$
with $\deg \h^{*}=1, \deg \h=-1$, and $\deg G=0$,
i.e., we have
$$[\partial, x]=x, [\partial, y]=-y, \text{ where }x\in \h^{*}, y\in \h.$$

So $\bh-\partial$ is a module endomorphism of $M_{c, n}(\CC G)$
where  $\bh$ is the operator defined in \eqref{eqn:h}. 
Moreover, $\bh-\partial$ acts locally finitely. In particular,
we have a decomposition of $M_{c, n}(\CC G)$
into generalized eigenspaces of $\bh-\partial$:
$$M_{c, n}(\CC G)=\bigoplus_{\beta\in \CC }M_{c, n}^{\beta}(\CC G).$$

We have $$\Hom(M_{c, n}(\CC G), N)=\{\text{vectors in $N$ which are killed by $\mathfrak{m}^{n+1}$} \},$$
and 
$$\Hom(M_{c, n}^{\beta}(\CC G), N)=\{\text{vectors in $N$ which are killed by $\mathfrak{m}^{n+1}$}$$
$$\qquad\text{and are generalized eigenvectors
of $\bh$ with generalized eigenvalue $\beta$} \}.$$

Let $\Sigma=\{h_{c}(\tau)|\tau \text{ is a irreducible representation of $G$}\}$
(recall that \linebreak
$h_c(\tau)=\frac{\dim \h}{2}-\sum_{s\in \cS}\frac{2c_s}
{1-\lambda_s}s|_\tau$), and let 
$$M^{\Sigma}_{c, n}(\CC G)=\bigoplus_{\beta\in \Sigma} M_{c, n}^{\beta}(\CC G).$$

{\bf Claim}: for large $n$, $M_{c, n}^{\Sigma}(\CC G)$ is a projective 
generator of $\cO_{c}(G, \h)_{0}$.

\begin{proof}[Proof of the claim]
First, for any $\beta$, there exists 
$n$ such that $M_{c, n}^{\Sigma}(\CC G)$ is projective
(since the condition of being killed by $\mathfrak{m}^{n+1}$ 
is vacuous for large $n$). 

Secondly, consider the functor $\Hom(M_{c, n}^{\Sigma}(\CC G),  \bullet)$. 
For any module $N\in \cO_c(G, \h)_0$, if $\Hom(M_{c, n}^{\Sigma}(\CC G), N)=0$, then
$\oplus_{\beta\in \Sigma}N[\beta]=0$. So $N=0$. Thus this functor 
does not kill nonzero objects, and so $M_{c, n}^{\Sigma}(\CC G)$ is a generator.
\end{proof}

Theorem \ref{thm-proj} is proved. 
\end{proof}

\begin{corollary}
\begin{enumerate}
\item[(i)] $\cO_{c}(G, \h)_{0}$ has enough projectives, so it is 
equivalent to the category of modules over a finite dimensional 
algebra.
\item[(ii)] Any object of $\cO_{c}(G, \h)_{0}$ is a quotient
of a multiple of $M_{c, n}(\CC G)$ for large enough $n$. 
\end{enumerate}

\end{corollary}

\begin{proof}
Directly from the definition and the above theorem.
\end{proof}

Now we can finish the proof of Theorem \ref{GGORt}.
We have shown that for generic $c$, \linebreak
$\KZ(M_{c, n}(\CC G))\in \Rep 
\mathscr{H}_{q}(G)$. Hence this is true for any $c$, since 
$M_{c, n}(\CC G)$ is a flat family of modules over $H_{c}(G, \h)$.
Then, $\KZ(M)$ is a $\mathscr{H}_{q}(G)$-module for all $M$,
since any $M$ is a quotient of $M_{c, n}(\CC G)$ and the functor 
$\KZ$ is exact.
\end{proof}

\begin{corollary} [Brou\'e, Malle, Rouquier, \cite{BMR}]
Let $q_{j, H}=\exp(t_{j, H})$ where $t_{j, H}$'s are formal parameters.
Then $\mathscr{H}_{q}(G)$ is a free module over $\CC[[t_{j, H}]]$
of rank $|G|$.  
\end{corollary}
\begin{proof}
We have 
$$\mathscr{H}_{q}(G)/(t)=\mathscr{H}_{1}(G)=\CC G.$$
So it remains to show that $\mathscr{H}_{q}(G)$ is free. To show this,
it is sufficient to show that any $\tau\in \mathrm{Irrep}G$ admits a flat 
deformation $\tau_{q}$ to a representation of $\mathscr{H}_{q}(G)$.
We can define this deformation by letting $\tau_{q}=\KZ(M_{c}(\tau))$.
\end{proof}

\begin{remark}
1. The validity of this Corollary in characteristic zero
implies that it is also valid over a field 
positive characteristic. 

2. It is not known in general if the Corollary holds for numerical $q$
(even generically). This is a conjecture of Brou\'e, Malle, and Rouquier. 
But it is known for many cases (including all Coxeter groups). 

3. The proof of the Corollary is analytic (it is based on the notion of 
monodromy). There is no known algebraic proof, except in special cases, and 
in the case of Coxeter groups, which we'll discuss later. 
\end{remark}

\subsection{The image of the KZ functor}
First, let us recall the definition of a quotient category.
Let $\mathcal{A}$ be an abelian category and $\mathcal{B}\subset \mathcal{A}$ a full abelian subcategory.

\begin{definition}
$\mathcal{B}$ is a Serre subcategory if it is closed under subquotients and 
extensions (i.e., if two terms in a short exact sequence are in $\mathcal{B}$, so is the third one).
\end{definition}

If $\mathcal{B}\subset \mathcal{A}$ is a Serre subcategory, one can define a category $\mathcal{A}/\mathcal{B}$ as follows:
$$
\text{objects in }\mathcal{A}/\mathcal{B} =\text{ objects in }\mathcal{A},
$$
$$
\Hom_{\mathcal{A}/\mathcal{B}}(X, Y)=
\mathop{\lim_{\longrightarrow}}_{ Y', X/X'\in \mathcal{B}}\Hom_{\mathcal{A}}
(X', Y/Y').
$$

The category $\mathcal{A}/\mathcal{B}$ is an abelian category with the
following universal property:
any exact functor $F: \mathcal{A}\to \mathcal{C}$ that 
kills $\mathcal{B}$ must factor through
$\mathcal{A}/\mathcal{B}$.

Now let $\cO_{c}(G, \h)^{\mathrm{tor}}_{0}$
be the full subcategory of $\cO_{c}(G, \h)_{0}$
consisting of modules supported on the reflection hyperplanes.
It is a Serre subcategory, and 
$\ker(\KZ)=\cO_{c}(G, \h)^{\mathrm{tor}}_{0}$.
Thus we have a functor:
$$\overline{\KZ}:\cO_{c}(G, \h)_{0}
/\cO_{c}(G, \h)^{\mathrm{tor}}_{0}\to \Rep 
\mathscr{H}_{q}(G).$$

\begin{theorem}[Ginzburg, Guay, Opdam, Rouquier, \cite{GGOR}]
If $\dim\mathscr{H}_{q}(G)=|G|$, the functor
$\overline{\KZ}$ is an equivalence of categories.
\end{theorem}

\begin{proof}
See \cite{GGOR}, Theorem 5.14.
\end{proof}

\subsection{Example: the symmetric group $\kS_{n}$}

Let $\h=\CC^n$, $G=\kS_n$. Then we have (for $q\in\CC ^{*}$):
$$
\mathscr{H}_{q}(\kS_n)=\langle T_{1}, \ldots, T_{n-1}\rangle/
\langle\text{the braid relations and }(T_{i}-1)(T_{i}+q)=0\rangle.
$$

The following facts are known:
\begin{enumerate}
\item $\dim \mathscr{H}_{q}(\kS_n)=n!$;
\item $\mathscr{H}_{q}(\kS_{n})$ is semisimple
if and only if $\mathrm{ord}(q)\neq 2,3,\ldots, n.$
\end{enumerate}

Now suppose $q$ is generic. 
Let $\lambda$ be a partition of $n$. Then
we can define an $\mathscr{H}_{q}(\kS_n)$-module $S_{\lambda}$,
the Specht module for the Hecke algebra in the sense of \cite{DJ}.
This is a certain deformation of the classical irreducible Specht module 
for the symmetric group. 
The Specht module carries an inner product $\langle\cdot,\cdot\rangle$.
Denote $D_{\lambda}=S_{\lambda}/\mathrm{Rad}\langle\cdot,\cdot\rangle$.

\begin{theorem}[Dipper, James, \cite{DJ}]
$D_{\lambda}$ is either an irreducible $\mathscr{H}_{q}(\kS_n)$-module, or $0$. 
$D_{\lambda}\neq 0$ if and only if $\lambda$ is $e$-regular where
$e=\mathrm{ord}(q)$ (i.e., every part of $\lambda$ occurs less than 
$e$ times).
\end{theorem}
\begin{proof}
See \cite{DJ}, Theorem 6.3, 6.8.
\end{proof}

Now let $M_{c}(\lambda)$ be the Verma module associated
to the Specht module for $\kS_{n}$ and $L_{c}(\lambda)$ be its 
irreducible quotient.
Then we have the following theorem.
\begin{theorem}
If $c\leq 0$, then 
$\KZ(M_{c}(\lambda))=S_{\lambda}$ and
$\KZ(L_{c}(\lambda))=D_{\lambda}$.
\end{theorem}
\begin{proof}
See Section 6.2 of \cite{GGOR}.
\end{proof}

\begin{corollary}
If $c\le 0$, then 
$\mathrm{Supp}L_{c}(\lambda)=\CC ^{n}$ if and only if
$\lambda$ is $e$-regular. If $c>0$, then 
$\mathrm{Supp}L_{c}(\lambda)=\CC ^{n}$ if and only if
$\lambda^{\vee}$ is $e$-regular, or equivalently, 
$\lambda$ is $e$-restricted 
(i.e., $\lambda_{i}-\lambda_{i+1}<e$ for $i=1, \ldots, n-1$).
\end{corollary}
\begin{proof}
Directly from the definition and the above theorem.
\end{proof}

\subsection{Notes}
The references for this section are \cite{GGOR}, \cite{BMR}.

\newpage \section{Rational Cherednik algebras and Hecke algebras for varieties with group actions}

\subsection{Twisted differential operators}{\label{sec:tdo}}

Let us recall the theory of twisted differential
operators (see \cite{BB}, section 2).

Let $X$ be a smooth affine algebraic variety over $\CC $. 
Given a closed 2-form $\omega$ on $X$, 
the algebra $\cD_\omega(X)$ of differential operators on $X$
twisted by $\omega$ can be defined 
as the algebra generated by $\cO_X$ and 
``Lie derivatives'' ${\mathbf L}_v$, $v\in \mathrm{Vect}(X)$,
with defining relations 
$$
f{\mathbf L}_v={\mathbf L}_{fv},\ [{\mathbf L}_v,f]=L_vf,\ [{\mathbf L}_v,{\mathbf L}_w]={\mathbf L}_{[v,w]}+\omega(v,w).
$$
This algebra depends only on the
cohomology class $[\omega]$ of 
$\omega$, and equals the algebra $\cD(X)$ of usual 
differential operators on $X$ if $[\omega]=0$. 

An important special case of twisted differential operators is 
the algebra of differential operators on a line bundle. 
Namely, let $L$ be a line bundle on $X$. Since $X$ is affine, $L$
admits an algebraic connection $\nabla$ with curvature $\omega$,
which is a closed 2-form on $X$. Then it is easy to show that 
the algebra $\cD(X,L)$ of differential operators on $L$ 
is isomorphic to $\cD_\omega(X)$. 

If the variety $X$ is smooth but not necessarily affine, 
then (sheaves of) algebras of twisted differential operators 
are classified by 
the space $\H^2(X,\Omega_X^{\ge 1})$, where 
$\Omega_X^{\ge 1}$ is the two-step complex 
of sheaves $\Omega^1_X\to \Omega^{2,\mathrm{cl}}_X$, given by the De Rham
differential acting from 1-forms to closed 2-forms (sitting in
degrees 1 and 2, respectively). 
If $X$ is projective then this space is isomorphic 
to $\H^{2,0}(X,\CC )\oplus \H^{1,1}(X,\CC )$.
We refer the reader to \cite{BB},
Section 2, for details. 

\begin{remark}
One can show that $\cD_{\omega}(X)$ is the universal deformation
of $\cD(X)$ (see \cite{E1}).
\end{remark}

\subsection{Some algebraic geometry preliminaries}{\label{sec:agp}}

Let $Z$ be a smooth hypersurface in a smooth affine variety $X$. 
Let $i: Z\to X$ be the corresponding closed embedding. 
Let $N$ denote the normal bundle of $Z$ in $X$
(a line bundle). Let $\cO_X(Z)$ 
denote the module of regular functions on
$X\setminus Z$ which have a pole of at most first order at $Z$.
Then we have a natural map of $\cO_X$-modules 
$\phi: \cO_X(Z)\to i_*N$. Indeed, 
we have a natural residue map 
$\eta: \cO_X(Z)\otimes_{\cO_X}\Omega^1_X\to 
i_*\cO_Z$ (where $\Omega^1_X$ is the module of 1-forms), hence 
a map $\eta': \cO_X(Z)\to 
i_*\cO_Z\otimes_{\cO_X}TX=i_*(TX|_Z)$ (where $TX$ is
the tangent bundle). The map $\phi$ is obtained by composing
$\eta'$ with the natural projection $TX|_Z\to N$.  

We have an exact sequence 
of $\cO_X$-modules: 
$$
0\to \cO_X\to \cO_X(Z)\xrightarrow{\phi} i_*N\to 0
$$
Thus we have a 
natural surjective map of $\cO_X$-modules 
$\xi_Z: TX\to \cO_X(Z)/\cO_X$.

\subsection{The Cherednik algebra of a variety with a finite
group action}{\label{sec:chav}}

We will now generalize the definition of $H_{t,c}(G,\h)$
to the global case. 
Let $X$ be an affine 
algebraic variety over 
$\CC$, and $G$ be a finite group of automorphisms of $X$. 
Let $E$ be a $G$-invariant subspace of the space of closed
$2$-forms on $X$, which projects isomorphically to $\H^2(X, \CC)$. 
Consider the algebra $G\ltimes\cO_{T^*X}$, 
where $T^*X$ is the
cotangent bundle of $X$. We are going to define a
deformation $H_{t,c,\omega}(G,X)$ of this algebra
parametrized by 
\begin{enumerate}
\item complex numbers $t$,
\item $G$-invariant functions $c$ 
on the (finite) set $\cS$ of pairs $s=(Y,g)$,
where $g\in G$, and $Y$ is a connected component of 
the set of fixed points $X^g$ such that ${\mathrm codim} Y=1$, and
\item elements $\omega\in E^G=\H^2(X, \CC)^G$. 
\end{enumerate}

If all the parameters are zero, this algebra will conicide with 
$G\ltimes \cO_{T^{*}X}$.

Let $t,c=\lbrace{c(Y,g)\rbrace},\omega\in E^{G}$ be variables. 
Let $\cD_{\omega/t}(X)_r$ be the algebra (over $\CC [t,t^{-1},\omega]$) of 
twisted (by $\omega/t$) differential operators on $X$ with
rational coefficients. 

\begin{definition}\label{doo}
A Dunkl-Opdam operator for $(X, G)$ is 
an element of $\cD_{\omega/t}(X)_r[c]$ given by the formula 
\begin{equation}\label{Dunkl}
D:=t{\mathbf L}_v-\sum_{(Y,g)\in \cS}\frac{2c(Y,g)}{1-\lambda_{Y,g}}
\cdot f_Y(x)\cdot (1-g),
\end{equation}
where $\lambda_{Y,g}$ is the eigenvalue of $g$ on 
the conormal bundle to $Y$, $v\in \Gamma(X,TX)$ is a vector field on $X$,
and $f_Y\in \cO _X(Z)$ 
is an element of the coset $\xi_Y(v)\in {\mathcal
O}_X(Z)/\cO _X$ (recall that $\xi_Y$ is defined in
Subsection \ref{sec:agp}).
\end{definition} 

\begin{definition}
The algebra $H_{t,c,\omega}(X, G)$ is the subalgebra of 
$G\ltimes \cD_{\omega/t}(X)_r[c]$ generated (over $\CC [t,c,\omega]$) 
by the function algebra 
$\cO _X$, the group $G$, and the Dunkl-Opdam 
operators. 
\end{definition} 

By specializing $t,c,\omega$ to numerical values, we can define 
a family of algebras over $\CC $, which we will also denote
$H_{t,c,\omega}(G,X)$. Note that when we set $t=0$, 
the term $t{\mathbf L}_v$
does not become $0$ but turns into the classical momentum. 

\begin{definition}
$H_{t,c,\omega}(G,X)$ is called {\em the Cherednik algebra of the
orbifold $X/G$}. 
\end{definition}

\begin{remark} One has 
$H_{1,0,\omega}(G,X)=G\ltimes \cD_\omega(X)$. 
Also, if $\lambda\ne 0$ then $H_{\lambda t,\lambda
c,\lambda \omega}(G,X)=H_{t,c,\omega}(G,X)$. 
\end{remark}

\begin{example}
$X=\h$ is a vector space and $G$ is a 
subgroup in $\GL(\h)$. Let $v$ be a constant vector field, 
and let $f_Y(x)=(\alpha_Y,v)/\alpha_Y(x)$, where 
$\alpha_Y\in \h^*$ is a nonzero functional vanishing on $Y$.
Then the operator $D$ is 
just the usual Dunkl-Opdam operator $D_v$ in the complex
reflection case (see Section \ref{sec:dop}). 
This implies that all the Dunkl-Opdam operators
in the sense of Definition \ref{doo} 
have the form $\sum f_iD_{y_i}+a$, where $f_i\in \CC [\h]$, $a\in
G\ltimes \CC [\h]$, and $D_{y_i}$ are the usual
Dunkl-Opdam operators (for some basis $y_i$ of $\h$). 
So the algebra $H_{t,c}(G,\h)=H_{t,c,0}(G,X)$ 
is the rational Cherednik algebra for $(G,\h)$, 
see Section \ref{sec:rca}.
\end{example}

The algebra $H_{t,c,\omega}(G,X)$ 
has a filtration $F^\bullet$ which is defined on
generators by $\deg(\cO _X)=\deg(G)=0$, 
$\deg(D)=1$ for Dunkl-Opdam operators $D$. 

\begin{theorem}[the PBW theorem]  \label{pbwth} 
$\,$ We have 
\begin{center}
$\gr_F(H_{t,c,\omega}(G,X))
=G\ltimes \cO (T^*X)[t,c,\omega]$. 
\end{center}
\end{theorem}

\begin{proof} 
Suppose first that $X=\h$ is a vector space and $G$ is a 
subgroup in $\GL(\h)$. Then, as we mentioned,
$H_{t,c,\omega}(G,\h)=H_{t,c}(G,\h)$ is the rational Cherednik algebra for $G,\h$. 
So in this case the theorem is true. 

Now consider arbitrary $X$. 
We have a homomorphism of graded algebras
$$\psi: \gr_F(H_{t,c,\omega}(G,X))\to 
G\ltimes\cO(T^*X)[t,c,\omega]\quad \text{
(the principal symbol homomorphism)}.$$
The homomorphism $\psi$ is clearly surjective, and our job is to show
that it is injective (this is the nontrivial part of the proof). 
In each degree, $\psi$ is a morphism 
of finitely generated $\cO_X^{G}$-modules. 
Therefore, to check its
injectivity, it suffices to check the injectivity on the formal 
neighborhood of each point $z\in X/G$. 

Let $x$ be a preimage of $z$ in $X$, and $G_x$ be the stabilizer
of $x$ in $G$. Then $G_x$ acts on the formal neighborhood $U_x$
of $x$ in $X$. 

\begin{lemma}\label{line}
Any action of a finite group on a formal polydisk over $\CC $ is linearizable.
\end{lemma}
\begin{proof}
Let $\mathscr{D}$ be a formal polydisk over $\CC $.
Suppose we have an action of a finite group $G$ on $\mathscr{D}$.
Then we have a group homomorphism:
$$\rho: G\to \mathrm{Aut}(\mathscr{D})=\GL_{n}(\CC )\ltimes 
\Aut_{U}(\mathscr{D}),$$
where $ \Aut_{U}(\mathscr{D})$ is the group of unipotent 
automorphisms of $\mathscr{D}$ (i.e. those whose derivative at the origin is $1$), 
which is a prounipotent algebraic group.

Our job is to show that the image of $G$ under $\rho$ can be
conjugated into $\GL_{n}(\CC )$. 
The obstruction to this is in the cohomology group
$\H^{1}(G,  \mathrm{Aut}_{U}(\mathscr{D}))$, which is trivial since 
$G$ is finite and $ \mathrm{Aut}_{U}(\mathscr{D})$ is prounipotent
over $\CC $.
\end{proof}

It follows from Lemma \ref{line} that 
it suffices to prove the theorem in the linear case,
which has been accomplished already. We are done.     
\end{proof}

\begin{remark} 
The following remark is meant to clarify
the proof of Theorem \ref{pbwth}. 
In the case $X=\h$, the proof of Theorem \ref{pbwth}
is based, essentially, on the (fairly nontrivial) fact 
that the usual Dunkl-Opdam operators $D_v$ commute with each
other. It is therefore very important to note that 
in contrast with the linear case, for a general $X$ we do {\bf
not} have any natural commuting family of Dunkl-Opdam operators. 
Instead, the operators (\ref{Dunkl}) satisfy a weaker property, which is 
still sufficient to validate the PBW theorem. This property says that 
if $D_1,D_2,D_3$ are Dunkl-Opdam operators corresponding to vector
fields $v_1,v_2, v_3:=[v_1,v_2]$ and some choices of the functions $f_Y$, 
then $[D_1,D_2]-D_3\in G\ltimes \cO(X)$ (i.e., it has no
poles). To prove this property, it is sufficient to consider 
the case when $X$ is a 
formal polydisk, with a linear action of $G$. But in this case
everything follows from the commutativity of the usual
Dunkl operators $D_v$. 
\end{remark}

\begin{example}
\begin{enumerate}
\item Suppose $G=1$. Then for $t\ne 0$, 
$H_{t,0,\omega}(G,X)=\cD_{\omega/t}(X)$. 
\item Suppose $G$ is a Weyl group and $X=H$ the corresponding 
torus. Then $H_{1, c,0}(G,H)$ is called the {\it trigonometric 
Cherednik algebra.}
\end{enumerate}

\end{example}

\subsection{Globalization}{\label{sec:glb}}
Let $X$ be any smooth algebraic variety, 
and $G\subset \Aut(X)$. 
Assume that $X$ admits a cover by affine $G$-invariant
open sets. Then the quotient variety $X/G$ exists. 

For any affine open set $U$ in $X/G$, let $U'$ be the 
preimage of $U$ in $X$. Then we can define 
the algebra $H_{t,c,0}(G,U')$ as above. 
If $U\subset V$, we have an obvious restriction map
$H_{t,c,0}(G,V')\to H_{t,c,0}(G,U')$. 
The gluing axiom is clearly satisfied. 
Thus the collection of algebras 
$H_{t,c,0}(G,U')$ can be extended (by sheafification)
to a sheaf of algebras on $X/G$. 
We are going to denote this sheaf by $H_{t,c,0,G,X}$
and call it the sheaf of Cherednik algebras on $X/G$.
Thus, $H_{t,c,0,G,X}(U)=H_{t,c,0}(G,U')$. 

Similarly, if $\psi\in \H^2(X,\Omega_X^{\ge 1})^G$, we can define 
the sheaf of twisted Cherednik algebras $H_{t,c,\psi,G,X}$. 
This is done similarly to the case of twisted differential
operators (which is the case $G=1$).

\begin{remark}
\begin{enumerate}
\item  The construction of $H_{t,c,\omega}(G,X)$ and 
the PBW theorem extend in a straightforward manner to the case 
when the ground field is not $\CC $ but an algebraically closed
field $k$ of positive characteristic, provided that the order 
of the group $G$ is relatively prime to the characteristic.
\item  The construction and main properties of the (sheaves of) 
Cherednik algebras of algebraic varieties can be extended 
without significant changes to the case when $X$ is a complex analytic 
manifold, and $G$ is not necessarily finite but acts properly 
discontinuously. In the following lectures, we will often work in
this generalized setting. 
\end{enumerate}
\end{remark}

\subsection{Modified Cherednik algebra}{\label{sec:modifyCA}}

It will be convenient for us to use a slight modification 
of the sheaf $H_{t,c,\psi,G,X}$. Namely, let 
$\eta$ be a function on the set of conjugacy classes of $Y$ such
that $(Y,g)\in \cS$. We define $H_{t,c,\eta,\psi,G,X}$ 
in the same way as $H_{t,c,\psi,G,X}$ except that 
the Dunkl-Opdam operators are defined by the 
formula  
\begin{equation}\label{Dunkl1}
D:=t{\mathbf L}_v+\sum_{(Y,g)\in \cS}f_Y(x)
\frac{2c(Y,g)}{1-\lambda_{Y,g}}(g-1)+
\sum_Y f_Y(x)\eta(Y).
\end{equation}

The following result shows that this modification is in fact
tautological. Let $\psi_Y$ be the class in $\H^2(X,\Omega_X^{\ge 1})$ 
defined by the line bundle $\cO_X(Y)^{-1}$,
whose sections are functions vanishing
on $Y$. 

\begin{proposition}
One has an isomorphism $$H_{t,c,\eta,\psi,G,X}\to
H_{t,c,\psi+\sum_Y \eta(Y)\psi_Y,G,X}.$$  
\end{proposition}

\begin{proof} Let $y\in Y$ and $z$ be a function on the formal 
neighborhood of $y$ such that $z|_Y=0$ and $\d z_y\ne 0$. 
Extend it to a system of local formal coordinates 
$z_1=z,z_2,\ldots,z_d$ near $y$. 
A Dunkl-Opdam operator near $y$ for the vector field 
$\frac{\partial}{\partial z}$ 
can be written in the form 
$$
D=\frac{\partial}{\partial
z}+\frac{1}{z}(\sum_{m=1}^{n-1}\frac{2c(Y,g^m)}{1-\lambda_{Y,g}^m}
(g^m-1)+\eta(Y)).
$$
Conjugating this operator by the formal expression 
$z^{\eta(Y)}:=(z^m)^{\eta(Y)/m}$, we get 
$$
z^{\eta(Y)}\circ D\circ z^{-\eta(Y)}=
\frac{\partial}{\partial
z}+\frac{1}{z}\sum_{m=1}^{n-1}\frac{2c(Y,g^m)}{1-\lambda_{Y,g}^m}(g^m-1)
$$
This implies the required statement. 
\end{proof}

We note that the sheaf $H_{1,c,\eta,0,G,X}$
localizes to $G\ltimes \cD_X$ on the complement of all the 
hypersurfaces $Y$. This follows from the fact that the 
line bundle $\cO_X(Y)$ is trivial on the complement of
$Y$.

\subsection{Orbifold Hecke algebras}{\label{sec:oha}}

Let $X$ be a connected and simply connected complex manifold, 
and $G$ is a discrete group
of automorphisms of $X$ which acts properly discontinuously.  
Then $X/G$ is a complex orbifold. 
Let $X'\subset X$ be the set of points with trivial stabilizer.
Fix a base point $x_{0}\in X'$.
Then the braid group of $X/G$ is defined to be 
$B_{G}=\pi_{1}(X'/G, x_{0})$.
We have an exact sequence 
$
1\to K\to B_{G}\to G\to 1.
$

Now let $\cS$ be the set of pairs $(Y,g)$ such that $Y$ is a
component of $X^g$ of codimension 1 in $X$ (such $Y$ will be called
a reflection hypersurface). 
For $(Y,g)\in \cS$, let $G_Y$ be the subgroup of $G$ 
whose elements act trivially on $Y$. This group is obviously
cyclic; let $n_Y=|G_Y|$. 
Let $C_{Y}$ be the conjugacy class in $B_{G}$ corresponding
to a small circle going counterclockwise around 
the image of $Y$ in $X/G$, and
$T_{Y}$ be a representative in $C_{Y}$. 

The following theorem follows from elementary topology: 

\begin{theorem}
$K$ is defined by relations $T_{Y}^{n_{Y}}=1$, for all reflection
hypersurfaces $Y$ (i.e., $K$ is the intersection of all normal subgroups 
of $B_{G}$ containing $T_{Y}^{n_{Y}}$). 
\end{theorem}

\begin{proof}
See, e.g., \cite{BMR} Proposition 2.17.
\end{proof}

For any conjugacy class of hypersurfaces 
$Y$ such that $(Y,g)\in \cS$ we introduce formal parameters
$\tau_{1Y},\ldots,\tau_{n_YY}$. 
The entire collection of these parameters 
will be denoted by $\tau$. Let $A_{0}=\CC [G]$. 

\begin{definition} We define the Hecke algebra of $(G,X)$, 
denoted $A={\mathscr H}_\tau(G,X,x_0)$, to be the quotient of 
the group algebra of the braid group, 
$\CC [B_{G}][[\tau]]$, by the relations
\begin{equation}\label{poly}
\prod_{j=1}^{n_Y} (T-\be^{2\pi j\bi /n_Y}\be^{\tau_{jY}})=0,\ T\in C_Y 
\end{equation}
(i.e., by the closed ideal in the formal series topology 
generated by these relations).  
\end{definition}

Thus, $A$ is a deformation of $A_{0}$.

It is clear that up to an isomorphism this
algebra is independent on the choice of $x_{0}$, so we will sometimes
drop $x_{0}$ form the notation. 

The main result of this section is the following theorem. 

\begin{theorem}\label{p2} 
Assume that $\H^{2}(X, \CC )=0$. 
Then $A={\mathscr H}_\tau(G,X)$ is a {\bf flat} formal deformation 
of $A_{0}$, which means $A=A_{0}[[\tau]]$ as a module over 
$\CC [[\tau]]$. 
\end{theorem}

\begin{example}
Let $\h$ be a finite dimensional vector space, and 
$G$ be a complex reflection group in $\GL(\h)$. 
Then ${\mathscr H}_\tau(G,\h)$ is the Hecke algebra of $G$
studied in \cite{BMR}. It follows from Theorem \ref{p2} that 
this Hecke algebra is flat. This proof of flatness 
is in fact the same as the original proof of this result 
given in \cite{BMR} (based on the Dunkl-Opdam-Cherednik
operators, and explained above). 
\end{example}

\begin{example}
Let $\h$ be a universal covering of a maximal torus of a simply connected 
simple Lie group $G$, $Q^\vee$ be the dual root lattice, 
and $\widehat{G}=G\ltimes Q^{\vee}$ 
be its affine Weyl group. 
Then ${\mathscr H}_\tau(\h,\widehat{G})$ is the affine
Hecke algebra. This algebra is also flat by Theorem \ref{p2}. In
fact, its flatness is a well known result from
representation theory; our proof of flatness is essentially due
to Cherednik \cite{Ch}.  
\end{example}

\begin{example}\label{chere}
Let $G,\h, Q^\vee$ be as in the previous example, $\eta\in \CC_+$ be a complex number with a positive imaginary part, 
and $\widehat{\widehat{G}}= G\ltimes (Q^\vee\oplus \eta Q^\vee)$ 
be the double affine Weyl group. Then 
$\mathscr{H}_\tau(\h,\widehat{\widehat{G}})$ 
is (one of the versions of) the double affine Hecke algebra of Cherednik
(\cite{Ch}), and it is flat by Theorem \ref{p2}. 
The fact that this algebra is flat was proved by
Cherednik, Sahi, Noumi, Stokman 
(see \cite{Ch},\cite{Sa},\cite{NoSt},\cite{St})
using a different approach (q-deformed Dunkl
operators). 
\end{example}

\subsection{Hecke algebras attached to Fuchsian groups}
{\label{sec:HaF}}

Let $H$ be a simply connected complex Riemann surface
(i.e., Riemann sphere, Euclidean plane, or Lobachevsky plane), 
and $\Gamma$ be a cocompact lattice in ${\mathrm Aut}(H)$
(i.e., a Fuchsian group). 
Let $\Sigma=H/\Gamma$. Then $\Sigma$ is a compact complex
Riemann surface. When $\Gamma$ contains elliptic elements (i.e.,
nontrivial elements of finite order), we are going
to regard $\Sigma$ as an orbifold: it has special points $P_i$, 
$i=1,\ldots, m$ with stabilizers $\ZZ_{n_i}$. 
Then $\Gamma$ is the orbifold fundamental group of
$\Sigma$.\footnote{Let $X$ be a connected topological space on with a properly discontinuous 
action of a discrete group $G$. Then the orbifold fundamental group of the orbifold $X/G$ 
with base point $x\in X$, denoted $\pi_1^{\mathrm orb}(X/G,x)$, is the set of pairs $(g,\gamma)$, where 
$g\in G$ and $\gamma$ is a homotopy class of paths leading from $x$ to $gx$, 
with multiplication law $(g_1,\gamma_1)(g_2,\gamma_2)=(g_1g_2,\gamma)$, where $\gamma$ is 
$\gamma_1$ followed by $g_1(\gamma_2)$. Obviously, in this situation we have an exact sequence 
$$
1\to \pi_1(X,x)\to \pi_1^{\mathrm orb}(X/G,x)\to G\to 1.
$$
}

Let $g$ be the genus of $\Sigma$, and $a_l,b_l, l=1,\ldots, g$, be the
$a$-cycles and $b$-cycles of $\Sigma$. Let $c_j$ be the 
counterclockwise loops around $P_j$.
Then $\Gamma$ is generated by $a_l,b_l,c_j$ with relations
$$
c_j^{n_j}=1,\ c_1c_2\cdots c_m=\prod_la_lb_la_l^{-1}b_l^{-1}.
$$
For each $j$, introduce formal parameters $\tau_{kj}$,
$k=1,\ldots, n_j$. 
Define the Hecke algebra ${\mathscr H}_\tau(\Sigma)$ 
of $\Sigma$ to be generated over $\CC [[\tau]]$ by 
the same generators $a_l,b_l,c_j$ with defining relations 
$$
\prod_{k=1}^{n_j}(c_j-\be^{2\pi j \bi /n_j}\be^{\tau_{kj}})=0, \
c_1c_2\cdots c_m=\prod_la_lb_la_l^{-1}b_l^{-1}. 
$$
 Thus ${\mathscr H}_\tau(\Sigma)$ is a deformation 
of $\CC [\Gamma]$.

This deformation is flat if $H$ is a Euclidean plane 
or a Lobachevsky plane. Indeed,  
${\mathscr H}_\tau(\Sigma)={\mathscr H}_\tau(\Gamma,H)$, so the result
follows from Theorem \ref{p2} and the fact that
$\H^{2}(H, \CC )=0$. 

On the other hand, if $H$ is the Riemann sphere 
(so that the condition $\H^{2}(H, \CC )=0$ is violated)
and $\Gamma\ne 1$ then 
this deformation is not flat. Indeed, let $\tau=\tau(\hbar)$ 
be a $1$-parameter subdeformation of ${\mathscr H}_\tau(\Sigma)$
which is flat. Let us compute the determinant of the product
$c_1\cdots c_m$ in the regular representation of this algebra (which
is finite dimensional if $H$ is the sphere). On the one hand, it is
$1$, as $c_1\cdots c_m$ is a product of commutators. On the other
hand, the eigenvalues of $c_j$ in 
this representation are 
$\be^{2\pi j\bi /n_j}\be^{\tau_{kj}}$ with multiplicity $|\Gamma|/n_j$. 
Computing determinants 
as products of eigenvalues, we get a nontrivial equation on
$\tau_{kj}(\hbar)$, which means that the deformation 
${\mathscr H}_\tau$ is not flat.

Thus, we see that ${\mathscr H}_\tau(\Sigma)$ fails to be flat in the 
following ``forbidden'' cases: 
$$
g=0,\ m=2, (n_1,n_2)=(n,n);
$$
$$
m=3, (n_1,n_2,n_3)=
(2,2,n),(2,3,3),(2,3,4),(2,3,5).  
$$
Indeed, the orbifold Euler characteristic of a closed surface $\Sigma$
of genus $g$ with $m$ special points $x_{1}, \ldots, x_{m}$ whose orders are
$n_1,\ldots,n_m$ is 
$$
\chi^{\mathrm{orb}}(\Sigma, x_{1}, \ldots, x_{m})
=2-2g-m+\sum_{i=1}^{m}\frac{1}{n_{i}},
$$
and above solutions are the solutions of the inequality 
$$
\chi^{\mathrm{orb}}(\CC P^{1}, x_{1}, \ldots, x_{m})>0.
$$
(note that the solutions for $m=1$ and solutions 
$(n_1,n_2)$ with $n_1\ne n_2$ don't arise, since they don't
correspond to any orbifolds). 

\subsection{Hecke algebras of wallpaper groups and del Pezzo
surfaces}{\label{sec:dpez}}

The case when $H$ is the Euclidean plane 
(i.e., $\Gamma$ is a wallpaper group) 
deserves special attention.
If there are elliptic elements, this reduces to the 
following configurations: $g=0$ and 
$$
m=3, (n_1,n_2,n_3)=
(3,3,3),(2,4,4),(2,3,6) \text{ 
(cases $E_6$, $E_7$, $E_8$)},$$
or 
$$
m=4, (n_1,n_2,n_3,n_4)=(2,2,2,2) \text{
(case $D_4$).} $$

In these cases, the algebra ${\mathscr H}_\tau(\Gamma,H)$ 
(for numerical $\tau$)
has Gelfand-Kirillov dimension 2, so it can be interpreted in terms of 
the theory of noncommutative surfaces. 

Recall that a del Pezzo surface (or a Fano surface) 
is a smooth projective surface,
whose anticanonical line bundle is ample.
It is known that such surfaces are $\CC \mathbb P^{1}\times \CC
\mathbb P^{1}$, or a blow-up of $\CC \mathbb P^{2}$ at up to 
$8$ generic points. The degree of a del Pezzo surface $X$ 
is by definition the self intersection number 
$K\cdot K$ of its canonical class $K$.
For example, a del Pezzo surface of degree $3$ is a cubic surface
in $\CC \mathbb P^{3}$, and the degree of $\CC \mathbb P^{2}$ with 
$n$ generic points blown up is $d=9-n$.

Now suppose $\tau$ is numerical. Let
$\hbar=\sum_{j,k} n_j^{-1}\tau_{kj}$. Also let 
$n$ be the largest of $n_j$, and $c$ be the corresponding $c_j$.   
Let $\e\in \CC [c]\subset {\mathscr H}_\tau(\Gamma,H)$
be the projector to an eigenspace of $c$. 
Consider the ``spherical'' 
subalgebra $B_\tau(\Gamma,H):
=\e {\mathcal H}_\tau(\Gamma,H)\e$. 

\begin{theorem} [Etingof, Oblomkov, Rains, \cite{EOR}] 
\begin{enumerate}
\item[(i)] If $\hbar=0$ then the algebra
$B_\tau(\Gamma,H)$ is commutative, and its spectrum 
is an affine del Pezzo surface. 
More precisely, in the case $(2, 2, 2, 2)$, it is 
a del Pezzo surface of degree 3 
(a cubic surface) with a triangle of lines removed; 
in the cases $(3, 3, 3), (2, 4, 4), (2, 3, 6)$ 
it is a del Pezzo surface of degrees 3,2,1 respectively 
with a nodal rational curve removed. 

\item[(ii)] The algebra 
$B_\tau(\Gamma,H)$ for $\hbar\ne 0$ is a quantization of the unique 
algebraic symplectic structure on the surface from (i) with Planck's constant $\hbar$. 
\end{enumerate}

\end{theorem}

\begin{proof}
See \cite{EOR}.
\end{proof}

\begin{remark}
In the case $(2,2,2,2)$, ${\mathcal H}_\tau(\Gamma,\Gamma)$
is the Cherednik-Sahi algebra of rank 1; it controls  
the theory of Askey-Wilson polynomials. 
\end{remark}

\begin{example}
This is a ``multivariate'' version of the Hecke algebras attached
to Fuchsian groups, defined in the previous subsection. 
Namely, letting $\Gamma,H$ be as in the previous subsection, 
and $N\ge 1$, we consider the manifold $X=H^N$ with the action of
$\Gamma_N=\kS_N\ltimes \Gamma^N$. If $H$ is a Euclidean or Lobachevsky plane, 
then by Theorem \ref{p2} ${\mathscr H}_\tau(\Gamma_N,X^N)$ is
a flat deformation of the group algebra $\CC [\Gamma_N]$. 
If $N>1$, this algebra has one more essential parameter 
than for $N=1$ (corresponding to reflections in $\kS_N$). 
In the Euclidean case, one expects that an appropriate 
``spherical'' subalgebra 
of this algebra is a quantization of the Hilbert scheme of 
a del Pezzo surface. 
\end{example}

\subsection{The Knizhnik-Zamolodchikov functor}{\label{sec:KZ}}

In this subsection we will define a global analog of the KZ
functor defined in \cite{GGOR}. This functor will be used as a 
tool of proof of Theorem \ref{p2}.

Let $X$ be a simply connected complex manifold, and 
$G$ a discrete group of holomorphic transformations of $X$ 
acting on $X$ properly discontinuously. 
Let $X'\subset X$ be the set of points with trivial stabilizer.
Then we can define 
the sheaf of Cherednik algebras $H_{1,c,\eta,0,G,X}$ on $X/G$. 
Note that the restriction of this sheaf to $X'/G$ 
is the same as the restriction of the sheaf $G\ltimes \cD_X$ to
$X'/G$ (i.e. on $X'/G$, the dependence of the sheaf on the parameters
$c$ and $\eta$ disappears). This follows from the fact 
that the line bundles $\cO _X(Y)$ become trivial when
restricted to $X'$. 

Now let $M$ be a module over $H_{1,c,\eta,0,G,X}$
which is a locally free coherent sheaf  
when restricted to $X'/G$. 
Then the restriction of $M$ to $X'/G$ is a
$G$-equivariant $\cD$-module on $X'$ which is coherent and locally free 
as an $\cO $-module. Thus, $M$ corresponds to 
a locally constant sheaf (local system) on $X'/G$, which 
gives rise to a monodromy representation of the braid group 
$\pi_1(X'/G,x_0)$ (where $x_0$ is a base point). 
This representation will be denoted by $\KZ(M)$. 
This defines a functor $\KZ$, which is analogous to the one in
\cite{GGOR}.

It follows from the theory of $\cD$-modules that any 
$\cO _{X/G}$-coherent 
$H_{1,c,\eta,0,G,X}$-module is locally free when restricted to 
$X'/G$. Thus the KZ functor acts from the abelian category 
${\mathcal C}_{c,\eta}$ of $\cO _{X/G}$-coherent 
$H_{1,c,\eta,0,G,X}$-modules to the category  
of representations of $\pi_1(X'/G,x_0)$. It is easy to see that
this functor is exact. 

For any $Y$, let $g_Y$ be the generator of $G_Y$ which has
eigenvalue $\be^{2\pi \bi/n_Y}$ in the conormal bundle to $Y$. 
Let $(c,\eta)\to \tau(c,\eta)$ be the invertible 
linear transformation defined by the formula
$$
\tau_{jY}=2\pi \bi(2\sum_{m=1}^{n_Y-1}c(Y,g_Y^m)\frac{1-\be^{2\pi
 jm\bi/n_Y}}
{1-\be^{-2\pi m\bi/n_Y}}-\eta(Y))/n_Y. 
$$ 

\begin{proposition} \label{Heck} The functor $\KZ$ maps
the category ${\mathcal C}_{c,\eta}$ to the category 
of representations of the algebra ${\mathscr
H}_{\tau(c,\eta)}(G,X)$. 
\end{proposition}

\begin{proof} 
The result follows from the corresponding result in the linear
case (which we have already proved) by restricting 
$M$ to the union of $G$-translates of a neighborhood of a 
generic point $y\in Y$, and then linearizing the action
of $G_Y$ on this neighborhood. 
\end{proof} 

\subsection{Proof of Theorem \ref{p2}}{\label{sec:proof}}
Consider the module $M=\Ind_{\cD_{X}}^{G\ltimes \cD_{X}}
\cO_{X}$. Then $\KZ(M)$ is the regular representation
of $G$ which is denoted by $\reg G$.
We want to show that $M$ deforms uniquely (up to an isomorphism)
to a module over $H_{1,c,0,\eta, G,X}$ for formal $c,\eta$.
The obstruction to this deformation is 
in $\Ext^{2}_{G\ltimes \cD_{X}}(M,M)$ and the freedom of this 
deformation is in $\Ext^{1}_{G\ltimes \cD_{X}}(M,M)$.
Since
$$
\Ext^{i}_{G\ltimes \cD_{X}}(M,M)
=\Ext^{i}_{\cD_{X}}(\cO_{X}, \Res M)
=\Ext^{i}_{\cD_{X}}(\cO_{X}, \cO_{X}\otimes \CC G)$$
$$
=\Ext^{i}_{\cD_{X}}(\cO_{X}, \cO_{X})\otimes \CC G
=\H^{i}(X, \CC )\otimes \CC G,
$$
and $X$ is simply connected, we have
$$
\Ext^{1}_{G\ltimes \cD_{X}}(M,M)=0,
\text{ and }\Ext^{2}_{G\ltimes \cD_{X}}(M,M)=0 \text{ if }
\H^{2}(X, \CC )=0.
$$
Thus such deformation exists and is unique if $\H^{2}(X, \CC )=0$.

Now let $M_{c, \eta}$ be the deformation. Then $\KZ(M_{c, \eta})$
is a ${\mathscr H}_{\tau(c,\eta)}(G,X)$-module from Proposition
\ref{Heck} and it deforms flatly the module $\reg G$.
This implies ${\mathscr H}_{\tau(c,\eta)}(G,X)$
is flat over $\CC [[\tau]]$.

\begin{remark}
When $X$ is not simply connected, the theorem is still true 
under the assumption $\pi_2(X)\otimes \CC =0$ (i.e. $\H^2(\widetilde X,\CC )=0$, 
where $\widetilde X$ is the universal cover of $X$), and the
proof is contained in \cite{E1}. 
\end{remark}

\subsection{Example: the simplest case of double affine Hecke algebras}{\label{sec:Exp}}

Now let $G=\ZZ_{2}\ltimes \ZZ^{2}$ acting on $\CC $.
Then the conjugacy classes of reflection hyperplanes are four points:
$0, 1/2, 1/2+\eta/2, \eta/2$, where we suppose the lattice
in $\CC $ is $\ZZ\oplus \ZZ\eta$. Correspondingly, the presentation of $G$ is as follows:
$$
\text{\bf generators:}\,\,T_{1}, T_{2}, T_{3}, T_{4};\qquad
\text{\bf relations:}\,\,T_{1}T_{2}T_{3}T_{4}=1, T_{i}^2=1.
$$
Thus, the corresponding orbifold Hecke algebra is the following deformation of $\CC G$:
$$
\text{\bf generators:}\,\,T_{1}, T_{2}, T_{3}, T_{4};\qquad
\text{\bf relations:}\,\,T_{1}T_{2}T_{3}T_{4}=1, (T_{i}-p_{i})(T_{i}-q_{i})=0,
$$
where $p_{i}, q_{i}$ $(i=1, \ldots, 4)$,  are parameters.

If we renormalize the $T_i$, these relations turn into
$$
(T_{i}-t_{i})(T_{i}+t_{i}^{-1})=0,\quad T_{1}T_{2}T_{3}T_{4}=q,
$$
and we get the type $C^{\vee}C_{1}$ double affine Hecke algebra.
If we set three of the four $T_{i}$'s satisfying the undeformed relation $T_{i}^{2}=1$,
we get the double affine Hecke algebra of type $A_{1}$. 
More precisely, this algebra is generated by $T_{1}, \ldots, T_{4}$
with relations
$$
T_{2}^{2}=T_{3}^{2}=T_{4}^{2}=1,\quad 
(T_{1}-t)(T_{1}+t^{-1})=0,\quad T_{1}T_{2}T_{3}T_{4}=q.
$$

Another presentation of this algebra (which is more widely used) is as
follows. Let $E=\CC /\ZZ^{2}$, an elliptic curve with an $\ZZ_{2}$ action
defined by $z\mapsto -z$. Define {\em the partial braid group}
$$B=\pi_{1}^{\mathrm{orb}}(E\backslash\{0\}/\ZZ_{2}, x),$$
where $x$ is a generic point.
Notice that comparing to the usual braid group, we do not delete three of the four reflection points.
The generators of the group $\pi_1(E\setminus \{0\},x)$ (the fundamental group of a punctured 2-torus) are $X$ 
(corresponding to the $a$-cycle on the torus), $Y$  
(corresponding to the $b$-cycle on the torus) and 
$C$ (corresponding to the loop around $0$).
In order to construct $B$, which is an extension of $\ZZ_{2}$ by $\pi_1
(E\setminus \{0\},x)$, we introduce
an element $T$ s.t. $T^{2}=C$ (the half-loop around the puncture).
Then $X, Y, T$ satisfy the following relations:
$$
TXT=X^{-1},\quad T^{-1}YT^{-1}=Y^{-1},\quad
Y^{-1}X^{-1}YXT^{2}=1.
$$

The Hecke algebra of the partial braid group is then defined to be
the group algebra of $B$ plus an extra relation:
$(T-q_{1})(T+q_{2})=0$.

A common way to present this Hecke algebra is to renormalize 
the generators so that one has the following relations:
$$
TXT=X^{-1}, T^{-1}YT^{-1}=Y^{-1}, Y^{-1}X^{-1}YXT^{2}=q,
(T-t)(T+t^{-1})=0.
$$
This is Cherednik's definition for $\HH(q,t)$, 
the double affine Hecke algebra of type 
$A_{1}$ which depends on two parameters $q, t$.

There are two degenerations of the algebra $\HH(q,t)$.

1. {\bf The trigonometric degeneration.}

Set $Y=\be^{\hbar y}$, $q=\be^{\hbar}$, $t=\be^{\hbar c}$ and
$T=s\be^{\hbar cs}$, where $s\in \ZZ_{2}$ 
is the reflection.
Then $s, X, y$ satisfy the following relations modulo $\hbar$:
$$
s^{2}=1,\quad sXs^{-1}=X^{-1},\quad sy+ys=2c,\quad X^{-1}yX-y=1-2cs.
$$  
The algebra generated by $s, X, y$ with these relations
is called the type $A_{1}$ trigonometric Cherednik algebra. 
It is easy to show that it is isomorphic to the Cherednik algebra
$H_{1, c}(\ZZ_2,\CC ^{*})$, where $\ZZ_2$ acts on
$\CC ^*$ by $z\to z^{-1}$.

2. {\bf The rational degeneration.}

In the trigonometric Cherednik algebra, set $X=\be^{\hbar x}$
and $y=\hat{y}/\hbar$. Then $s, x, \hat{y}$ satisfy the following relations
modulo $\hbar$:
$$
s^{2}=1, sx=-xs, s\hat{y}=-\hat{y}s, \hat{y}x-x\hat{y}=1-2cs.
$$
The algebra generated by $s, x, \hat{y}$ with these relations
is the rational Cherednik algebra $H_{1, c}(\ZZ_2,\Bbb C)$
with the action of $\ZZ_2$ on $\CC$ is given by $z\to -z$. 

\subsection{Affine and extended affine Weyl groups}  {\label{sec:affw}}

Let $R=\{\alpha\}\subset \RR^{n}$ be a root system 
with respect to a nondegenerate symmetric bilinear form $(\cdot,\cdot)$ on 
$\RR^{n}$. We will assume that $R$ is reduced.  
Let $\{\alpha_{i}\}_{i=1}^{n}\subset R$ 
be the set of simple roots and 
$R_{+}$ (respectively $R_{-}$) be the set of positive (respectively
negative) roots. 
The coroots are denoted by 
$\alpha^{\vee}=2\alpha/(\alpha,\alpha)$.
Let $Q^{\vee}=\bigoplus_{i=1}^{n}\ZZ\alpha_{i}^{\vee}$ 
be the coroot lattice and 
$P^{\vee}=\bigoplus_{i=1}^{n}\ZZ 
\omega^{\vee}_{i}$ the coweight lattice, where 
$\omega_{i}^{\vee}$'s are the fundamental coweights, i.e., 
$(\omega^{\vee}_{i}, \alpha_{j})=\delta_{ij}$. 
Let $\theta$ be the maximal positive root, and assume that 
the bilinear form is normalized by the condition $(\theta,\theta)=2$.
Let $\bW$ be the Weyl group which is generated by the reflections
$s_{\alpha}$ ($\alpha\in R$).

By definition, the affine root system is
$$R^{a}=\{ \tilde{\alpha}= [\alpha,j]\in 
\RR^{n}\times\RR|
\text{ where } \alpha\in R, j\in \ZZ\}.
$$ 
The set of positive 
affine roots is 
$R^{a}_{+}=\{[\alpha,j]\,|\,j\in \ZZ_{>0}\}\cup 
\{[\alpha,0]\,|\,\alpha\in R_{+}\}$.
Define $\alpha_{0}=[-\theta, 1]$.  
We will identify $\alpha\in R$ with 
$\tilde{\alpha}=[\alpha,0]\in R^{a}$. 

For an arbitrary affine root $\tilde{\alpha}=[\alpha,j]$ and 
a vector $\tilde{z}=[z, \zeta]\in \RR^{n}\times \RR$, 
the corresponding affine reflection is defined as follows: 
$$
s_{\tilde{\alpha}}(\tilde{z})
=\tilde{z}-2\frac{(z, \alpha)}{(\alpha,\alpha)}\,\tilde{\alpha}
=\tilde{z}-(z,\alpha^{\vee})\,\tilde{\alpha}.
$$
The affine Weyl group $\aW$ is generated by the affine reflections
$\{s_{\tilde{\alpha}}\,|\,\tilde{\alpha}\in \widetilde{R}_{+}\}$, and
we have an isomorphism: 
$$\aW\cong \bW\ltimes Q^{\vee},$$ 
where the translation $\alpha^\vee\in Q^\vee$ is naturally
identified with the composition $s_{[-\alpha,1]}s_{\alpha}\in \aW$.

Define the extended affine Weyl group to be 
$\eW=\bW\ltimes 
P^{\vee}$ acting on $\RR^{n+1}$ via 
$b(\tilde{z})=[z,\zeta-(b,z)]$ for $\tilde{z}=[z,\zeta],\, b\in P^\vee$.
Then $\aW\subset\eW$. 
Moreover, 
$\aW$ is a normal subgroup of 
$\eW$ and 
$\eW/\aW=P^{\vee}/Q^{\vee}$. 
The latter group can be identified with the group 
$\Pi=\{\pi_r\}$ of the elements of $\eW$ permuting simple 
affine roots under their action in $\RR^{n+1}$. 
It is a normal commutative subgroup of 
$\Aut=\Aut(\mathrm{Dyn}^{a})$
($\mathrm{Dyn}^{a}$ denotes the affine Dynkin diagram).
The quotient $\Aut/\Pi$ is isomorphic to the group 
of the automorphisms preserving $\alpha_0$, i.e. the group 
$\Aut{\mathrm{Dyn}}$ of automorphisms of the finite Dynkin diagram.

\subsection{Cherednik's double affine Hecke algebra of a root system}{\label{sec:daha}}

In this subsection, we will give an explicit presentation of Cherednik's
DAHA for a root system, defined in Example \ref{chere}. 
This is done by giving an explicit presentation of the corresponding braid group (which is called {\it the elliptic
braid group}), and then imposing quadratic relations on the generators 
corresponding to reflections. 

For a root system $R$, let $m=2$ if $R$ is of type $D_{2k}$, $m=1$
if $R$ is of type $B_{2k}, C_{k}$, and otherwise $m=|\Pi|$.
Let $m_{ij}$ be the number of edges between vertex $i$ 
and vertex $j$ in the affine Dynkin diagram of $R^{a}$.
Let $X_{i}$ $(i=1, \ldots, n)$ be a family of 
pairwise commutative and algebraically independent elements.
Set
$$
X_{[b,j]}=\prod_{i=1}^{n}X_{i}^{\ell_{i}}q^{j}, \text{ where }
b=\sum_{i=1}^{n}\ell_{i}\omega_{i}\in P, j\in \ZZ/m\ZZ.
$$

For an element $\hat{w}\in \eW$, we can define an action on these 
$X_{[b, j]}$ by
$\hat{w}X_{[b, j]}=X_{\hat{w}[b,j]}$.

\begin{definition}[Cherednik]
{\em The double affine Hecke algebra} (DAHA) of the root system $R$,
denoted by $\HH$, is an algebra
defined over the field
$\CC _{q,t}= \CC (q^{1/m},t^{1/2})$, 
generated by  
$T_{i}, i=0,\ldots,n, \Pi, X_{b}, b\in P$, 
subject to the following relations:
\begin{enumerate}
\item $T_{i}T_{j}T_{i}\cdots=
T_{j}T_{i}T_{j}\cdots$, $m_{ij}$ factors each side;
\item $(T_{i}-t_{i})(T_{i}+t_{i}^{-1})=0$ for $i=0, \ldots, n$;
\item $\pi T_{i}\pi^{-1}=T_{\pi(i)}$, for $\pi \in \Pi$ 
and $i=0, \ldots, n$;
\item $\pi X_{b}\pi^{-1}=X_{\pi(b)}$, for $\pi\in \Pi, b\in P$;
\item $T_{i}X_{b}T_{i}=X_{b}X_{\alpha_{i}}^{-1}$, if
$i>0$ and $(b,\alpha_{i}^{\vee})=1$; 
$T_{i}X_{b}=X_{b}T_{i}$, if $i>0$ and $(b, \alpha_{i}^{\vee})=0$;
\item $T_{0}X_{b}T_{0}=X_{b-\alpha_{0}}$ if $(b, \theta)=-1$;
$T_{0}X_{b}=X_{b}T_{0}$ if $(b, \theta)=0$.
\end{enumerate}
\end{definition}

Here $t_i$ are parameters attached to simple affine roots 
(so that roots of the same length give rise to the same
parameters). 

{\em The degenerate double affine Hecke algebra 
(trigonometric Cherednik algebra)} $\HH_{\mathrm trig}$
is generated by the group algebra of $\eW$, $\Pi$ and pairwise
commutative
$y_{\tilde{b}}=\sum_{i=1}^{n}(b, \alpha^{\vee})y_{i}+u$ for 
$\tilde{b}=[b,u]\in P\times \ZZ$, with the following relations:
$$
s_{i}y_{b}-y_{s_{i}(b)}s_{i}=-k_{i}(b, \alpha_{i}^{\vee}),
\text{ for }i=1, \ldots, n, 
$$
$$
s_{0}y_{b}-y_{s_{0}(b)}s_{0}=k_{0}(b, \theta),\qquad
\pi_{r}y_{b}\pi_{r}^{-1}=y_{\pi_{r}(b)} \text{ for }\pi_{r}\in \Pi. 
$$

\begin{remark}
This degeneration can be obtained from the DAHA 
similarly to the case of $A_1$, which is described above.  
\end{remark}

\subsection{Algebraic flatness of Hecke algebras of 
polygonal Fuchsian groups}{\label{sec:flatF}}

Let $W$ be the Coxeter group of rank $r$ corresponding to a Coxeter
datum: 
$$m_{ij}\,  (i,j=1, \ldots, r, i\ne j), \text{ such that } 2\leq m_{ij}\leq\infty \text{ and } m_{ij}=m_{ji}.$$ 
So the group $W$ has generators $s_i$ $i=1, \ldots, r$, and defining relations 
$$
s_i^2=1,\ (s_is_j)^{m_{ij}}=1 \text{ if }m_{ij}\ne \infty.
$$
It has a sign character 
$\xi: W\to \lbrace{\pm 1\rbrace}$ given
by $\xi(s_i)=-1$. Denote by $W_+$ the kernel of $\xi$
(the even subgroup of $W$). It is generated by 
$a_{ij}=s_{i}s_{j}$ with relations:
$$
a_{ij}=a_{ji}^{-1},\quad a_{ij}a_{jk}a_{ki}=1,\quad a_{ij}^{m_{ij}}=1.
$$ 

We can deform the group algebra $\CC [W]$ as follows.
Define the algebra $A(W)$ with invertible generators 
$s_i$, and $t_{ij,k}$, $i,j=1, \ldots, r$, $k\in \ZZ_{m_{ij}}$ 
for $(i,j)$ such that $m_{ij}<\infty$ 
and defining relations 
$$
t_{ij,k}=t_{ji,-k}^{-1},\quad
s_i^2=1,\quad
[t_{ij,k},t_{i'j',k'}]=0,\quad
s_pt_{ij,k}=t_{ji,k}s_p,
$$
$$
\prod_{k=1}^{m_{ij}}(s_is_j-t_{ij,k})=0\text{ if }m_{ij}<\infty.
$$

Notice that if we set $t_{ij,k}=\exp(2\pi k\bi/m_{ij})$, 
we get $\CC [W]$.

Define also the algebra $A_+(W)$ over $\mathcal{R}:=\CC [t_{ij,k}]$
($t_{ij,k}=t_{ji,-k}^{-1}$) by generators 
$a_{ij}$, $i\ne j$ ($a_{ij}=a_{ji}^{-1}$), and relations 
$$
\prod_{k=1}^{m_{ij}}(a_{ij}-t_{ij,k})=0\text{ if }m_{ij}<\infty,
\quad
a_{ij}a_{jp}a_{pi}=1.
$$

If $w$ is a word in letters $s_i$, let $T_w$ be the corresponding
element of $A(W)$. Choose a function $w(x)$ which attaches
to every element $x\in W$, a reduced word $w(x)$
representing $x$ in $W$.

\begin{theorem}[Etingof, Rains, \cite{ER}]\label{spa} 
\begin{enumerate}
\item[(i)] The elements $T_{w(x)}$, $x\in W$, form a spanning set
in $A(W)$ as a left $\mathcal{R}$-module.
\item[(ii)]The elements $T_{w(x)}$, $x\in W_+$, form a spanning set
in $A_+(W)$ as a left $\mathcal{R}$-module.
\item[(iii)]The elements $T_{w(x)}$, $x\in W$, are linearly independent 
if $W$ has no finite parabolic subgroups of rank $3$.
\end{enumerate}
\end{theorem}

\begin{proof} We only give the proof of (i). Statement (ii) follows from (i).
Proof of (iii), which is quite nontrivial, can be found in \cite{ER} (it uses the geometry 
of constructible sheaves on the Coxeter complex of $W$).

Let us write the relation
$$
\prod_{k=1}^{m_{ij}}(s_is_j-t_{ij,k})=0
$$
as a deformed braid relation:
$$
s_js_is_j\ldots+S.L.T.=t_{ij}s_is_js_i\ldots+S.L.T., 
$$
where $t_{ij}=(-1)^{m_{ij}+1}t_{ij,1}\cdots t_{ij,m_{ij}}$, S.L.T.
mean ``smaller length terms'', 
and the products on both sides have length $m_{ij}$.
This can be done by multiplying the relation by $s_is_j\cdots$ 
($m_{ij}$ factors).

Now let us show that $T_{w(x)}$ span $A(W)$ over $\mathcal{R}$.
Clearly, $T_w$ for all words $w$ span $A(W)$. 
So we just need to take any word $w$ and
express $T_w$ via $T_{w(x)}$. 

It is well known from the theory of Coxeter groups (see e.g. \cite{B}) that 
using the braid relations, one can turn any non-reduced word
into a word that is not square free, and any
reduced expression of a given element 
of $W$ into any other reduced expression of the same element.   
Thus, if $w$ is non-reduced, then by using the deformed braid relations
we can reduce $T_w$ to a linear combination of $T_u$ with words $u$ of
smaller length than $w$. 
On the other hand, if $w$ is a reduced expression for some
element $x\in W$, then using the deformed braid relations 
we can reduce $T_w$ to
a linear combination of $T_u$ with $u$ shorter than $w$, and
$T_{w(x)}$.
Thus $T_{w(x)}$ are a spanning set.
This proves (i). 
\end{proof}

Thus, $A_+(W)$ is a ``deformation'' of $\CC [W_+]$ over $\mathcal{R}$, and  
similarly $A(W)$ is a ``twisted deformation'' of $\CC [W]$.

Now let $\Gamma=\Gamma(m_1,\ldots,m_r)$, 
$r\ge 3$, be the Fuchsian group 
defined by generators $c_j$, 
$j=1,\ldots,r$, with defining relations 
$$
c_j^{m_j}=1, 
\ \prod_{j=1}^r c_j=1.
$$
Here $2\le m_j<\infty$. 

Suppose $\Gamma$ acts on $H$ where $H$ is a simply connected complex Riemann surface as in Section \ref{sec:HaF}.
We have the Hecke algebra 
of $\Gamma$, ${\mathscr H}_{\tau}(H, \Gamma)$, 
defined by the same (invertible) generators $c_{j}$ and 
relations 
$$
\prod_{k}(c_j-\exp(2\pi\mathbf{i} k/n_{j})q_{jk})=0,\ 
\prod_{j=1}^r c_j=1,
$$
where $q_{jk}=\exp(\tau_{jk})$.

We saw above (Theorem \ref{p2}) that 
if $\tau_{jk}$'s are formal, the algebra
${\mathscr H}_{\tau}(\Gamma,H)$ is flat
in $\tau$ if $|\Gamma|$ is infinite (i.e., $H$ is
Euclidean or hyperbolic). 
Here is a much stronger non-formal 
version of this theorem. 

\begin{theorem} \label{algver}
The algebra ${\mathscr H}_{\tau}(\Gamma,H)$ is free as a left module 
over $R:=\CC [q_{jk}^{\pm 1}]$
if and only if $\sum_j (1-1/m_j)\ge 2$ (i.e., $H$ is
Euclidean or hyperbolic). 
\end{theorem}

\begin{proof} Let us consider the Coxeter datum: $m_{ij}$, 
$i,j=1, \ldots, r$, 
such that $m_{i,i+1}:=m_i$ ($i\in \ZZ/r\ZZ$), 
and $m_{ij}=\infty$ otherwise. 
Suppose the corresponding Coxeter group is $W$.
Then we can see that $\Gamma=W_{+}$.
Notice that the algebra $\mathscr{H}_{\tau}(\Gamma,H)$ for genus
$0$ orbifolds is the algebra $A_{+}(W)$, 
i.e., we have
$\mathscr{H}_{\tau}(\Gamma,H)=A_{+}(W)$.

The condition
$\sum_j (1-1/m_j)\ge 2$ is equivalent to the condition 
that $W$ has no finite parabolic subgroups of rank $3$.
From Theorem \ref{spa} (ii) and Theorem \ref{p2}, 
we can see that $A_{+}(W)$ is  
free as a left module over $R$. 
We are done. 
\end{proof}

\subsection{Notes}
Section \ref{sec:dpez} follows Section 6 of the paper \cite{EOR};
Cherednik's definition of the double affine Hecke algebra of a root system
is from Cherednik's book \cite{Ch}; Sections \ref{sec:HaF} 
and \ref{sec:flatF} follow the paper \cite{ER}; 
The other parts of this section follow the paper \cite{E1}.

\newpage \section{Symplectic reflection algebras}

\subsection{The definition of symplectic reflection algebras}

Rational Cherednik algebras for finite Coxeter groups are 
a special case of a wider class of algebras called 
symplectic reflection algebras. To define them, let 
$V$ be a finite dimensional symplectic vector space 
over $\CC $ with a symplectic form $\omega$, and $G$
be a finite group acting symplectically (linearly)
on $V$. For simplicity let us assume that $(\wedge^2V^*)^G=\CC \omega$
(i.e., $V$ is symplectically irreducible) 
and that $G$ acts faithfully on $V$ (these assumptions are not 
important, and essentially not restrictive).  

\begin{definition}
A symplectic reflection in $G$ is an element $g$ such that the rank of the operator
$1-g$ on $V$ is $2$. 
\end{definition}

If $s$ is a symplectic reflection, then let $\omega_s(x,y)$ be 
the form $\omega$ applied to the projections of $x,y$ 
to the image of $1-s$ along the kernel of $1-s$; thus $\omega_s$ is a skewsymmetric form of rank $2$ on $V$.   

Let $\cS\subset G$ be the set of symplectic reflections, and 
$c: \cS\to \CC $ be a function which is invariant under the action of $G$. 
Let $t\in \CC $. 

\begin{definition} The symplectic reflection algebra 
$\sH_{t,c}=\sH_{t,c}[G,V]$
is the quotient of the algebra $\CC [G]\ltimes \mathbf T(V)$ 
by the ideal generated by the relation 
\begin{equation}\label{mainrel}
[x,y]=t\omega(x,y)-2\sum_{s\in \cS}c_s\omega_s(x,y)s.
\end{equation}
\end{definition}

\begin{example}
Let $W$ be a finite Coxeter group with reflection representation $\h$.
We can set $V=\h\oplus \h^*$, 
$\omega(x,x')=\omega(y,y')=0$, $\omega(y,x)=(y,x)$, 
for $x,x'\in \h^*$ and $y,y'\in \h$. In this case
\begin{enumerate}
\item symplectic reflections are the usual reflections in $W$;
\item $\omega_s(x,x')=\omega_s(y,y')=0$, 
$\omega_s(y,x)=(y,\alpha_s)(\alpha_s^\vee,x)/2$. 
\end{enumerate}
Thus, $\sH_{t,c}[G,\h\oplus\h^*]$ coincides with the rational Cherednik
algebra $H_{t,c}(G,\h)$ defined in Section \ref{rcera}. 
\end{example}

\begin{example} Let $\Gamma$ be a finite subgroup of $SL_2(\Bbb C)$,
and $V=\Bbb C^2$ be the tautological representation, with its standard symplectic form. 
Then all nontrivial elements of $\Gamma$ are symplectic reflections,
and for any symplectic reflection $s$, $\omega_s=\omega$. 
So the main commutation relation of $\sH_{t,c}[\Gamma,V]$ takes the form 
$$
[y,x]=t-\sum_{g\in \Gamma,g\ne 1}2c_gg.
$$ 
\end{example}

\begin{example}\label{wp} (Wreath products) 
Let $\Gamma$ be as in the previous example, 
$G=\kS_n\ltimes \Gamma^n$, and  
$V=(\Bbb C^2)^n$. Then symplectic reflections 
are conjugates $(g,1,...,1)$, $g\in \Gamma$, $g\ne 1$, and 
also conmjugates of transpositions in $\kS_n$ (so there is one more conjugacy class 
of reflections than in the previous example). 
\end{example}

Note also that for any $V,G$, $\sH_{0,0}[G,V]=G\ltimes SV$, and $\sH_{1,0}[G,V]=G\ltimes {\mathrm Weyl}(V)$, where ${\mathrm Weyl}(V)$ is the 
Weyl algebra of $V$, i.e. the quotient of the tensor algebra $\mathbf T(V)$ by the relation 
$xy-yx=\omega(x,y)$, $x,y\in V$. 

\subsection{The PBW theorem for symplectic reflection algebras}

To ensure that the symplectic reflection algebras $\sH_{t,c}$ have good properties, 
we need to prove a PBW theorem for them, which is an analog of Proposition \ref{prop-pbw}. 
This is done in the following theorem, which also explains the special role 
played by symplectic reflections. 

\begin{theorem}\label{pbw1}
Let $\kappa: \wedge^2V\to \CC [G]$ be a linear $G$-equivariant function. 
Define the algebra $\sH_\kappa$ to be the quotient of the algebra 
$\CC [G]\ltimes {\mathbf T}(V)$ by the relation 
$[x,y]=\kappa(x,y)$, $x,y\in V$. Put an increasing filtration on $\sH_\kappa$ 
by setting $\deg(V)=1$, $\deg(G)=0$, and define $\xi: \CC  G\ltimes SV\to \gr\sH_\kappa$ 
to be the natural surjective homomorphism. Then $\xi$ is an isomorphism if and only if 
$\kappa$ has the form 
$$
\kappa(x,y)=
t\omega(x,y)-2\sum_{s\in \cS}c_s\omega_s(x,y)s,
$$
for some $t\in \CC $ and $G$-invariant function $c: \cS\to \CC $. 
\end{theorem}

Unfortunately, for a general symplectic reflection algebra 
we don't have a Dunkl operator representation, so
the proof of the more difficult ``if'' part of this Theorem 
is not as easy as the proof of Proposition \ref{prop-pbw}. 
Instead of explicit computations with Dunkl operators, it makes use 
of the deformation theory of Koszul algebras, which we will now discuss. 

\subsection{Koszul algebras}

Let $R$ be a finite dimensional semisimple algebra (over $\CC $). 
Let $A$ be a $\ZZ_+$-graded algebra, such that $A[0]=R$, and
assume that the graded components of $A$ are finite dimensional. 

\begin{definition} 
\begin{enumerate}
\item[(i)] The algebra $A$ is said to be quadratic if 
it is generated over $R$ by $A[1]$, and has 
defining relations in degree 2.
\item[(ii)] $A$ is Koszul if all elements of $\Ext^i(R,R)$ 
(where $R$ is the augmentation module over $A$)
have grade degree precisely $i$. 
\end{enumerate}

\end{definition}

\begin{remark}
\begin{enumerate}
\item Thus, in a quadratic algebra, 
$A[2]=A[1]\otimes_R A[1]/E$, where $E$ is the 
subspace ($R$-subbimodule) of relations. 

\item It is easy to show that 
a Koszul algebra is quadratic, since 
the condition to be quadratic is just the Koszulity condition 
for $i=1,2$.
\end{enumerate}

\end{remark}

Now let $A_0$ be a quadratic algebra, $A_0[0]=R$. 
Let $E_0$ be the space of relations for $A_0$.
Let $E\subset A_0[1]\otimes_R A_0[1][[\hbar]]$ 
be a free (over $\CC [[\hbar]]$) $R$-subbimodule which 
reduces to $E_0$ modulo $\hbar$ (``deformation of the relations''). 
Let $A$ be the ($\hbar$-adically complete) 
algebra generated over $R[[\hbar]]$ by $A[1]=A_0[1][[\hbar]]$ 
with the space of defining relations $E$. 
Thus $A$ is a $\ZZ_+$-graded algebra. 

The following very important theorem 
is due to Beilinson, Ginzburg, and Soergel, \cite{BGS} 
(less general versions appeared earlier 
in the works of Drinfeld \cite{Dr}, Polishchuk-Positselski \cite{PP}, 
Braverman-Gaitsgory \cite{BG}). We will not give its proof.  

\begin{theorem}[Koszul deformation principle]\label{KDP} 
If $A_0$ is Koszul then $A$ is a topologically free $\CC [[\hbar]]$ module
if and only if so is $A[3]$.  
\end{theorem} 

{\bf Remark.} Note that $A[i]$ for $i<3$ are obviously topologically free. 

We will now apply this theorem to the proof of Theorem \ref{pbw1}. 

\subsection{Proof of Theorem \ref{pbw1}}

Let $\kappa: \wedge^2V\to \CC [G]$ be a linear $G$-equivariant map. 
We write $\kappa(x,y)=\sum_{g\in G}\kappa_g(x,y)g$, where $\kappa_g(x,y)\in \wedge^2V^*$.
To apply Theorem \ref{KDP}, let us homogenize our algebras. 
Namely, let $A_0=(\CC G\ltimes SV)\otimes \CC [u]$ (where $u$
has degree $1$). Also 
let $\hbar$ be a formal parameter, and 
consider the deformation $A=\sH_{\hbar u^2\kappa}$ of 
$A_0$. That is, $A$ is the quotient of 
$G\ltimes {\mathbf T}(V)[u][[\hbar]]$ by the relations $[x,y]=\hbar u^2\kappa(x,y)$. 
This is a deformation of the type considered in Theorem \ref{KDP}, and 
it is easy to see that its flatness 
in $\hbar$ is equivalent to Theorem \ref{pbw1}. 
Also, the algebra $A_0$ is Koszul, because the polynomial algebra 
$SV$ is a Koszul algebra. 
Thus by Theorem \ref{KDP}, it suffices to show that 
$A$ is flat in degree 3. 

The flatness condition in degree 3 is ``the Jacobi identity''
$$
[\kappa(x,y),z]+[\kappa(y,z),x]+[\kappa(z,x),y]=0,
$$
which must be satisfied in $\CC G\ltimes V$. In components, 
this equation transforms into the system of equations
$$
\kappa_g(x,y)(z-z^g)+\kappa_g(y,z)(x-x^g)+\kappa_g(z,x)(y-y^g)=0
$$
for every $g\in G$ (here $z^g$ denotes the result of the action of 
$g$ on $z$). 

This equation, in particular, implies that if $x,y,g$ are such
that $\kappa_g(x,y)\ne 0$ then for any $z\in V$ 
$z-z^g$ is a linear combination of $x-x^g$ and $y-y^g$. 
Thus $\kappa_g(x,y)$ is identically zero unless the 
rank of $(1-g)|_V$ is at most 2, i.e. 
$g=1$ or $g$ is a symplectic reflection. 

If $g=1$ then $\kappa_g(x,y)$ has to be $G$-invariant, 
so it must be of the form $t\omega(x,y)$, where $t\in \CC $. 

If $g$ is a symplectic reflection, then
$\kappa_g(x,y)$ must be zero for any $x$ such that $x-x^g=0$. 
Indeed, if for such an $x$ there had existed $y$ with 
$\kappa_g(x,y)\ne 0$ then $z-z^g$ for any $z$ would be a multiple of $y-y^g$, which is impossible since 
$\mathrm{Im}(1-g)|_V$ is 2-dimensional. 
This implies that $\kappa_g(x,y)=
2c_g\omega_g(x,y)$, and $c_g$ 
must be invariant. 

Thus we have shown that if $A$ is flat (in degree 3) then $\kappa$ must have the form 
given in Theorem \ref{pbw1}. Conversely, it is easy to see that 
if $\kappa$ does have such form, then the Jacobi identity holds. 
So Theorem \ref{pbw1} is proved.

\subsection{The spherical subalgebra of the symplectic reflection algebra}
{\label{sec:ssra}}

The properties of symplectic reflection algebras are similar to
the properties 
of rational Cherednik algebras we have studied before. 
The main difference is that we no longer have the Dunkl representation and localization results, 
so some proofs are based on different ideas and are more complicated. 

The spherical subalgebra of the symplectic reflection algebra is defined 
in the same way as in the Cherednik algebra case. 
Namely, let $\e=|G|^{-1}\sum_{g\in G}g$, and $\sB_{t,c}=\e \sH_{t,c}\e$. 

\begin{proposition} 
$\sB_{t,c}$ is commutative if and only if 
$t=0$. 
\end{proposition}

\begin{proof}
Let $A$ be a $\ZZ_+$-filtered algebra. If $A$ is not commutative, 
then we can define a nonzero Poisson bracket on $\gr A$ in the following way. 
Let $m$ be the minimum of $\deg(a)+\deg(b)-\deg([a,b])$ (over $a,b\in A$ such that $[a,b]\ne 0$). 
Then for homogeneous elements $a_0,b_0\in A_0$ of degrees $p,q$, we can define 
$\lbrace{a_0,b_0\rbrace}$ to be the image in $A_0[p+q-m]$ of $[a,b]$, where $a,b$ are any lifts 
of $a_0,b_0$ to $A$. It is easy to check that $\{\cdot\, ,\cdot\}$ is a Poisson bracket 
on $A_0$ of degree $-m$. 

Let us now apply this construction to the filtered algebra $A=\sB_{t,c}$.
We have $\gr(A)=A_0=(SV)^G$. 

\begin{lemma}\label{poisbr} $A_0$ has a unique, up to scaling, Poisson bracket 
of degree $-2$, and no nonzero Poisson brackets of degrees $<-2$. 
\end{lemma}

\begin{proof}
A Poisson bracket on $(SV)^G$ is the same thing as a Poisson bracket on the variety $V^*/G$. 
On the smooth part $(V^*/G)_s$ of $V^*/G$, it is simply a bivector field, and we can lift it 
to a bivector field on the preimage $V_s^*$ of $(V^*/G)_s$ in $V^*$, which 
is the set of points in $V$ with trivial stabilizers. 
But the codimension on $V^*\setminus V_s^*$ in $V^*$ is 2 (as $V^*\setminus V_s^*$ is a union 
of symplectic subspaces), so the bivector on $V_s^*$ extends to a regular bivector on $V^*$. 
So if this bivector is homogeneous, it must have degree $\ge -2$, and if it has
degree $-2$ then it must be with constant coefficients, so being $G$-invariant, it 
is a multiple of $\omega$. The lemma is proved.   
\end{proof}

Now, for each $t,c$ we have a natural Poisson bracket on 
$A_0$ of degree $-2$, which depends linearly on $t,c$. 
So by the lemma, this bracket has to be of the form $f(t,c)\Pi$, 
where $\Pi$ is the unique up to scaling 
Poisson bracket of degree $-2$, and $f$ a homogeneous linear function. 
Thus the algebra $A=\sB_{t,c}$ is not commutative unless $f(t,c)=0$. 
On the other hand, if $f(t,c)=0$, and $\sB_{t,c}$ is not commutative, then, as we've shown, $A_0$ 
has a nonzero Poisson bracket of degree $<-2$. But By Lemma \ref{poisbr}, there is no such brackets. 
So $\sB_{t,c}$ must be commutative if $f(t,c)=0$.

It remains to show that $f(t,c)$ is in fact a nonzero multiple of $t$. 
First note that $f(1,0)\ne 0$, since $\sB_{1,0}$ is definitely noncommutative. 
Next, let us take a point $(t,c)$ such that $\sB_{t,c}$ is commutative. 
Look at the $\sH_{t,c}$-module $\sH_{t,c}\e$, which has a commuting action of $\sB_{t,c}$ from the right. 
Its associated graded is $SV$ as an $(\CC  G\ltimes SV,(SV)^G)$-bimodule, which implies that 
the generic fiber of $\sH_{t,c}\e$ as a $\sB_{t,c}$-module is the regular representation of $G$. 
So we have a family of finite dimensional representations of $\sH_{t,c}$ 
on the fibers of $\sH_{t,c}\e$, all realized in the regular representation 
of $G$. Computing the trace of the main commutation relation 
(\ref{mainrel}) of $\sH_{t,c}$ in this representation, we obtain
that $t=0$ (since $\tr(s)=0$ for any reflection $s$). The proposition is proved. 
\end{proof}

Note that $\sB_{0,c}$ has no zero divisors, since its associated graded algebra $(SV)^G$ does not. 
Thus, like in the Cherednik algebra case, we can define 
a Poisson variety $\sM_c$, the spectrum of $\sB_{0,c}$, called the Calogero-Moser space of $G,V$. 
Moreover, the algebra $\sB_c:=\sB_{\hbar,c}$ over $\CC [\hbar]$ is an algebraic quantization of $\sM_c$. 

\subsection{The center of the symplectic reflection algebra $\sH_{t,c}$}

Consider the bimodule $\sH_{t,c}\e$, which has a left action of 
$\sH_{t,c}$ and a right 
commuting action of $\sB_{t,c}$. It is obvious that 
$\End_{\sH_{t,c}}\sH_{t,c}\e=\sB_{t,c}$ (with opposite product). The following theorem shows that 
the bimodule $\sH_{t,c}\e$ has the double centralizer property
(i.e., $\End_{\sB_{t,c}}\sH_{t,c}\e=\sH_{t,c}$). 

Note that we have a natural map $\xi_{t,c}: \sH_{t,c}\to 
\End_{\sB_{t,c}}\sH_{t,c}\e$. 

\begin{theorem}\label{doubcen}
$\xi_{t,c}$ is an isomorphism for any $t,c$. 
\end{theorem}

\begin{proof}
The complete proof is given \cite{EG}.
We will give the main ideas of the proof skipping straightforward technical details.
The first step is to show that the result is true in the graded case, $(t,c)=(0,0)$.
To do so, note the following easy lemma:

\begin{lemma}{\label{lem:iso}}
If $X$ is an affine complex algebraic variety with algebra 
of functions $\cO_X$
and $G$ a finite group acting freely on $X$ then 
the natural map $\xi_X: G\ltimes \cO_X\to 
\End_{\cO_X^G}\cO_X$ is an isomorphism.
\end{lemma}

Therefore, the map $\xi_{0,0}: G\ltimes SV\to \End_{(SV)^G}(SV)$ 
is injective, and moreover becomes an isomorphism after localization 
to the field of quotients $\CC (V^*)^G$. To show it's surjective, take $a\in \End_{(SV)^G}(SV)$. 
There exists $a'\in G\ltimes \CC (V^*)$ which maps to $a$. Moreover, by Lemma \ref{lem:iso},
$a'$ can have poles only at fixed points of $G$ on $V^*$. 
But these fixed points form a subset of codimension $\ge 2$, so there can be no poles and we are done
in the case $(t,c)=(0,0)$. 

Now note that the algebra $\End_{\sB_{t,c}}\sH_{t,c}\e$
has an increasing integer filtration (bounded below) induced by the filtration on $\sH_{t,c}$. 
This is due to the fact that $\sH_{t,c}\e$ is a finitely generated 
$\e \sH_{t,c}\e$-module 
(since it is true in the associated graded situation, by Hilbert's 
theorem about invariants). Also, the natural map
$\gr\End_{\sB_{t,c}}\sH_{t,c}\e\to 
\End_{\gr\sB_{t,c}}\gr\sH_{t,c}\e$
is clearly injective. Therefore, our result in the case $(t,c)=(0,0)$ implies that
this map is actually an isomorphism (as so is its composition with the associated graded of $\xi_{t,c}$). 
Identifying the two algebras by this isomorphism, 
we find that $\gr(\xi_{t,c})=\xi_{0,0}$. Since $\xi_{0,0}$ is an isomorphism, 
$\xi_{t,c}$ is an isomorphism for all $t,c$, as desired.
\footnote{Here we use the fact that the filtration is bounded from below. 
In the case of an unbounded filtration, it is possible for a map not to be an isomorphism 
if its associated graded is an isomorphism. An example of this is the operator 
of multiplication by $1+t^{-1}$ in the space of Laurent polynomials in $t$, filtered 
by degree.}   
\end{proof} 

Denote by $\sZ_{t,c}$ the center of the symplectic reflection 
algebra $\sH_{t,c}$. We have the following theorem.
\begin{theorem}
If $t\ne 0$, the center of $\sH_{t,c}$ is trivial.
If $t=0$, we have $\gr \sZ_{0,c}=\sZ_{0,0}$.
In particular, $\sH_{0,c}$ is finitely generated over its center.
\end{theorem}

\begin{proof}
The $t\neq 0$ case was proved by Brown and Gordon \cite{BGo}
as follows.
If $t\neq 0$, we have $\gr \sZ_{t,c}\subseteq \sZ_{0,0}=(SV)^{G}$.
Also, we have a map 
$$
\tau_{t,c}:\sZ_{t,c}\to \sB_{t,c}=\e \sH_{t,c}\e,
\quad z\mapsto z\e=\e z\e.$$
The map $\tau_{t, c}$ is injective since $\gr (\tau_{t,c})$ is injective.
In particular, the image of $\gr(\tau_{t,c})$ is contained in $Z(\sB_{t,c})$, 
the center of $\sB_{t,c}$. 
Thus it is enough to show that $Z(\sB_{t,c})$ is trivial.
To show this, note that $\gr Z(\sB_{t,c})$ is contained in the Poisson center 
of $\sB_{0,0}$ which is trivial.
So $Z(\sB_{t,c})$ is trivial.

Now suppose $t=0$. 
We need to show that 
$
\gr(\tau_{0,c}): \gr(Z_{0,c})\to Z_{0,0}
$
is an isomorphism.
It suffices to show that $\tau_{0,c}$ is an isomorphism.
To show this, we construct $\tau^{-1}_{0,c}:\sB_{0,c}\to \sZ_{0,c}$
as follows.

For any $b\in \sB_{0,c}$, since $\sB_{0,c}$ is commutative,
we have an element $\tilde{b}\in \End_{\sB_{0,c}}(\sH_{0,c}\e)$
which is defined as the right multiplication by $b$.
From Theorem \ref{doubcen}, 
$\tilde{b}\in \sH_{0,c}$.
Moreover, $\tilde{b}\in \sZ_{0,c}$ since it commutes with 
$\sH_{0,c}$ as a linear operator on the faithful 
$\sH_{0,c}$-module $\sH_{0,c}\e$.
So $\tilde{b}\in \sZ_{0,c}$.
It is easy to see that $\tilde{b}\e=b$.
So we can set $\tilde{b}=\tau_{0,c}^{-1}(b)$
which defines the inverse map to $\tau_{0,c}$.
\end{proof}

\subsection{A review of deformation theory}

Now we would like to explain that symplectic reflection algebras are 
the most general deformations of algebras of the from $G\ltimes {\mathrm Weyl}(V)$. 
Before we do so, we give a brief review of classical deformation theory of associative algebras. 

\subsubsection{Formal deformations of associative algebras}
Let $A_0$ be an associative algebra with unit over $\CC$.
Denote by $\mu_{0}$ the multiplication in $A_{0}$. 

\begin{definition} A (flat) formal
$n$-parameter deformation of $A_0$ is an algebra $A$ 
over $\CC[[\hbar]]=\CC[[\hbar_1, \ldots, \hbar_n]]$ which is topologically free as a $\CC[[\hbar]]$-module, together
with an algebra isomorphism $\eta_0: A/{\mathfrak{m}} A\to A_0$
where $\mathfrak{m} = \langle\hbar_1,\ldots,\hbar_n\rangle$ 
is the maximal ideal in $\CC[[\hbar]]$. 
\end{definition}

When no confusion is possible, we will call $A$ a
deformation of $A_0$ (omitting ``formal''). 

Let us restrict ourselves to one-parameter deformations with
parameter $\hbar$. Let us choose an identification $\eta: A\to A_0[[\hbar]]$ as 
$\CC[[\hbar]]$-modules, 
such that $\eta=\eta_0$ modulo $\hbar$. Then the
product in $A$ is completely determined by the product of
elements of $A_0$, which has the form of a ``star-product''
$$
\mu(a,b)=a\ast b=\mu_0(a,b)+\hbar \mu_1(a,b)+\hbar^2 \mu_2(a,b)+\cdots,
$$
where $\mu_i: A_0\otimes A_0\to A_0$ are linear maps, and
$\mu_0(a,b)=ab$. 

\subsubsection{Hochschild cohomology}
The main tool in deformation theory of associative algebras is Hochschild cohomology.
Let us recall its definition. 

Let $A$ be an associative algebra. 
Let $M$ be a bimodule over $A$.  
A Hochschild $n$-cochain of $A$ with coefficients in $M$ 
is a linear map $A^{\otimes n}\to M$. 
The space of such cochains is denoted by $C^n(A,M)$. 
The differential $\d:C^n(A,M)\to C^{n+1}(A,M)$ is defined 
by the formula 
\begin{eqnarray*}
\d f(a_1,\ldots,a_{n+1})
&=&f(a_1,\ldots,a_n)a_{n+1}-f(a_1,\ldots,a_na_{n+1})
+f(a_1,\ldots,a_{n-1}a_n,a_{n+1})\\
&&\quad-\cdots+(-1)^nf(a_1a_2,\ldots,a_{n+1})+(-1)^{n+1}a_1f(a_2,\ldots,a_{n+1}).
\end{eqnarray*}
It is easy to show that $\d^2=0$. 

\begin{definition}
The Hochschild cohomology $\hoc^\bullet(A,M)$ is defined  
to be the cohomology of the complex 
$(C^\bullet(A,M), \d)$. 
\end{definition}

\begin{proposition}
One has a natural isomorphism 
$$
\hoc^i(A,M)\to \mathrm{Ext}^i_{A\mathrm{-bimod}}(A,M),
$$ 
where $A\mathrm{-bimod}$ denotes the category of $A$-bimodules. 
\end{proposition}

\begin{proof}
The proof is obtained immediately by considering the bar resolution 
of the bimodule $A$:
$$
\cdots\to A\otimes A\otimes A\to A\otimes A\to A,
$$
where the bimodule structure on $A^{\otimes n}$ is given by 
$$
b(a_1\otimes a_2\otimes \cdots\otimes a_n)c=ba_1\otimes
a_2\otimes\cdots\otimes a_nc,
$$
and the map $\partial_n: A^{\otimes n}\to A^{\otimes {n-1}}$ is given by the formula 
$$
\partial_n(a_1\otimes a_2\otimes...\otimes a_n)=a_1a_2\otimes\cdots
\otimes a_n-\cdots+(-1)^{n}a_1\otimes\cdots\otimes a_{n-1}a_n.
$$ 
\end{proof}

Note that we have the associative 
Yoneda product 
$$
\hoc^i(A,M)\otimes \hoc^j(A,N)\to \hoc^{i+j}(A,M\otimes_A N),
$$ 
induced by tensoring of cochains.   

If $M=A$, the algebra itself, then we will 
denote $\hoc^\bullet(A,M)$ by $\hoc^\bullet(A)$. 
For example, $\hoc^0(A)$ is the center of $A$, and 
$\hoc^1(A)$ is the quotient of the Lie algebra of derivations of $A$ by inner derivations.  
The Yoneda product induces a graded algebra 
structure on $\hoc^\bullet(A)$; it can be shown that 
this algebra is supercommutative. 

\subsubsection{Hochschild cohomology and deformations}

Let $A_0$ be an algebra, and let us look 
for 1-parameter deformations $A=A_0[[\hbar]]$ of $A_0$. 
Thus, we look for such series $\mu$ which satisfy the associativity 
equation, modulo the automorphisms of the 
$\CC[[\hbar]]$-module $A_0[[\hbar]]$
which are the identity modulo $\hbar$. 
\footnote{Note that we don't have to worry about the existence of a unit in 
$A$ since a formal deformation of an algebra with unit always has a 
unit.}  

The associativity equation 
$\mu\circ(\mu\otimes \Id)=\mu\circ (\Id\otimes \mu)$ 
reduces to a hierarchy of linear equations: 
\begin{equation*}
\sum_{s=0}^N \mu_s(\mu_{N-s}(a,b),c)=
\sum_{s=0}^N \mu_s(a,\mu_{N-s}(b,c)).
\end{equation*}
(These equations are linear in $\mu_N$ if $\mu_i$, $i<N$, are known).

To study these equations, one can use Hochschild cohomology. 
Namely, we have the following standard facts
(due to Gerstenhaber, \cite{Ge}), which can be
checked directly.  
\begin{enumerate}
\item The linear equation for $\mu_1$ says that $\mu_1$ is a Hochschild 
$2$-cocycle. Thus algebra structures on $A_0[\hbar]/\hbar^2$ 
deforming $\mu_0$ are parametrized by the space $Z^2(A_0)$ of 
Hochschild $2$-cocycles of $A_0$ with values in $M=A_0$.
\item If $\mu_1,\mu_1'$ are two $2$-cocycles such that 
$\mu_1-\mu_1'$ is a coboundary, then the algebra structures 
on $A_0[\hbar]/\hbar^2$ corresponding to $\mu_1$ and $\mu_1'$ are 
equivalent by a transformation of $A_0[\hbar]/\hbar^2$ 
that equals the identity 
modulo $\hbar$, and vice versa. Thus equivalence classes 
of multiplications on $A_0[\hbar]/\hbar^2$ deforming $\mu_0$
are parametrized by the cohomology $\hoc^2(A_0)$. 
\item The linear equation for $\mu_N$ says that $\d\mu_N$ 
is a certain quadratic expression $b_N$ in 
$\mu_1,\ldots,\mu_{N-1}$. 
This expression is always a Hochschild $3$-cocycle, and the equation is 
solvable if and only if it is a coboundary. Thus the cohomology class 
of $b_N$ in $\hoc^3(A_0)$ is the only obstruction to solving this equation. 
\end{enumerate}

\subsubsection{Universal deformation}

In particular, if $\hoc^3(A_0)=0$ then the equation for $\mu_n$ can be solved 
for all $n$, and for each $n$ the freedom in choosing the solution, 
modulo equivalences, is the space $H=\hoc^2(A_0)$. Thus there exists 
an algebra structure over $\CC[[H]]$ on the space $A_u:=A_0[[H]]$ 
of formal functions from $H$ to $A_0$, $a,b\mapsto \mu_u(a,b)\in A_0[[H]]$,
($a,b\in A_0$), such that $\mu_u(a,b)(0)=ab\in A_0$, and 
every 1-parameter flat formal deformation $A$ of $A_0$ is given 
by the formula $\mu(a,b)(\hbar)=\mu_u(a,b)(\gamma(\hbar))$
for a unique formal series $\gamma\in \hbar H[[\hbar]]$,
with the property that $\gamma'(0)$ is the cohomology class 
of the cocycle $\mu_1$. 

Such an algebra $A_u$ is called a {\em universal deformation} of
$A_0$. It is unique up to an isomorphism 
(which may involve an automorphism of $\CC[[H]]$). 
\footnote{In spite of the universal property of $A_u$, 
it may happen that there is an isomorphism between 
the algebras $A^1$ and $A^2$ corresponding to different paths
$\gamma_1(\hbar),\gamma_2(\hbar)$ (of course, reducing to a
nontrivial automorphism of $A_0$ modulo $\hbar$). For this
reason, sometimes $A_u$ is called a {\it semiuniversal}, rather
than universal, deformation of $A_0$.}

Thus in the case $\hoc^3(A_0)=0$, deformation theory allows us 
to completely classify 1-parameter flat formal deformations of 
$A_0$. In particular, we see that the 
``moduli space'' parametrizing formal deformations of $A_0$
is a smooth space -- it is the formal neighborhood of zero in $H$.

If $\hoc^3(A_0)$ is nonzero then in general the universal 
deformation parametrized by $H$ does not exist, as there are obstructions
to deformations. In this case, the moduli space of deformations 
will be a closed subscheme of the formal neighborhood of zero in $H$, which is often singular. 
On the other hand, even when $\hoc^3(A_0)\ne 0$, the universal 
deformation parametrized by (the formal neighborhood of zero in) $H$ may exist (although 
its existence may be more difficult to prove than in the vanishing case).
In this case one says that the deformations of $A_0$ are unobstructed
(since all obstructions vanish even though the space of obstructions doesn't).

\subsection{Deformation-theoretic interpretation of symplectic reflection algebras}

Let $V$ be a symplectic vector space (over $\CC $) and
${\mathrm Weyl}(V)$ the Weyl algebra of $V$.
Let $G$ be a finite group acting symplectically on $V$.
Then from the definition, we have 
$$A_{0}:=\sH_{1,0}[G,V]=G\ltimes {\mathrm Weyl}(V).$$
Let us calculate the Hochschild 
cohomology of this algebra.

\begin{theorem}[Alev, Farinati, Lambre, Solotar, \cite{AFLS}] \label{afls} 
The cohomology space \linebreak $\hoc^i(G\ltimes {\mathrm Weyl}(V))$ is 
naturally isomorphic to the space of 
conjugation invariant functions on the set $S_i$ 
of elements $g\in G$ such that $\rank(1-g)|_V=i$. 
\end{theorem}

\begin{corollary}\label{h2}
The odd cohomology of $G\ltimes {\mathrm Weyl}(V)$ vanishes, 
and $\hoc^2(G\ltimes {\mathrm Weyl}(V))$ is the space $\CC[\cS]^G$ of conjugation invariant 
functions on the set of symplectic reflections. 
In particular, there exists a universal deformation $A$ of $A_0=G\ltimes {\mathrm Weyl}(V)$
parametrized by $\CC [\cS]^G$.
\end{corollary}

\begin{proof}
Directly from the theorem.
\end{proof}

\begin{proof}[Proof of Theorem \ref{afls}]
\begin{lemma}
Let $B$ be a $\CC $-algebra together with an action of 
a finite group $G$. Then 
$$
\hoc^*(G\ltimes B,G\ltimes B)=(\bigoplus_{g\in G}\hoc^*(B,Bg))^G,
$$ 
where $Bg$ is the bimodule isomorphic to $B$ as a space where 
the left action of $B$ is the usual one and the right action is the usual action twisted by $g$. 
\end{lemma}

\begin{proof} The algebra $G\ltimes B$ is a projective $B$-module. 
Therefore, using the Shapiro lemma, we get
\begin{eqnarray*}
&&\hoc^*(G\ltimes B,G\ltimes B)=\Ext^*_{(G\times G)\ltimes (B\otimes B^{op})}(G\ltimes B,G\ltimes B)\\
&=&
\Ext^*_{G_{\mathrm diagonal}\ltimes (B\otimes B^{op})}(B,G\ltimes B)=
\Ext^*_{B\otimes B^{op}}(B,G\ltimes B)^G\\
&=&
(\bigoplus_{g\in G}\Ext^*_{B\otimes B^{op}}(B,Bg))^G=
(\bigoplus_{g\in G}\hoc^*(B,Bg))^G,
\end{eqnarray*}
as desired. 
\end{proof}

Now apply the lemma to $B={\mathrm Weyl}(V)$. 
For this we need to calculate $\hoc^*(B,Bg)$, where 
$g$ is any element of $G$. We may assume that 
$g$ is diagonal in some symplectic basis: 
$g=\diag(\lambda_1,\lambda_1^{-1},\ldots,\lambda_n,\lambda_n^{-1})$. 
Then by the K\"unneth formula we find that 
$$\hoc^*(B,Bg)=\bigotimes_{i=1}^n \hoc^*(\mathbb A_1,\mathbb A_1g_i),$$ where 
$\mathbb A_1$ is the Weyl algebra of the $2$-dimensional space, 
(generated by $x,y$ with defining relation $xy-yx=1$), and 
$g_i=\diag(\lambda_i,\lambda_i^{-1})$. 

Thus we need to calculate 
$\hoc^*(B,Bg)$, $B=\mathbb A_1$, $g=\diag(\lambda,\lambda^{-1})$. 

\begin{proposition}
$\hoc^*(B,Bg)$ is 1-dimensional, concentrated in degree 
$0$ if $\lambda=1$ and in degree $2$ otherwise. 
\end{proposition}

\begin{proof} 
If $B=\mathbb A_1$ then $B$ has the following Koszul resolution 
as a B-bimodule:
$$
B\otimes B\to B\otimes \CC^2\otimes B\to B\otimes B\to B.
$$
Here the first map is given by the formula 
\begin{eqnarray*}
b_1\otimes b_2&\mapsto& b_1\otimes x\otimes yb_2-b_1\otimes y\otimes xb_2
-b_1y\otimes x\otimes b_2+b_1x\otimes y\otimes b_2,
\end{eqnarray*} 
the second map is given by 
$$
b_1\otimes x\otimes b_2\mapsto b_1x\otimes b_2-b_1\otimes xb_2,\quad 
b_1\otimes y\otimes b_2\mapsto b_1y\otimes b_2-b_1\otimes yb_2,
$$
and the third map is the multiplication. 

Thus the cohomology of $B$ with coefficients in $Bg$ can be computed by mapping 
this resolution into $Bg$ and taking the cohomology. 
This yields the following complex $C^\bullet$: 
\begin{equation}\label{comple}
0\to Bg\to Bg\oplus Bg\to Bg\to 0,
\end{equation}
where the first nontrivial map is given by 
$bg\mapsto [bg,y]\otimes x-[bg,x]\otimes y$, and the second nontrivial map is 
given by $bg\otimes x\mapsto [x,bg]$, $bg\otimes y\mapsto [y,bg]$. 

Consider first the case $g=1$. Equip the complex $C^\bullet$ with the 
Bernstein filtration ($\deg(x)=\deg(y)=1$), starting with $0,1,2$, for $C^0,C^1,C^2$, respectively
(this makes the differential preserve the filtration). 
Consider the associated graded complex $C_{\gr}^\bullet$. In this complex, brackets are replaced with 
Poisson brackets, and thus it is easy to see that $C_{\gr}^\bullet$ is the De Rham complex for 
the affine plane, so its cohomology is $\CC$ in degree 0 and 0 in other
degrees. Therefore, the cohomology of $C^\bullet$ is the same. 

Now consider $g\ne 1$. In this case, 
declare that $C^0,C^1,C^2$ start in degrees 2,1,0 respectively
(which makes the differential preserve the filtration), and 
again consider the graded complex $C_{\gr}^\bullet$. 
The graded Euler characteristic of this complex 
is $(t^2-2t+1)(1-t)^{-2}=1$. 

The cohomology in the $C^0_{\gr}$ term is 
the set of $b\in \CC [x,y]$ such that $ab=ba^g$ for all $a$. This means 
that $\hoc^0=0$.  

The cohomology of the $C^2_{\gr}$ term is the quotient of 
$\CC[x,y]$ by the ideal generated by 
$a-a^g$, $a\in \CC[x,y]$. Thus the cohomology $\hoc^2$ of the rightmost term is 1-dimensional, in degree 0. 
By the Euler characteristic argument, this implies that $\hoc^1=0$.
The cohomology of the filtered complex $C^\bullet$ is therefore the same, and we are done. 
\end{proof} 

The proposition implies that in the n-dimensional case 
$\hoc^*(B,Bg)$ is 1-dimensional, concentrated in 
degree $\rank(1-g)$. It is not hard to check that 
the group $G$ acts on the sum of these 1-dimensional spaces by 
simply permuting the basis vectors. Thus the theorem is proved. 
\end{proof}

\begin{remark}
Another proof of Theorem \ref{afls} 
is given in \cite{Pi}. 
\end{remark}

\begin{theorem} The algebra $\sH_{1,c}[G,V]$, with formal $c$, is 
the universal deformation of $\sH_{1,0}[G,V]=G\ltimes {\mathrm Weyl}(V)$. 
More specifically, the map 
$f: \CC[\cS]^G\to \hoc^2(G\ltimes {\mathrm Weyl}(V))$
induced by this deformation coincides with the isomorphism 
of Corollary \ref{h2}.  
\end{theorem}

\begin{proof} 
The proof (which we will not give) can be obtained by a direct computation 
with the Koszul resolution for $G\ltimes {\mathrm Weyl}(V)$. 
Such a proof is given in \cite{Pi}. 
The paper \cite{EG} proves a slightly weaker statement that 
the map $f$ is an isomorphism, which suffices
to show that $\sH_{1,c}(G,V)$ is the universal deformation of 
$\sH_{1,0}[G,V]$.  
\end{proof}

\subsection{Finite dimensional representations of $\sH_{0,c}$.} 
Let $\sM_{c}=\Spec\sZ_{0,c}$.
We can regard $\sH_{0,c}=\sH_{0,c}[G,V]$ as 
a finitely generated module over
$\sZ_{0,c}=\cO(\sM_{c})$.
Let $\chi\in \sM_c$ be a central character, $\chi:\sZ_{0,c}\to \CC$. 
Denote by $\langle\chi\rangle$ the ideal in $\sH_{0,c}$ 
generated by the kernel of $\chi$. 

\begin{proposition}\label{gen} 
If $\chi$ is generic then 
$\sH_{0,c}/\langle\chi\rangle$ is the matrix algebra of size $|G|$. In particular, 
$\sH_{0,c}$ has a unique irreducible representation $V_\chi$ 
with central character $\chi$. This representation 
is isomorphic to $\CC G$ as a $G$-module.
\end{proposition}

\begin{proof} It is shown by a standard argument (which we will skip)
that it is sufficient to check the statement 
in the associated graded case $c=0$. In this case, for generic $\chi$,  
$G\ltimes SV/\langle\chi\rangle=G\ltimes {\mathrm Fun}(\cO_\chi)$, where 
$\cO_\chi$ is the (free) orbit of $G$ consisting of the points of $V^*$ that map to $\chi\in V^*/G$,
and ${\mathrm Fun}(\cO_\chi)$ is the algebra of functions on $\cO_\chi$.  
It is easy to see that this algebra is isomorphic to a matrix algebra, and 
has a unique irreducible representation, ${\mathrm Fun}(\cO_\chi)$, 
which is a regular representation of $G$. 
\end{proof}

\begin{corollary}\label{est}
Any irreducible representation of $\sH_{0,c}$ has dimension $\le |G|$.  
\end{corollary}

\begin{proof} We will use the following lemma.

\begin{lemma}[The Amitsur-Levitzki identity]{\label{lem:ali}}
For any $N\times N$ matrices $X_{1}, \ldots, X_{2N}$ 
with entries in a commutative ring $A$, 
$$
\sum_{\sigma\in \kS_{2n}}(-1)^{\sigma}X_{\sigma(1)}
\cdots X_{\sigma(2N)}=0.
$$
\end{lemma}
\begin{proof}
Consider the ring $\mathrm{Mat}_{N}(A)\otimes \wedge(\xi_{1}, \ldots, \xi_{2n})$. Let $X=\sum_{i} X_{i}\xi_{i}\in R$. So we have
$$
X^{2}=\sum_{i<j}[X_{i},X_{j}]\xi_{i}\xi_{j}\in 
\mathrm{Mat}_{N}(A\otimes \wedge^{\mathrm{even}}(\xi_{1}, \ldots, \xi_{2n})).
$$
It is obvious that $\tr X^{2}=0$. Similarly, one can easily show that
$\tr X^{4}=0, \ldots, \tr X^{2N}=0$. Since the ring 
$A\otimes \wedge^{\mathrm{even}}(\xi_{1}, \ldots, \xi_{2n})$ is commutative,
from the Cayley-Hamilton theorem,
we know that $X^{2N}=0$ which implies the lemma.
\end{proof}

Since for generic $\chi$ the algebra $\sH_{0,c}/\langle\chi\rangle$ is a matrix algebra, 
the algebra $\sH_{0,c}$ satisfies the Amitsur-Levitzki identity.
Next, note that since $\sH_{0,c}$ is a finitely generated 
$\sZ_{0,c}$-module
(by passing to the associated graded and using Hilbert's theorem), every irreducible
representation of $\sH_{0,c}$ is finite dimensional. 
If $\sH_{0,c}$ had an irreducible representation 
$E$ of dimension $m>|G|$, then
by the density theorem the matrix algebra $\mathrm{Mat}_m$ would be a quotient of 
$\sH_{0,c}$. But one can show that 
the Amitsur-Levitzki identity of degree $|G|$ is not satisfied 
for matrices of bigger size than $|G|$. Contradiction. 
Thus, $\dim E\le |G|$, as desired.      
\end{proof}

In general, for special central characters there are representations 
of $\sH_{0,c}$ of dimension less than $|G|$. 
However, in some cases one can show 
that all irreducible representations have dimension exactly $|G|$.
For example, we have the following result. 

\begin{theorem}\label{ratcher}
Let $G=\kS_n$, $V=\h\oplus \h^*$, $\h=\CC^{n}$ 
(the rational Cherednik algebra for $\kS_n$). 
Then for $c\ne 0$, every irreducible representation of $\sH_{0,c}$ 
has dimension $n!$ and is isomorphic to the regular representation of 
$\kS_n$. 
\end{theorem}

\begin{proof}
Let $E$ be an irreducible representation of $\sH_{0,c}$. 
Let us calculate the trace in $E$ of any permutation $\sigma\ne 1$. 
Let $j$ be an index such that $\sigma(j)=i\ne j$. 
Then $s_{ij}\sigma(j)=j$. Hence in $\sH_{0,c}$ we have 
$$
[y_j,x_is_{ij}\sigma]=[y_j,x_i]s_{ij}\sigma=cs_{ij}^2\sigma=c\sigma.
$$
Hence $\tr_E(\sigma)=0$, and thus $E$ is a multiple of the regular representation of $\kS_n$. 
But by Theorem \ref{est}, $\dim E\le n!$, so we get that $E$ is the regular representation, as desired. 
\end{proof}

\subsection{Azumaya algebras}

Let $Z$ be a finitely generated commutative algebra over $\CC$, 
$M=\Spec Z$ the corresponding affine scheme, and $A$ a finitely generated $Z$-algebra. 

\begin{definition}
$A$ is said to be an Azumaya algebra of degree $N$ if the completion 
$\hat A_\chi$ of $A$ at every maximal ideal 
$\chi$ in $Z$ is the matrix algebra of size $N$ over the completion $\hat Z_\chi$ of $Z$. 
\end{definition}

Thus, an Azumaya algebra should be thought of as a bundle of matrix algebras on $M$.
\footnote{If $M$ is not affine, one can define, in a standard manner, the notion of a sheaf of 
Azumaya algebras on $M$.}
For example, if $E$ is an algebraic vector bundle on $M$ then $\End(E)$ is an Azumaya algebra. 
However, not all Azumaya algebras are of this form. 

\begin{example} 
For $q\in \CC^{*}$, consider the quantum torus 
$$\mathbb{T}_{q}=\CC\langle X^{\pm 1},Y^{\pm 1}\rangle/
\langle XY-qYX\rangle.$$
If $q$ is a root of unity of order $N$, then the center of 
$\mathbb{T}_{q}$ is $\langle X^{\pm N}, Y^{\pm N}\rangle=
\CC[M]$ where $M=(\CC^{*})^{2}$.
It is not difficult to show that 
$\mathbb{T}_q$ is an Azumaya algebra of degree $N$, but
$\mathbb{T}_{q}\otimes_{\CC[M]}\CC(M)\not\cong
\mathrm{Mat}_{N}(\CC(M))$, so $\mathbb{T}_q$ is not the endomorphism algebra of a vector bundle. 
\end{example}

\begin{example}
Let $X$ be a smooth irreducible variety over a field of characteristic $p$.
Then $\cD(X)$, the algebra of differential operators on $X$, 
is an Azumaya algebra with rank $p^{\dim X}$, which is not an
endomorphism algebra of a vector bundle. Its center is 
$Z=\cO(T^{*}X)^{\Bbb F}$, the Frobenius twisted functions
on $T^{*}X$.
\end{example}

It is clear that if $A$ is an Azumaya algebra (say, over $\Bbb C$) then 
for every central character $\chi$ of $A$, 
$A/\langle\chi\rangle$ is the algebra $\mathrm{Mat}_N(\CC)$ of complex $N$ by $N$ matrices, and 
every irreducible representation of $A$ has dimension $N$. 

The following important result is due to M. Artin.   

\begin{theorem} Let $A$ be a finitely generated (over $\CC$) 
polynomial identity (PI) algebra of degree $N$ 
(i.e. all the polynomial relations of the matrix algebra of size $N$ are satisfied in $A$). 
Then $A$ is an Azumaya algebra if and only if every irreducible representation of $A$ 
has dimension exactly $N$.  
\end{theorem}

\begin{proof}
See \cite{Ar} Theorem 8.3.
\end{proof}

Thus, by Theorem \ref{ratcher}, 
for $G=\kS_n$, 
the rational Cherednik algebra $H_{0,c}(\kS_{n}, \CC^{n})$ 
for $c\ne 0$ is an Azumaya algebra of degree $n!$.
Indeed, this algebra is PI of degree $n!$ because the classical Dunkl representation 
embeds it into matrices of size $n!$ over 
$\CC (x_1,\ldots,x_n,p_1,\ldots,p_n)^{\kS_n}$.  

Let us say that $\chi\in M$ is an Azumaya point 
if for some affine neighborhood $U$ of $\chi$ the localization of $A$ to $U$ 
is an Azumaya algebra. Obviously, the set $\mathrm{Az}(M)$ of Azumaya points of $M$ is open. 

Now we come back to the study the space $\sM_c$ corresponding to a symplectic 
reflection algebra $\sH_{0,c}$. 

\begin{theorem} \label{smo}
The set $\mathrm{Az}(\sM_c)$ coincides with the set of smooth points of $\sM_c$.
\end{theorem}

The proof of this theorem is given in the following two subsections. 

\begin{corollary} If $G=\kS_n$ and 
$V=\h\oplus \h^*$, $\h=\CC ^{n}$ 
(the rational Cherednik algebra case) then 
$\sM_c$ is a smooth algebraic variety for $c\ne 0$. 
\end{corollary}
\begin{proof}
Directly from the above theorem.
\end{proof}
\subsection{Cohen-Macaulay property and homological dimension}

To prove Theorem \ref{smo}, we will need some commutative algebra tools. 
Let $Z$ be a finitely generated commutative algebra over $\CC$ without zero divisors. 
By Noether's normalization lemma, there exist elements $z_1,\ldots,z_n\in Z$ which are 
algebraically independent, such that $Z$ is a finitely generated module over $\CC[z_1,\ldots,z_n]$.

\begin{definition} The algebra $Z$ (or the variety $\Spec Z$) is said to be Cohen-Macaulay if 
$Z$ is a locally free (=projective) module over $\CC[z_1,\ldots,z_n]$. 
\footnote{It was proved by Quillen that a locally free module over 
a polynomial algebra is free; this is a difficult theorem, which will not be needed here.}
\end{definition}

\begin{remark}
It was shown by Serre that if $Z$ is locally free over $\CC[z_1,\ldots,z_n]$ for some choice
of $z_1,\ldots,z_n$, then it happens for any 
choice of them (such that $Z$ is finitely generated as a module over 
$\CC [z_1,\ldots,z_n]$).   
\end{remark}

\begin{remark} Another definition of the Cohen-Macaulay property is that 
the dualizing complex $\omega_Z^\bullet$ of $Z$ is concentrated in degree zero. 
We will not discuss this definition here. 
\end{remark}

It can be shown that the Cohen-Macaulay property is stable under 
localization. Therefore, it makes sense to make the following definition. 

\begin{definition} An algebraic variety $X$ is Cohen-Macaulay if 
the algebra of functions on every affine open set in $X$ is Cohen-Macaulay. 
\end{definition}

Let $Z$ be a finitely generated commutative algebra over $\CC $ without zero divisors, and  
let $M$ be a finitely generated module over $Z$.

\begin{definition} $M$ is said to be Cohen-Macaulay
if for some algebraically independent $z_1,\ldots,z_n\in Z$ such that $Z$ is finitely generated over 
$\CC [z_1,\ldots,z_n]$, $M$ is locally free over $\CC [z_1,\ldots,z_n]$.
\end{definition}

Again, if this happens for some $z_1,\ldots,z_n$, then it happens for any of them. 
We also note that $M$ can be Cohen-Macaulay without $Z$ being Cohen-Macaulay, 
and that $Z$ is a Cohen-Macaulay algebra iff it is a Cohen-Macaulay module over itself. 

We will need the following standard properties of Cohen-Macaulay algebras and 
modules.

\begin{theorem}\label{cmac}
\begin{enumerate}
\item[(i)] Let $Z_1\subset Z_2$ be a finite extension of finitely generated commutative $\CC $-algebras,
without zero divisors, and $M$ be a finitely generated module over $Z_2$. Then 
$M$ is Cohen-Macaulay over $Z_2$ iff it is Cohen-Macaulay over $Z_1$. 
\item[(ii)]Suppose that $Z$ is the algebra of functions on a smooth affine variety. 
Then a $Z$-module $M$ is Cohen-Macaulay if and only if it is projective. 
\end{enumerate}
\end{theorem}

\begin{proof}
The proof can be found in the text book \cite{Ei}.
\end{proof}

In particular, this shows that the algebra of functions 
on a smooth affine variety is Cohen-Macaulay. 
Algebras of functions on many singular varieties are also Cohen-Macaulay. 

\begin{example}
The algebra of regular functions on the cone $xy=z^2$ is Cohen-Macaulay. 
This algebra can be identified as 
$\CC[a,b]^{\ZZ_{2}}$ by letting $x=a^{2}, y=b^{2}$ and $z=ab$,
where the $\ZZ_{2}$ action is defined by $a\mapsto -a$,
$b\mapsto -b$.
It contains a subalgebra $\CC[a^{2},b^{2}]$,
and as a module over this subalgebra, it is free of rank $2$ with 
generators $1$, $ab$.
\end{example}

\begin{example}
Any irreducible affine algebraic curve is Cohen-Macaulay.
For example, the algebra of regular functions on $y^{2}=x^{3}$
is isomorphic to the subalgebra of $\CC [t]$ spanned by $1, t^{2}, t^{3},\ldots$.
It contains a subalgebra $\CC[t^{2}]$ and as a module over this subalgebra, it is free of rank $2$ with generators $1$, $t^{3}$.
\end{example}

\begin{example}
Consider the subalgebra in $\CC[x,y]$ spanned by $1$ and $x^iy^j$ with $i+j\ge 2$. 
It is a finite generated module over $\CC[x^{2}, y^{2}]$, but not free.
So this algebra is not Cohen-Macaulay.
\end{example}

Another tool we will need is homological dimension.
We will say that an algebra $A$ has homological dimension 
$\le d$ if any (left) $A$-module $M$ has a projective resolution of length 
$\le d$. The homological dimension of $A$ is the smallest integer having this property. 
If such an integer does not exist, $A$ is said to have infinite homological dimension. 

It is easy to show that the homological dimension of $A$ is $\le d$ if and only if 
for any $A$-modules $M,N$ one has $\Ext^i(M,N)=0$ for $i>d$.
Also, the homological dimension clearly does not decrease 
under taking associated graded of the algebra under 
a positive filtration (this is clear from considering 
the spectral sequence attached to the filtration). 

It follows immediately from this definition that 
homological dimension is Morita invariant. 
Namely, recall that a Morita equivalence 
between algebras $A$ and $B$ is 
an equivalence of categories
$A$-mod $\to$ $B$-mod. Such an equivalence maps projective modules to projective ones, 
since projectivity is a categorical property ($P$ is projective if and only if
the functor $\Hom(P,\cdot)$ is exact). This implies that if $A$ and $B$ are Morita equivalent then 
their homological dimensions are the same. 

Then we have the following important theorem. 

\begin{theorem}\label{fini} The homological dimension of a commutative 
finitely generated $\CC $-algebra $Z$ is finite if and only if $Z$ is regular, i.e. is the algebra of functions 
on a smooth affine variety. 
\end{theorem}

\subsection{Proof of Theorem \ref{smo}}

First let us show that any smooth point $\chi$ of 
$\sM_c$ is an Azumaya point. Since $\sH_{0,c}=\End_{\sB_{0,c}}\sH_{0,c}\e=
\End_{\sZ_{0,c}}(\sH_{0,c}\e)$, it is sufficient to 
show that the coherent sheaf on $\sM_c$ corresponding to the module 
$\sH_{0,c}\e$ is a vector bundle near $\chi$. 
By Theorem \ref{cmac} (ii), for this it suffices to show 
that $\sH_{0,c}\e$ is a Cohen-Macaulay $\sZ_{0,c}$-module.

To do so, first note that the statement is true for $c=0$. 
Indeed, in this case the claim is that $SV$ is a Cohen-Macaulay module over $(SV)^G$. 
But $SV$ is a polynomial algebra, which is Cohen-Macaulay, 
so the result follows from Theorem \ref{cmac}, (i). 

Now, we claim that if $Z,M$ are positively filtered
and $\gr M$ is a Cohen-Macaulay $\gr Z$-module then 
$M$ is a Cohen-Macaulay $Z$-module. Indeed, let $z_1,\ldots,z_n$ be homogeneous algebraically independent 
elements of $\gr Z$ such that $\gr Z$ is a finite module over the subalgebra generated by them. 
Let $z_1', \ldots, z_n'$ be their liftings to $Z$. 
Then $z_1', \ldots, z_n'$ are algebraically independent, and 
the module $M$ over $\CC[z_1', \ldots, z_n']$ is finitely generated and (locally) free since so is 
the module $\gr M$ over $\CC [z_1, \ldots, z_n]$. 

Recall now that $\gr\sH_{0, c}\e=SV$, $\gr\sZ_{0, c}=(SV)^G$. 
Thus the $c=0$ case implies the general case, and we are done. 

Now let us show that any Azumaya point of $\sM_c$ is smooth. 
Let $U$ be an affine open set in $\sM_c$ consisting of Azumaya points. 
Then the localization $\sH_{0,c}(U):=\sH_{0,c}\otimes_{\sZ_{0,c}}
\cO_U$ is an Azumaya algebra. 
Moreover, for any $\chi\in U$, the unique irreducible representation of 
$\sH_{0,c}(U)$ with central character $\chi$ is the regular representation 
of $G$ (since this holds for generic $\chi$ by Proposition \ref{gen}). 
This means that $\e$ is a rank $1$ idempotent 
in $\sH_{0,c}(U)/\langle\chi\rangle$ for all $\chi$. 
In particular, $\sH_{0,c}(U)\e$ is a vector bundle on $U$. 
Thus the functor $F: \cO_U$-mod $\to \sH_{0,c}(U)$-mod 
given by the formula 
$F(Y)=\sH_{0,c}(U)\e\otimes_{\cO_U}Y$ is an equivalence of categories
(the quasi-inverse functor is given by the formula $F^{-1}(N)=\e N$). 
Thus $\sH_{0,c}(U)$ is Morita equivalent to $\cO_U$, and therefore 
their homological dimensions are the same. 

On the other hand, the homological dimension of 
$\sH_{0,c}$ is finite (in fact, it equals to $\dim V$). 
To show this, note that by the Hilbert syzygies theorem, 
the homological dimension of $SV$ is $\dim V$. 
Hence, so is the homological dimension of $G\ltimes SV$ (as 
$\Ext^{*}_{G\ltimes SV}(M,N)=\Ext^{*}_{SV}(M,N)^G$). 
Thus, since $\gr \sH_{0,c}=G\ltimes SV$, 
we get that $\sH_{0,c}$ has homological dimension $\le \dim V$. 
Hence, the homological dimension of $\sH_{0,c}(U)$ is also $\le \dim V$
(as the homological dimension clearly does not increase under the localization). 
But $\sH_{0,c}(U)$ is Morita equivalent to $\cO_U$, so $\cO_U$ has a finite homological dimension. 
By Theorem \ref{fini}, this implies that $U$ consists of smooth points. 

\begin{corollary} $\mathrm{Az}(\sM_c)$ is also the set of points at which 
the Poisson structure of $\sM_c$ is symplectic.  
\end{corollary}

\begin{proof} The variety $\sM_c$ is symplectic outside of a subset
of codimension $2$, because so is $\sM_0$.
Thus the set $\sS$ of smooth points of $\sM_c$ where the top exterior power 
of the Poisson bivector vanishes is of codimension $\ge 2$. 
Since the top exterior power of the Poisson bivector is locally a regular function, 
this implies that $\sS$ is empty. 
Thus, every smooth point is symplectic, 
and the corollary follows from the theorem.  
\end{proof}

\subsection{Notes}
Our exposition in this section follows Section 8 -- Section 10 of \cite{E4}.

\newpage \section{Calogero-Moser spaces}

\subsection{Hamiltonian reduction along an orbit}{\label{sec:hrao}}
 
Let $\cM$ be an affine algebraic variety and $G$ a reductive algebraic group. 
Suppose $\cM$ is Poisson and the action of $G$ preserves 
the Poisson structure.
Let $\g$ be the Lie algebra of $G$ and $\g^{*}$ the dual of $\g$.
Let $\mu: \cM\to \g^{*}$ be a moment map for this action
(we assume it exists). It induces a map $\mu^{*}: S\g\to \CC[\cM]$.

Let $\Ob$ be a closed coadjoint orbit
of $G$,  $I_\Ob$ be the ideal in $S\g$ 
corresponding to $\Ob$, and let $J_\Ob$ be the ideal in $\CC[\cM]$
generated by $\mu^*(I_\Ob)$. Then $J_\Ob^G$ is a Poisson ideal in 
$\CC[\cM]^G$, and $A=\CC[\cM]^G/J_\Ob^G$ is a Poisson algebra. 

Geometrically, $\Spec(A)=\mu^{-1}(\Ob)/G$ (categorical quotient). 
It can also be written as $\mu^{-1}(z)/G_z$, where $z\in \Ob$ and
$G_z$ is the stabilizer of $z$ in $G$. 

\begin{definition}
The scheme $\mu^{-1}(\Ob)/G$ is
called {\em the Hamiltonian reduction of $\cM$ with respect to $G$ 
along $\Ob$}. We will denote by $R(\cM,G,\Ob )$. 
\end{definition}

The following proposition is standard. 

\begin{proposition}{\label{prop:freeact}}
 If $\cM$ is a symplectic variety and the action of $G$ on
$\mu^{-1}(\Ob )$ is free, then 
$R(\cM,G,\Ob )$ is a symplectic variety, of dimension 
$\dim(\cM)-2\dim(G)+\dim(\Ob )$. 
\end{proposition}

\subsection{The Calogero-Moser space}{\label{sec:cmspace}}

Let $\cM=T^*{\mathrm Mat}_n(\CC )$, and $G=\PGL_n(\CC)$ (so $\g={\gsl}_n(\CC )$). 
Using the trace form we can identify $\g^*$ with $\g$, and 
$\cM$ with ${\mathrm Mat}_n(\CC )\oplus {\mathrm Mat}_n(\CC )$.
Then a moment map is given by the formula $\mu(X,Y)=[X,Y]$, for
$X,Y\in {\mathrm Mat}_n(\CC )$. 

Let $\Ob $ be the orbit of the matrix ${\rm diag}(-1,-1,\ldots,-1,n-1)$, 
i.e. the set of traceless matrices $T$ such that $T+1$ has rank $1$. 
\begin{definition}[Kazhdan, Kostant, Sternberg, \cite{KKS}]
The scheme $\mC_n:=R(\cM,G,\Ob )$ is called {\em the Calogero-Moser space}. 
\end{definition}

\begin{proposition}
The action of $G$ on $\mu^{-1}(\Ob )$ is free, and thus 
(by Proposition \ref{prop:freeact}) $\mC_n$ is a smooth 
symplectic variety (of dimension $2n$). 
\end{proposition} 

\begin{proof}
It suffices to show that if $X,Y$ 
are such that $XY-YX+1$ has rank 1, then 
$(X,Y)$ is an irreducible set of matrices. 
Indeed, in this case, by Schur's lemma,
if $B\in \GL_n$ is such that $BX=XB$ and $BY=YB$ 
then $B$ is a scalar, so the stabilizer of $(X,Y)$ in $\PGL_n$ is
trivial. 

To show this, assume that $\cW\ne 0$ is an invariant subspace 
of $X,Y$. In this case, the eigenvalues 
of $[X,Y]$ on $\cW$ are a subcollection of the collection of $n-1$
copies of $-1$ and one copy of $n-1$. The sum of the elements of
this subcollection must be zero, since it is the trace of $[X,Y]$
on $\cW$. But then the subcollection must be the entire collection,
so $\cW=\CC^n$, as desired. 
\end{proof}

Thus, $\mC_n$ is the space of conjugacy classes of pairs of
$n\times n$ matrices $(X,Y)$ such that the matrix $XY-YX+1$ has rank $1$. 

In fact, one also has the following more complicated theorem. 

\begin{theorem}[G. Wilson, \cite{Wi}]{\label{thm:wil}} 
The Calogero-Moser space is connected. 
\end{theorem}

We will give a proof of this theorem later, in Subsection \ref{proofwil}. 

\subsection{The Calogero-Moser integrable system}

Let $\cM$ be a symplectic variety,
and let $H_1,\ldots,H_n$ be regular functions on $\cM$ such that
$\{H_i,H_j\}=0$ and $H_i$'s are algebraically independent
everywhere. Assume that $\cM$ carries a symplectic action of a reductive
algebraic group $G$ with moment map $\mu: \cM\to
\g^*$, which preserves the functions $H_i$, and let 
$\Ob $ be a coadjoint orbit of $G$. Assume that $G$ acts freely on
$\mu^{-1}(\Ob )$, and 
so the Calogero-Moser space $R(\cM, G, \Ob )$ is symplectic.
The functions $H_{i}$ descend to $R(\cM, G,\Ob )$.
If they are still algebraically independent and 
$n=\dim R(\cM,G,\Ob )/2$, then we get an integrable system 
on $R(\cM,G,\Ob )$. 
 
A vivid example of this is the Kazhdan-Kostant-Sternberg
construction of the Calogero-Moser system. In this case
$\cM=T^*{\rm Mat}_n(\CC )$ (regarded as the set of pairs of
matrices $(X,Y)$ as in Section \ref{sec:cmspace}), with the usual
symplectic form $\omega=\tr(\d Y\wedge \d X)$. Let
$H_i=\tr(Y^i)$, $i=1,\ldots, n$. 
Let $G=\PGL_n(\CC)$ act on $\cM$ by conjugation, and let $\Ob $ be
the coadjoint orbit of $G$ considered in Subsection \ref{sec:cmspace}. 
Then the system $H_1,\ldots,H_n$ descends to a system of
functions in involution on $R(\cM,G,\Ob )$, which is the
Calogero-Moser space $\mC_n$. Since this space is
$2n$-dimensional, $H_1,\ldots,H_n$ form an integrable system on
$\mC_n$. It is called {\em the (rational) Calogero-Moser
system}.

The Calogero-Moser flow is, by definition, the Hamiltonian flow
on $\mC_n$ defined by the Hamiltonian
$H=H_2=\tr(Y^2)$. Thus this flow is integrable, in the sense that
it can be included in an integrable system. In particular, its
solutions can be found in quadratures using the inductive
procedure of reduction of order. However (as often happens with
systems obtained by reduction), solutions can also be found by a
much simpler procedure, since they can be found already on the
``non-reduced'' space $\cM$: indeed, on $\cM$ the Calogero-Moser flow
is just the motion of a free particle in the space of matrices,
so it has the form $g_t(X,Y)=(X+2Yt,Y)$. The same formula is
valid on $\mC_n$. In fact, we can use the same method to
compute the flows corresponding to all the Hamiltonians
$H_i=\tr(Y^i)$, $i\in \mathbb{N}$: these flows are given by the
formulas $$ g_t^{(i)}(X,Y)=(X+iY^{i-1}t,Y).  $$

Let us write the Calogero-Moser system explicitly in coordinates. To
do so, we first need to introduce local coordinates on the
Calogero-Moser space $\mC_n$.

To this end, let us restrict our attention to the open set
$U_n\subset \mC_n$ which consists of conjugacy classes
of those pairs $(X,Y)$ for which the matrix $X$ is
diagonalizable, with distinct eigenvalues; 
by Wilson's Theorem \ref{thm:wil},
this open set is dense in $\mC_n$.

A point $P\in U_n$ may be represented by a pair $(X,Y)$ such that
$X=\diag(x_1,\ldots,x_n)$, $x_i\ne x_j$. In this case, the
entries of $T:=XY-YX$ are $(x_i-x_j)y_{ij}$. In particular, the
diagonal entries are zero. Since the matrix $T+1$ has rank $1$,
its entries $\kappa_{ij}$ have the form $a_ib_j$ for some numbers
$a_i,b_j$. On the other hand, $\kappa_{ii}=1$, so $b_j=a_j^{-1}$
and hence $\kappa_{ij}=a_ia_j^{-1}$. By conjugating $(X,Y)$ by
the matrix $\diag(a_1,\ldots,a_n)$, we can reduce to the
situation when $a_i=1$, so $\kappa_{ij}=1$. Hence the matrix $T$
has entries $1-\delta_{ij}$ (zeros on the diagonal, ones off the
diagonal). Moreover, the representative of $P$ with diagonal $X$
and $T$ as above is unique up to the action of the symmetric
group $\kS_n$. Finally, we have $(x_i-x_j)y_{ij}=1$ for $i\ne j$,
so the entries of the matrix $Y$ are 
$y_{ij}=1/(x_i-x_j)$ if $i\neq j$. On the other hand, the
diagonal entries $y_{ii}$ 
of $Y$ are unconstrained. Thus
we have obtained the following result. 

\begin{proposition}
Let ${\CC }^n_{\reg}$ be the open set 
of $(x_1,\ldots,x_n)\in {\CC }^n$ such that $x_i\ne x_j$ for $i\ne j$. 
Then there exists an isomorphism of algebraic varieties 
$\xi: T^*({\CC }^n_{\reg}/\kS_n)\to U_n$ given by the formula 
$$(x_1,\ldots,x_n,p_1,\ldots,p_n)\mapsto (X,Y),$$ where 
$X=\diag(x_1,\ldots,x_n)$, and $Y=Y(\mathbf x,\mathbf p):=(y_{ij})$, 
$$
y_{ij}=\frac{1}{x_i-x_j}, i\ne j,\ y_{ii}=p_i.
$$   
\end{proposition}

In fact, we have a stronger result: 

\begin{proposition}\label{symvar} $\xi$ is an isomorphism  
of symplectic varieties (where the cotangent bundle 
is equipped with the usual symplectic structure). 
\end{proposition}

For the proof of Proposition \ref{symvar}, 
we will need the following general and important but easy
theorem. 

\begin{theorem}[The necklace bracket formula]\label{nbf} 
Let $a_1,\ldots ,a_r$ and $b_1,\ldots,b_s$ 
be either $X$ or $Y$. 
Then on $\cM$ we have 
\begin{eqnarray*}
\lbrace{\tr(a_1\cdots a_r),\tr(b_1\cdots b_s)\rbrace}
&=&
\sum_{(i,j): a_i=Y, b_j=X}
\tr(a_{i+1}\cdots a_ra_1\cdots a_{i-1}b_{j+1}
\cdots b_sb_1\cdots b_{j-1})-\\
&&\quad\sum_{(i,j): a_i=X, b_j=Y}
\tr(a_{i+1}\cdots  a_ra_1\cdots  a_{i-1}b_{j+1}\cdots  b_sb_1\cdots  b_{j-1}). 
\end{eqnarray*}
\end{theorem}

\begin{proof}[Proof of Proposition \ref{symvar}]
Let $a_k=\tr(X^k)$, $b_k=\tr(X^kY)$.
It is easy to check using the necklace bracket formula 
that on $\cM$ we have 
$$
\lbrace{a_m,a_k\rbrace}=0,\quad
\lbrace{b_m,a_k\rbrace}=ka_{m+k-1},\quad
\lbrace{b_m,b_k\rbrace}=(k-m)b_{m+k-1}.
$$
On the other hand, $\xi^*a_k=\sum x_i^k$, 
$\xi^*b_k=\sum x_i^kp_i$. Thus we see that 
$$
\lbrace{f,g\rbrace}=\lbrace{\xi^*f,\xi^*g\rbrace},
$$ 
where $f,g$ are either $a_k$ or $b_k$. 
But the functions $a_k,b_k$, $k=0,\ldots,n-1$, form a local
coordinate system near a generic point of $U_n$, so we are done. 
\end{proof} 

Now let us write the Hamiltonian of the Calogero-Moser
system in coordinates. It has the form 
\begin{equation}\label{cmeq}
H=\tr(Y(\mathbf x,\mathbf p)^2)=\sum_i p_i^2-\sum_{i\ne j}\frac{1}{(x_i-x_j)^2}.
\end{equation}
Thus the Calogero-Moser Hamiltonian describes the motion 
of a system of $n$ particles on the line with interaction
potential $-1/x^2$, which we considered in Section \ref{Olper}.

Now we finally see the usefulness of the Hamiltonian reduction
procedure. The point is that it is not clear at all 
from formula (\ref{cmeq}) why the Calogero-Moser Hamiltonian
should be completely integrable. However, our reduction procedure
implies the complete integrability of $H$, and gives an explicit
formula for the first integrals: 
\footnote{Thus, for type $A$ we have two methods of proving the
integrability of the Calogero-Moser system - one using 
Dunkl operators and one using Hamiltonian reduction.}
$$
H_i=\tr(Y(\mathbf x,\mathbf p)^i).
$$

Moreover, this procedure immediately 
gives us an explicit solution of the system.
Namely, assume that $\mathbf x(t),\mathbf p(t)$ is the solution 
with initial condition $\mathbf x(0),\mathbf p(0)$.
Let $(X_0,Y_0)=\xi(\mathbf x(0),\mathbf p(0))$.  
Then $x_i(t)$ are the eigenvalues of the matrix 
$X_t:=X_0+2tY_0$, and $p_i(t)=x_i'(t)/2$. 

\subsection{Proof of Wilson's theorem}\label{proofwil}

Let us now give a proof of Theorem \ref{thm:wil}.

We have already shown that all components of 
$\mC_{n}$ are smooth and have dimension
$2n$. Also, we know that there is at least one component (the
closure of $U_n$), and that the other components, if they exist, 
do not contain pairs $(X,Y)$ in which $X$ is regular semisimple. 
This means that these components are contained in the hypersurface 
$\Delta(X)=0$, where $\Delta(X)$ stands for the discriminant of $X$ (i.e.,
$\Delta(X):=\prod_{i\ne j}(x_i-x_j)$, where $x_i$ are the eigenvalues of $X$).

Thus, to show that such additional components don't in fact exist, it suffices to show that 
the dimension of the subscheme $\mC_n(0)$ cut out in $\mC_n$ by the equation 
$\Delta(X)=0$ is $2n-1$.  

To do so, first notice that the condition $\rank([X,Y]+1)=1$ is equivalent to the equation
$\wedge^{2}([X,Y]+1)=0$; thus, the latter can be used 
as the equation defining $\mC_n$ inside $T^*{\rm Mat}_n/PGL_n$. 

Define ${\mC}_{n}^0:=\Spec(\gr{\cO}(\mC_n))$
to be the degeneration of $\mC_n$
(we use the filtration on ${\cO}(\mC_n)$ defined by $\deg(X)=0$, $\deg(Y)=1$). 
Then $\mC_n^0$ is a closed subscheme in the scheme 
$\widetilde{\mC}_n^0$ cut out by the equations 
$\wedge^{2}([X,Y])=0$ in $T^*{\mathrm Mat}_n/\PGL_n$. 

Let $(\widetilde{\mC}_n^0)_{\mathrm red}$ be the reduced part of 
$\widetilde{\mC}_n^0$. Then $(\widetilde{\mC}_n^0)_{\mathrm red}$ 
coincides with the categorical quotient
$\{(X, Y)|\rank([X, Y])\le 1\}/\PGL_{n}$.

Our proof is based on the following proposition.

\begin{proposition}\label{catquo} 
The categorical quotient $\{(X, Y)|\rank([X, Y])\le 1\}/\PGL_{n}$ 
coincides with the categorical quotient
$\{(X, Y)|[X, Y]=0\}/\PGL_{n}$. 
\end{proposition}

\begin{proof}
It is clear that $\{(X, Y)|[X, Y]=0\}/\PGL_{n}$ 
is contained in $\{(X, Y)|\rank([X, Y])\le 1\}/\PGL_{n}$. 
For the proof of the opposite inclusion we need to show that 
any regular function on $\{(X, Y)|\rank([X, Y])\le 1\}/\PGL_{n}$ 
is completely determined by its values on the subvariety
$\{(X, Y)|[X, Y]=0\}/\PGL_{n}$, i.e. that any invariant polynomial 
on the set of pairs of matrices with commutator of rank at most $1$ 
is completely determined by its values on pairs of commuting matrices. 
To this end, we need the following 
lemma from linear algebra.

\begin{lemma}\label{Rud}
If $A, B$ are square matrices such that $[A, B]$ has rank $\leq 1$, then
there exists a basis in which both $A, B$ are upper triangular.
\end{lemma}
\begin{proof}
Without loss of generality, we can assume $\ker A\neq 0$ (by replacing $A$ with $A-\lambda$ if needed) 
and that $A\neq 0$. It suffices to show that there exists a proper nonzero subspace invariant  
under $A, B$; then the statement will follow by induction in dimension. 

Let $C=[A,B]$ and suppose $\rank C=1$ (since the case $\rank C=0$
is trivial). If $\ker A\subset \ker C$, then $\ker A$ is $B$-invariant: 
if $Av=0$ then $ABv=BAv+Cv=0$. Thus $\ker A$ is the required subspace.
If $\ker A\nsubseteq \ker C$, 
then there exists a vector 
$v$ such that $Av=0$ but $Cv\neq 0$.
So $ABv=Cv\neq 0$. Thus $\mathrm{Im} C\subset \mathrm{Im} A$.
So $\mathrm{Im} A$ is $B$-invariant: 
$BAv=ABv+Cv\in {\rm Im}A$. So ${\rm Im}A$ is the required subspace.

This proves the lemma.
\end{proof}

Now we are ready to prove Proposition \ref{catquo}.
By the fundamental theorem of invariant theory, 
the ring of invariants of $X$ and $Y$
is generated by traces of words of $X$ and $Y$: ${\rm
Tr}(w(X,Y))$. If $X$ and $Y$ are upper triangular 
with eigenvalues $x_i,y_i$, then any such trace has the form 
$\sum x_i^my_i^r$, i.e. coincides with the value of the corresponding 
invariant on the diagonal parts $X_{\rm diag}, Y_{\rm diag}$ of $X$ and $Y$, 
which commute. The proposition is proved.
\end{proof} 

We will also need the following proposition:

\begin{proposition}\label{catquo1}
The categorical quotient 
$\{(X, Y)|[X, Y]=0\}/\PGL_{n}$
is isomorphic to $(\CC^n\times \CC^n)/\kS_n$, i.e. 
its function algebra is $\CC[x_1,\ldots,x_n,y_1,\ldots,y_n]^{\kS_n}$. 
\end{proposition}

\begin{proof} 
Restriction to diagonal matrices defines a homomorphism 
$$
\xi: \cO(\{(X, Y)|[X, Y]=0\}/\PGL_{n})\to 
\CC[x_1,\ldots,x_n,y_1,\ldots,y_n]^{\kS_n}.
$$
Since (as explained in the proof of Proposition \ref{catquo}),
any invariant polynomial of entries of commuting 
matrices is determined by its values on diagonal matrices, 
this map is injective. Also, 
$\xi(\tr(X^mY^r))=\sum x_i^my_i^r$, where $x_i,y_i$ are the 
eigenvalues of $X$ and $Y$. 

Now we use the following well known theorem of H. Weyl
(from his book ``Classical groups'').

\begin{theorem}\label{Wey} Let $B$ be an algebra over $\CC$. 
Then the algebra $S^nB$ is generated by elements of the form 
$$
b\otimes 1\otimes\cdots\otimes 1+
1\otimes b\otimes\cdots\otimes 1+\cdots+
1\otimes 1\otimes\cdots\otimes b.
$$
\end{theorem}

\begin{proof} Since $S^nB$ is linearly spanned by elements of the
form $x\otimes\cdots\otimes x$, $x\in B$, it suffices to prove the 
theorem in the special case $B=\CC[ x]$. In this case, 
the result is simply the fact that the ring of symmetric
functions is generated by power sums, which is well known. 
\end{proof} 

\begin{corollary} \label{gene}
The ring $\CC[x_1,\ldots,x_n,y_1,\ldots,y_n]^{\kS_n}$ 
is generated by the polynomials $\sum x_i^my_i^r$ for $m,r\ge 0$,
$m+r>0$. 
\end{corollary}

\begin{proof}
Apply Theorem \ref{Wey} in the case $B=\CC[x,y]$. 
\end{proof} 

Corollary \ref{gene} implies that $\xi$ is surjective. 
Proposition \ref{catquo1} is proved.
\end{proof}

Now we are ready to prove Wilson's theorem. 
Let $\mC_n(0)^0$ be the degeneration of $\mC_n(0)$, 
i.e. the subscheme of $\mC_n^0$ cut out by the equation $\Delta(X)=0$. 
According to Propositions \ref{catquo} and \ref{catquo1}, 
the reduced part $(\mC_n(0)^0)_{\rm red}$ is contained 
in the hypersurface in $(\CC^n\times \CC^n)/\kS_n$ 
cut out by the equation $\prod_{i<j}(x_i-x_j)=0$. 
This hypersurface has dimension $2n-1$, so we are done. 

\subsection{The Gan-Ginzburg theorem}

Let ${\rm Comm}(n)$ be  {\em the commuting scheme} 
defined in $T^*{\rm Mat}_n={\rm Mat}_n\times {\rm Mat}_n$ 
by the equations $[X,Y]=0$, $X,Y\in {\rm Mat}_n$. Let $G=\PGL_n$, and consider  
the categorical quotient ${\rm Comm}(n)/G$ (i.e., the Hamiltonian reduction $\mu^{-1}(0)/G$ of 
$T^*{\rm Mat}_n$ by the action of $G$), whose algebra of regular functions is 
$A=\CC[{\rm Comm}(n)]^G$. 

It is not known whether the commuting scheme ${\rm Comm}(n)$ is reduced (i.e. 
whether the corresponding ideal is a radical ideal);
this is a well known open problem. The underlying variety is
irreducible (as was shown by Gerstenhaber \cite{Ge1}), 
but very singular, and has a very complicated
structure. However, we have the following result. 

\begin{theorem}[Gan, Ginzburg, \cite{GG}] \label{thm:GG} 
${\rm Comm}(n)/G$ is reduced,
and isomorphic to $\CC^{2n}/\kS_n$. Thus 
$A=\CC[x_1,\ldots,x_n,y_1,\ldots,y_n]^{\kS_n}$. 
The Poisson bracket on this algebra is induced from the standard
symplectic structure on $\CC^{2n}$.  
\end{theorem}
\begin{proof}[Sketch of the proof]
Look at the {\it almost commuting variety}
$\cM_{n}\subset \gl_{n}\times \gl_{n}\times\CC^{n}\times(\CC^{n})^{*}$
defined by
$$
\cM_{n}=\{(X, Y, \mathbf v, \mathbf f)|[X, Y]+\mathbf v\otimes \mathbf f=0\}.
$$
Gan and Ginzburg proved the following result.
\begin{theorem}
$\cM_{n}$ is a complete intersection. It has $n+1$ irreducible 
components denoted by $\cM_{n}^{i}$, labeled by
$i=\dim\CC\langle X, Y\rangle \mathbf v$. Also, $\cM_{n}$ is generically reduced.
\end{theorem}
Since $\cM_{n}$ is generically reduced and is a complete intersection, 
by a standard result of commutative algebra it is reduced.
Thus $\CC[\cM_{n}]$ has no nonzero nilpotents. This implies
$\CC[\cM_{n}]^{G}$ has no nonzero nilpotents.

However, it is easy to show that the algebra $\CC[\cM_n]^G$ is isomorphic to the algebra 
of invariant polynomials of entries of $X$ and $Y$ modulo the ``rank 1'' relation 
$\wedge^2[X,Y]=0$. By a scheme-theoretic version of Proposition \ref{catquo}
(proved in \cite{EG}), the latter is isomorphic to $A$. This implies the theorem
(the statement about Poisson structures is checked directly in coordinates on the open part where 
$X$ is regular semisimple).
\end{proof}

\subsection{The space $\sM_c$ for $\kS_n$
and the Calogero-Moser space.}{\label{sec:tsacms}}

Let $\sH_{0,c}=\sH_{0,c}[\kS_n,V]$ be the symplectic reflection 
algebra of the symmetric group $\kS_n$
and space $V=\h\oplus \h^*$, where $\h=\CC ^n$
(i.e., the rational Cherednik algebra $H_{0,c}(\kS_n,\h)$).
Let $\sM_{c}=\Spec\,\sB_{0,c}[\kS_n,V]$ 
be the Calogero-Moser space defined in Section \ref{sec:ssra}.
It is a symplectic variety for $c\ne 0$.

\begin{theorem} For $c\ne 0$ the space $\sM_c$ is isomorphic to the Calogero-Moser 
space ${\mC}_n$ as a symplectic variety. 
\end{theorem}

\begin{proof} To prove the theorem, we will first construct a map 
$f: \sM_c\to \mC_n$, and then prove that $f$ is an isomorphism.  

Without loss of generality, we may assume that $c=1$. 
As we have shown before, the algebra $\sH_{0,c}$ is an Azumaya algebra. 
Therefore, $\sM_c$ can be regarded as the moduli space of irreducible representations of $\sH_{0,c}$. 

Let $E\in \sM_c$ be an irreducible representation of $\sH_{0,c}$. 
We have seen before that $E$ has dimension $n!$ and is isomorphic to the regular representation 
as a representation of $\kS_n$. Let $\kS_{n-1}\subset \kS_n$ be the subgroup which preserves the element $1$.
Then the space of invariants $E^{\kS_{n-1}}$ has dimension $n$. On this space we have operators 
$X,Y: E^{\kS_{n-1}}\to E^{\kS_{n-1}}$ obtained by restriction of the operators $x_1,y_1$ on $E$
to the subspace of invariants. We have 
$$
[X,Y]=T:=\sum_{i=2}^ns_{1i}.
$$
Let us now calculate the right hand side of this equation explicitly. 
Let $\e$ be the symmetrizer of $\kS_{n-1}$. Let us realize the regular
representation $E$ of $\kS_n$ as $\CC[\kS_n]$ with action of 
$\kS_n$ by left multiplication. 
Then $v_1=\e,v_2=\e s_{12},\ldots,v_n=\e s_{1n}$ 
is a basis of $E^{\kS_{n-1}}$. 
The element $T$ commutes with $\e$, so we have 
$$
Tv_i=\sum_{j\ne i}v_j.
$$
This means that $T+1$ has rank $1$, and hence the pair $(X,Y)$ 
defines a point on the Calogero-Moser space 
$\mC_n$. \footnote{Note that the pair $(X,Y)$ is well defined only up to conjugation, because 
the representation $E$ is well defined only up to an isomorphism.}

We now set $(X,Y)=f(E)$. It is clear that $f: \sM_c\to \mC_n$ is a regular map. 
So it remains to show that $f$ is an isomorphism. 
This is equivalent to showing that the corresponding map 
of function algebras $f^*: \cO(\mC_n)\to \sB_{0,c}$ is an isomorphism. 

Let us calculate $f$ and $f^*$ more explicitly. 
To do so, consider the open set $\sU$ in $\sM_c$ consisting of representations 
in which $x_i-x_j$ acts invertibly. These are exactly the representations that 
are obtained by restricting representations of 
$\kS_n\ltimes \CC[x_1,\ldots,x_n,p_1,\ldots,p_n,\delta(\mathbf x)^{-1}]$ 
using the classical Dunkl embedding. 
Thus representations $E\in \sU$ are of the form $E=E_{\lambda,\mu}$ 
($\lambda,\mu\in \CC^n$, and $\lambda$ has distinct coordinates), where 
$E_{\lambda,\mu}$ is the space of complex valued functions 
on the orbit $\Ob_{\lambda,\mu}\subset \CC^{2n}$, 
with the following action of $\sH_{0,c}$: 
$$
(x_iF)(\mathbf a,\mathbf b)=a_iF(\mathbf a,\mathbf b), \quad
(y_iF)(\mathbf a,\mathbf b)=b_iF(\mathbf a,\mathbf b)+\sum_{j\ne i}\frac{(s_{ij}F)(\mathbf a,\mathbf b)}{a_i-a_j}.
$$
(the group $\kS_n$ acts by permutations). 

Now let us consider the space $E_{\lambda,\mu}^{\kS_{n-1}}$.
A basis of this space is formed by characteristic functions of 
$\kS_{n-1}$-orbits on $\Ob_{\lambda,\mu}$. 
Using the above presentation, it is straightforward to calculate the matrices
of the operators $X$ and $Y$ in this basis:
$$
X=\diag(\lambda_1,\ldots,\lambda_n), 
$$
and 
$$
Y_{ij}=\mu_i
\text{ if }j=i,\ Y_{ij}=\frac{1}{\lambda_i-\lambda_j}\text{ if }j\ne i. 
$$
This shows that $f$ induces an isomorphism $f|_\sU: \sU\to U_n$, where 
$U_n$ is the subset of $\mC_n$ consisting of pairs $(X,Y)$ for which $X$ has distinct eigenvalues. 

From this presentation, it is straigtforward that 
$f^*(\tr(X^p))=x_1^p+\cdots+x_n^p$ for every positive integer $p$. 
Also, $f$ commutes with the natural $\SL_2(\CC)$-action
on $\sM_c$ and $\mC_n$ (by $(X,Y)\to (aX+bY,cX+dY)$), so we also get 
$f^*(\tr(Y^p))=y_1^p+\cdots+y_n^p$, and 
$$
f^*(\tr(X^pY))=\frac{1}{p+1}\sum_{m=0}^p\sum_i x_i^my_ix_i^{p-m}.
$$
Now, using the necklace bracket formula on ${\mC}_n$ and 
the commutation relations of $\sH_{0,c}$, 
we find, by a direct computation, that 
$f^*$ preserves Poisson bracket on the elements 
$\tr(X^p)$, $\tr(X^qY)$. 
But these elements are a local coordinate system near a generic point, 
so it follows that $f$ is a Poisson map. Since the algebra $\sB_{0,c}$ 
is Poisson generated by 
$\sum x_i^p$ and $\sum y_i^p$ for all $p$, we get that $f^*$ is a surjective map. 

Also, $f^*$ is injective. Indeed, by Wilson's theorem the Calogero-Moser space is connected, and hence 
the algebra $\cO (\mC_n)$ has no zero divisors, while 
$\mC_n$ has the same dimension as $\sM_c$. 
This proves that $f^*$ is an isomorphism, so $f$ is an isomorphism. 
\end{proof}

\subsection{The Hilbert scheme $\Hilb_{n}(\CC^{2})$ and the Calogero-Moser space}

The Hilbert scheme $\Hilb_{n}(\CC^{2})$ is defined to be
\begin{eqnarray*}
&&\Hilb_{n}(\CC^{2})=\{\text{ ideals }I\subset\CC[x, y]|\mathrm{codim} I=n \}\\
&=&\{(E, v)|E \text{ is a $\CC[x, y]$-module of dimension $n$}, \, 
v \text{ is a cyclic vector of $E$}\}.
\end{eqnarray*}
The second equality can be easily seen from the short exact sequence
$$0\to I\to \CC[x, y]\to E\to 0.$$
Let $S^{(n)}\CC^{2}
=\underbrace{\CC^{2}\times\cdots\times\CC^{2}}_{n \text{ times}}
/\kS_{n}$, where $\kS_{n}$ acts by permutation.
We have a natural map $\Hilb_{n}(\CC^{2})\to S^{(n)}\CC^{2}$
which sends every ideal $I$ to its zero set (with multiplicities). 
This map is called the {\it Hilbert-Chow map}. 

\begin{theorem}[Fogarty, \cite{F}]
\begin{enumerate}
\item[(i)] $\Hilb_{n}(\CC^{2})$ is a smooth quasiprojective variety.
\item[(ii)] The Hilbert-Chow map $\Hilb_{n}(\CC^{2})\to S^{(n)}\CC^{2}$
is projective. It is a resolution of singularities.
\end{enumerate}
\end{theorem}
\begin{proof}
Proof can be found in \cite{Na}.
\end{proof}

Now consider the Calogero-Moser space $\mC_{n}$
defined in Section \ref{sec:cmspace}.

\begin{theorem}[see \cite{Na}]\label{na} 
The Hilbert Scheme $\Hilb_{n}(\CC^{2})$ is $C^{\infty}$-diffeomorphic
to $\mC_{n}$. 
\end{theorem}

\begin{remark}\label{na1}
More precisely there exists a family of algebraic varieties
over $\mathbb{A}_{1}$, say $X_{t}$, $t\in \Bbb A_1$, such that $X_{t}$ is isomorphic to $\mC_{n}$
if $t\neq 0$ and $X_{0}$ is the Hilbert scheme;
and also if we denote by $\overline{X_{t}}$ the deformation of $\CC^{2n}/\kS_n$ 
into the Calogero-Moser space, then 
there exists a map $f_{t}: X_{t}\twoheadrightarrow \overline{X_{t}}$,
such that for $t\neq 0$, $f_{t}$ is an isomorphism and $f_{0}$ is the
Hilbert-Chow map.
\end{remark}

\begin{remark}
Consider the action of $G=\PGL_{n}$ on $T^{*}{\mathrm Mat_{n}}$. 
As we have discussed, the corresponding moment map is $\mu(X, Y)=[X, Y]$, 
so $\mu^{-1}(0)=\{(X,Y)|[X, Y]=0\}$ is the commuting variety.
We can consider two kinds of quotient $\mu^{-1}(0)/G$ (i.e., of Hamiltonian reduction):

(1) The categorical quotient, i.e.,
$$
\Spec(\CC[x_{ij}, y_{ij}]/\langle[X, Y]=0\rangle)^{G}\cong (\CC^{n}\times\CC^{n})/\kS_{n}.
$$ 
It is a reduced (by Gan-Ginzburg Theorem \ref{thm:GG}), affine 
but singular variety.

(2) The GIT quotient, in which the stability condition
is that there exists a cyclic vector for $X,Y$. This quotient is 
$\Hilb_{n}(\CC^{2})$, which is smooth but not affine. 

Both of these reductions are degenerations of the 
reduction along the orbit of matrices $T$ such that $T+1$ has rank $1$, 
which yields the space $\mC_n$. 
This explains why Theorem \ref{na} and 
the results mentioned in Remark \ref{na1} are natural to expect. 
\end{remark}

\subsection{The cohomology of $\mC_n$}

We also have the following result describing the cohomology of $\mC_{n}$
(and hence, by Theorem \ref{na}, of $\Hilb_n(\CC^2)$).
Define the {\it age filtration} for the symmetric group $\kS_{n}$ by setting
$$
\mathrm{age}(\text{transposition})=1.
$$
Then one can show that for any $\sigma\in \kS_{n}$, $\mathrm{age}(\sigma)=\rank(1-\sigma)|_{\text{reflection representation}}$. 
It is easy to see that $0\leq\mathrm{age}\leq n-1$. 
Notice also that the age filtration can be defined for any Coxeter group.

\begin{theorem}[Lehn-Sorger, Vasserot]{\label{thm:coho}}
The cohomology ring $\H^{*}(\mC_{n}, \CC)$ lives in even degrees only and is isomorphic 
to $\gr({\mathrm Center}(\CC[\kS_{n}]))$ under the age filtration (with doubled degrees). 
\end{theorem}

\begin{proof}
Let us sketch a noncommutative-algebraic proof of this theorem, given in \cite{EG}. 
This proof is based on the following result. 
\begin{theorem}[Nest-Tsygan, \cite{NT}]
If $M$ is an affine symplectic variety and $A$ is a quantization of $M$, 
then
$$
\hoc^*(A[\hbar^{-1}], A[\hbar^{-1}])\cong \H^{*}(M, \CC((\hbar)))
$$
as an algebra over $\CC((\hbar))$.
\end{theorem}

Now, we know that the algebra $\sB_{t,c}$ is a quantization of $\mC_{n}$. 
Therefore by above theorem, the cohomology algebra of 
$\mC_n$ is the cohomology of $\sB_{t,c}$ (for generic $t$). 
But $\sB_{t,c}$ is Morita equivalent to $\sH_{t,c}$, so 
this cohomology is the same as the Hochschild cohomology of $\sH_{t,c}$. 
However, the latter can be computed by using that $\sH_{t,c}$ is given by generators and relations
(by producing explicit representatives of cohomology classes and
computing their product), which gives the result.  
\end{proof}

\subsection{Notes}

Sections \ref{sec:hrao}--\ref{sec:tsacms} follow Section 1, 2, 4 of \cite{E4};
the parts about the Hilbert scheme and its relation to Calogero-Moser 
spaces follow the book \cite{Na} (see also \cite{GS}); 
the other parts follow the paper \cite{EG}.
  
\newpage \section{Quantization of Claogero-Moser spaces}

\subsection{Quantum moment maps and quantum Hamiltonian reduction}

Now we would like to quantize the notion of a moment map. 
Let $\g$ be a Lie algebra, and 
$A$ be an associative algebra equipped with a $\g$-action, 
i.e. a Lie algebra map $\phi: \g\to {\mathrm Der}A$. 
A quantum moment map for $(A,\phi)$ is an associative algebra
homomorphism $\mu: U(\g)\to A$ such that for any $a\in \g$, $b\in
A$ one has $[\mu(a),b]=\phi(a)b$. 

The space of
$\g$-invariants $A^\g$, i.e. elements $b\in A$ such that
$[\mu(a),b]=0$ for all $a\in \g$, is a subalgebra of $A$.
Let $J\subset A$ be the left ideal generated by $\mu(a)$, $a\in
\g$. Then $J$ is not a 2-sided ideal, but $J^\g:=J\cap A^\g$ is a
2-sided ideal in $A^\g$. 
Indeed, let $c\in A^\g$, and $b\in
J^\g$, $b=\sum_i b_i\mu(a_i)$, $b_i\in A,a_i\in \g$. Then 
$bc=\sum b_i\mu(a_i)c=\sum b_ic\mu(a_i)\in J^\g$. 

Thus, the algebra $A//\g:=A^\g/J^\g$ is an associative algebra,
which is called the quantum Hamiltonian reduction of $A$
with respect to the quantum moment map $\mu$.

\subsection{The Levasseur-Stafford theorem}

In general, similarly to the classical case,  it is rather
difficult to compute the quantum reduction $A//\g$. 
For example, in this subsection we will 
describe $A//\g$ in the case when $A=\cD(\g)$ is the algebra of differential operators 
on a reductive Lie algebra $\g$, and $\g$ acts on $A$ through the adjoint 
action on itself. This description is a very nontrivial result
of Levasseur and Stafford.

Let $\h$ be a Cartan subalgebra of $\g$, and $W$ 
the Weyl group of $(\g,\h)$. 
Let $\h_{\reg}$ denote the set of regular points in $\h$, 
i.e. the complement of the reflection hyperplanes. 
To describe $\cD(\g)//\g$, we 
will construct a homomorphism $\HC: \cD(\g)^\g\to \cD(\h)^W$, 
called the Harish-Chandra homomorphism (as it was first
constructed by Harish-Chandra). 
Recall that we have the classical Harish-Chandra isomorphism
$\zeta: \CC [\g]^\g\to\CC[\h]^W$, defined simply by restricting
$\g$-invariant functions on $\g$ to the Cartan subalgebra $\h$. 
Using this isomorphism, we can define an action of $\cD(\g)^\g$ on
$\CC[\h]^W$, which is clearly given by $W$-invariant 
differential operators. 
However, these operators will, in general, have poles on the
reflection hyperplanes. Thus we get a homomorphism 
$\HC': \cD(\g)^\g\to \cD(\h_{\reg})^W$. 

The homomorphism $\HC'$ is called the radial part homomorphism, as
for example for $\g={\su}(2)$ it computes the radial parts of
rotationally invariant differential operators on $\RR^3$ 
in spherical coordinates. This homomorphism is not yet
what we want, since it does not actually land in $\cD(\h)^W$ (the
radial parts have poles). 

Thus we define the Harish-Chandra homomorphism by twisting 
$\HC'$ by the discriminant
$\delta(\mathbf x)=\prod_{\alpha>0}(\alpha,\mathbf x)$
($\mathbf x\in \h$, and $\alpha$ runs over positive roots of $\g$):
$$
\HC(D):=\delta \circ \HC'(D)\circ \delta^{-1}\in \cD(\h_{\reg})^W.
$$

\begin{theorem}\label{thm:LS} 
\begin{enumerate}
\item[(i)] $($Harish-Chandra, \cite{HC}$)$
For any reductive $\g$, $\HC$ lands in $\cD(\h)^W\subset
\cD(\h_{\reg})^W$.  
\item[(ii)]$($Levasseur-Stafford \cite{LS}$)$
The homomorphism $\HC$ defines an isomorphism $\cD(\g)//\g=\cD(\h)^W$.
\end{enumerate}  
\end{theorem}

\begin{remark} 

\begin{enumerate}
\item Part (i) of the theorem says that the poles magically disappear after 
conjugation by $\delta$.
\item Both parts of this theorem are quite nontrivial. 
The first part was proved by Harish-Chandra using analytic methods, and 
the second part by Levasseur and Stafford using the theory of $\cD$-modules.
\end{enumerate}
\end{remark}

In the case $\g=\gl_n$, Theorem \ref{thm:LS}
is a quantum analog of Theorem \ref{thm:GG}. 
The remaining part of this subsection is devoted to the proof   
of Theorem \ref{thm:LS} in this special case, 
using Theorem \ref{thm:GG}. 

We start the proof with the following proposition, valid for any reductive Lie algebra. 

\begin{proposition}\label{concoe}
If $D\in (S\g)^\g$ is a differential
operator with constant coefficients, then $\HC(D)$ is 
the $W$-invariant differential operator with constant
coefficients on $\h$, obtained from $D$ via the classical
Harish-Chandra isomorphism $\eta: (S\g)^\g\to (S\h)^W$. 
\end{proposition}

\begin{proof} Without loss of generality, we may assume that $\g$ is simple.

\begin{lemma}\label{lem:lapl} 
Let $D$ be the Laplacian $\Delta_\g$ of $\g$, corresponding 
to an invariant form. Then $\HC(D)$ is the Laplacian $\Delta_\h$. 
\end{lemma}

\begin{proof}
Let us calculate $\HC'(D)$. 
We have 
$$
D=\sum_{i=1}^r \partial_{x_i}^2+2\sum_{\alpha>0}\partial_{f_\alpha}\partial_{e_\alpha}, 
$$
where $x_i$ is an orthonormal basis of $\h$, and $e_\alpha,f_\alpha$ are root elements such that 
$(e_\alpha,f_\alpha)=1$. 
Thus if $F(\mathbf x)$ is a $\g$-invariant function on $\g$, then we get 
$$
(DF)|_{\h}=\sum_{i=1}^r \partial_{x_i}^2(F|_\h)+2\sum_{\alpha>0}(\partial_{f_\alpha}\partial_{e_\alpha}F)|_{\h}. 
$$
Now let $\mathbf x\in \h$, and consider 
$(\partial_{f_\alpha}\partial_{e_\alpha}F)(\mathbf x)$. We have 
$$
(\partial_{f_\alpha}\partial_{e_\alpha}F)(\mathbf x)=
\partial_s\partial_t|_{s=t=0}F(\mathbf x+tf_\alpha+se_\alpha).
$$
On the other hand, we have 
$$
{\rm Ad}(\be^{s\alpha(\mathbf x)^{-1}e_\alpha})(\mathbf x+tf_\alpha+se_\alpha)=
\mathbf x+tf_\alpha+ts\alpha(\mathbf x)^{-1}h_\alpha+\cdots,
$$
where $h_\alpha=[e_\alpha,f_\alpha]$. Hence, 
$
\partial_s\partial_t|_{s=t=0}F(\mathbf x+tf_\alpha+se_\alpha)=
\alpha(\mathbf x)^{-1}(\partial_{h_\alpha}F)(\mathbf x). 
$
This implies that 
$$
\HC'(D)F(\mathbf x)=\Delta_\h F(\mathbf x)
+2\sum_{\alpha>0}\alpha(\mathbf x)^{-1}\partial_{h_\alpha}F(\mathbf x).
$$
Now the statement of the Lemma reduces to the identity 
$
\delta^{-1}\circ \Delta_\h \circ \delta=\Delta_\h+2\sum_{\alpha>0}\alpha(\mathbf x)^{-1}\partial_{h_\alpha}. 
$
This identity follows immediately from the identity
$
\Delta_\h \delta=0.
$
To prove the latter, it suffices to note that $\delta$ is the lowest degree 
nonzero polynomial on $\h$, which is antisymmetric under the action of $W$. 
The lemma is proved. 
\end{proof} 

Now let $D$ be any element of $(S\g)^\g\subset \cD(\g)^\g$ of degree $d$
(operator with constant coefficients). It is obvious that the leading 
order part of the operator $\HC(D)$ is the operator $\eta(D)$ 
with constant coefficients, whose symbol is just the 
restriction of the symbol of $D$ from $\g^*$ to $\h^*$. 
Our job is to show that in fact $\HC(D)=\eta(D)$. To do so, denote by 
$Y$ the difference $\HC(D)-\eta(D)$. Assume $Y\ne 0$. 
By Lemma \ref{lem:lapl}, the operator $\HC(D)$ commutes with 
$\Delta_\h$. Therefore, 
so does $Y$. Also $Y$ has homogeneity degree $d$ but order $m\le d-1$.
Let $S(\mathbf x,\mathbf p)$ be the symbol of $Y$ ($\mathbf x\in \h, \mathbf p\in \h^*$). 
Then $S$  is a homogeneous function of homogeneity degree $d$
under the transformations $\mathbf x\to t^{-1}\mathbf x$, $\mathbf p\to t\mathbf p$, polynomial in 
$\mathbf p$ of degree $m$. 
From these properties of $S$ it is clear that $S$ is not a polynomial 
(its degree in $\mathbf x$ is $m-d<0$). 
On the other hand, since $Y$ commutes with $\Delta_\h$, 
the Poisson bracket of $S$ with $\mathbf p^2$ is zero.    
Thus Proposition \ref{concoe} follows from Lemma \ref{lem:Raj}. 
\end{proof}

Now we continue the proof of Theorem \ref{thm:LS}. 
Consider the filtration on $\cD(\g)$ in which
$\deg(\g)=1 \deg(\g^*)=0$ (the order filtration), and the associated graded map 
$\gr\HC: \CC[\g\times \g^*]^\g\to \CC[\h_{\reg}\times
\h^*]^W$, which attaches to every differential operator
the symbol of its radial part. 
It is easy to see that this map is just the restriction map to
$\h\oplus \h^*\subset \g\oplus \g^*$, so it actually lands in 
$\CC[\h\oplus \h^*]^W$. 

Moreover, $\gr\HC$ is a map {\bf onto} 
$\CC[\h\oplus \h^*]^W$. 
Indeed, $\gr\HC$ is a Poisson map, so the surjectivity follows 
from the following Lemma. 

\begin{lemma}\label{poisgen}
$\CC[\h\oplus \h^*]^W$ is generated as a
Poisson algebra by $\CC[\h]^W$ and $\CC[\h^*]^W$, i.e. by 
functions $f_m=\sum x_i^m$ and $f_m^*=\sum p_i^m$, $m\ge 1$. 
\end{lemma}

\begin{proof}
We have $\lbrace{f_m^*,f_r\rbrace}=mr\sum x_i^{r-1}p_i^{m-1}$.
Thus the result follows from Corollary \ref{gene}.  
\end{proof} 

Let $K_0$ be the kernel of $\gr\HC$. Then by 
Theorem \ref{thm:GG}, $K_0$ is the ideal of the commuting scheme
$\mathrm{Comm}(\g)/G$. 

Now consider the kernel $K$ of the homomorphism $\HC$. 
It is easy to see that $K\supset J^\g$, so $\gr(K)\supset
\gr(J)^\g$. On the other hand, since $K_0$ is the ideal of
the commuting scheme, we clearly have $\gr(J)^\g\supset K_0$, and $K_0\supset \gr K$. 
This implies that $K_0=\gr K=\gr(J)^\g$, and $K=J^\g$.

It remains to show that $\Im\HC=\cD(\h)^W$. 
Since $\gr K=K_0$, we have 
$\gr\Im\HC=\CC[\h\oplus \h^*]^W$. 
Therefore, to finish the proof of the Harish-Chandra and 
Levasseur-Stafford theorems, 
it suffices to prove the following proposition. 

\begin{proposition}\label{conta}
$\Im \HC\supset \cD(\h)^W$. 
\end{proposition}

\begin{proof} 
We will use the following Lemma. 

\begin{lemma}[N. Wallach, \cite{Wa}] \label{wal} 
$\cD(\h)^W$ is generated as an algebra 
by $W$-invariant functions and $W$-invariant differential
operators with constant coefficients. 
\end{lemma}
 
\begin{proof} 
The lemma follows by taking associated graded algebras from 
Lemma \ref{poisgen}.
\end{proof}

\begin{remark}
Levasseur and Stafford showed \cite{LS} that this lemma is valid 
for any finite group $W$ acting on a finite dimensional vector
space $\h$. However, the above proof does not apply, since, as
explained in \cite{Wa}, Lemma \ref{poisgen} fails for many groups $W$, 
including Weyl groups of exceptional Lie
algebras $E_6,E_7,E_8$ (in fact it even fails for the cyclic group of order
$>2$ acting on a 1-dimensional space!). 
The general proof is more complicated and
uses some results in noncommutative algebra. 
\end{remark}

Lemma \ref{wal} and Proposition \ref{concoe} imply Proposition \ref{conta}.
\end{proof}

Thus, Theorem \ref{thm:LS} is proved.

\subsection{Corollaries of Theorem \ref{thm:LS}}

Let $\g_{\RR}$ be the compact form of $\g$, and
$\Ob$ a regular coadjoint orbit in $\g_{\RR}^*$. Consider the map 
$$
\psi_{\Ob}:\h\to \CC, \quad\psi_\Ob(\mathbf x)=\int_\Ob \be^{(\mathbf b,\mathbf x)}\d \mathbf b,\ \mathbf x\in \h,
$$
where $\d \mathbf b$ is the measure on the orbit coming from the Kirillov-Kostant
symplectic structure. 

\begin{theorem}[Harish-Chandra formula]\label{hcf}
For a regular element ${\mathbf x}\in \h$, we have
$$
\psi_\Ob(\mathbf x)=\delta^{-1}(\mathbf x)\sum_{w\in
W}(-1)^{\ell(w)}\be^{(w\lambda,\mathbf x)},
$$
where $\lambda$ is the intersection of $\Ob$ with the 
dominant chamber in the dual Cartan subalgebra $\h_{\RR}^*\subset
\g_{\RR}^*$, and $\ell(w)$ is the length of an element $w\in W$.
\end{theorem}

\begin{proof}
Take $D\in (S\g)^{\g}$. 
Then $\delta(\mathbf x)\psi_\Ob$ 
is an eigenfunction of $\HC(D)=\eta(D)\in (S\h)^W$ with eigenvalue $\chi_\Ob(D)$, where
$\chi_\Ob(D)$ is the value of the invariant polynomial $D$ at
the orbit $\Ob$.

Since the solutions of the equation $\eta(D)\varphi=\chi_{\Ob}(D)\varphi$
have a basis $\be^{(w\lambda, \mathbf x)}$ where $w\in W$,   
we have 
$$
\delta(\mathbf x)\psi_\Ob(\mathbf x)=\sum_{w\in W}c_{w}\cdot \be^{(w\lambda, \mathbf x)}.
$$
Since it is antisymmetric, we have $c_{w}=c\cdot (-1)^{\ell(w)}$, where $c$ is a constant. 
The fact that $c=1$ can be shown by comparing the asymptotics of both sides as ${\mathbf x}\to \infty$ 
in the regular chamber (using the stationary phase approximation for the integral). 
\end{proof}

From Theorem \ref{hcf} and the Weyl Character formula, we have the following 
corollary.

\begin{corollary}[Kirillov character formula for finite
dimensional representations, \cite{Ki}]
If $\lambda$ is a dominant integral weight, and $L_\lambda$ is 
the corresponding representation of $G$, then 
$$
\tr_{L_\lambda}(\be^{\mathbf x})=
\frac{\delta(\mathbf x)}{\delta_{\tr}(\mathbf x)}\int_{\Ob_{\lambda+\rho}} \be^{(\mathbf b,\mathbf x)}\d \mathbf b,
$$
where $\delta_{\tr}(\mathbf x)$ is the trigonometric version of
$\delta(\mathbf x)$, i.e. the Weyl denominator \linebreak 
$\prod_{\alpha>0}(\be^{\alpha(\mathbf x)/2}-\be^{-\alpha(\mathbf x)/2})$, and $\Ob_\mu$ denotes the
coadjoint orbit passing through $\mu$.
\end{corollary}

\subsection{The deformed Harish-Chandra homomorphism}

Finally, we would like to explain how to quantize the Calogero-Moser space 
$\mC_n$, using the procedure of quantum Hamiltonian reduction. 

Let $\g=\gl_n$, $A=\cD(\g)$ as above. 
Let $k$ be a complex number, and 
$W_k$ be the representation of $\gsl_n$ 
on the space of functions of the form 
$(x_1\cdots x_n)^{k}f(x_1,\ldots,x_n)$, 
where $f$ is a Laurent polynomial
of degree $0$. We regard $W_k$ as a $\g$-module by pulling it
back to $\g$ under the natural projection 
$\g\to \gsl_n$. Let $I_k$ be the annihilator of $W_k$ in
$U(\g)$. The ideal $I_k$ is the quantum counterpart of 
the coadjoint orbit of matrices $T$ such that $T+1$ has rank $1$. 

Let $B_k=\cD(\g)^{\g}/(\cD(\g)\mu(I_{k}))^{\g}$
where $\mu:U(\g)\to A$ is the quantum momentum map
(the quantum Hamiltonian reduction with respect to the ideal $I_k$). 
Then $B_k$ has a filtration 
induced from the order filtration of $\cD(\g)^\g$. 

Let $\HC_k: \cD(\g)^\g\to B_k$ be the natural homomorphism, and 
$K(k)$ be the kernel of $\HC_k$. 

\begin{theorem}[Etingof-Ginzburg, \cite{EG}] \label{defhch}
\begin{enumerate}
\item[(i)] $K(0)=K$, $B_0=\cD(\h)^W$,
$\HC_{0}=\HC$. 
\item[(ii)] $\gr K(k)=\Ker(\gr\HC_k)=K_0$ for all complex $k$.
Thus, $\HC_{k}$ is a flat family of homomorphisms. 
\item[(iii)] The algebra $\gr B_k$ is commutative and isomorphic to
$\CC[\h\oplus \h^*]^W$ as a Poisson algebra. 
\end{enumerate} 
\end{theorem}

Because of this theorem, the homomorphism 
$\HC_k$ is called the deformed Harish-Chandra homomorphism. 

Theorem \ref{defhch} implies that $B_k$ 
is a quantization of the Calogero-Moser space $\mC_n$ 
(with deformation parameter $1/k$). 
But we already know one such quantization - 
the spherical Cherednik algebra 
$B_{1,k}$ for the symmetric group. 
Therefore, the following theorem 
comes as no surprize. 

\begin{theorem}[\cite{EG}] The 
algebra $B_k$ is isomorphic to the spherical rational
Cherednik algebra $B_{1,k}(\kS_n,\Bbb C^n)$. 
\end{theorem}

Thus, quantum Hamiltonian reduction provides a Lie-theoretic construction 
of the spherical rational Cherednik algebra for the symmetric group. 
A similar (but more complicated) Lie theoretic construction 
exists for symplectic reflection algebras for wreath product groups 
defined in Example \ref{wp} (see \cite{EGGO}). 

\subsection{Notes}
Our exposition in this section follows Section 4, Section 5 of \cite{E4}.

\end{document}